\documentclass[a4paper]{article}

\usepackage[letterpaper,left=1in,right=1in,top=0.8in,bottom=0.8in]{geometry}

\usepackage{amssymb}

\usepackage[all]{xy}
\usepackage{ucs}   % package to add unicode support
\usepackage[utf8x]{inputenc}  
\usepackage[english]{babel}
\usepackage{graphicx} %load extra symbols and environments

\usepackage[export]{adjustbox}
\usepackage{enumerate}
\usepackage{amsthm}
\usepackage{amsmath}
\usepackage{mathtools}
\usepackage{tikz-cd}
\usepackage{bbm}
\usepackage{enumitem}

\usepackage{verbatim}

\theoremstyle{plain}
\newtheorem{Th}{Theorem}[section]
\newtheorem{Lemma}[Th]{Lemma}
\newtheorem{Cor}[Th]{Corollary}
\newtheorem{Prop}[Th]{Proposition}

\theoremstyle{definition}
\newtheorem{Def}[Th]{Definition}

\newtheorem{Construction}[Th]{Construction}
\newtheorem{Rem}[Th]{Remark}
\newtheorem{Not}[Th]{Notation}
\newtheorem{?}[Th]{Problem}
\newtheorem{Ex}[Th]{Example}

\newtheorem{Ass}[Th]{Assumption}
\newtheorem{Proper}[Th]{Property}

\def\<{\langle}
\def\>{\rangle}

\newcommand{\Aa}{{\mathbb A}}

\newcommand{\Spec}{\operatorname{Spec}}
\newcommand{\Proj}{\operatorname{Proj}}
\newcommand{\Sym}{\operatorname{Sym}}

\newcommand{\End}{\operatorname{End}}
\newcommand{\GW}{\operatorname{GW}}
\newcommand{\SH}{\operatorname{SH}}
\newcommand{\one}{\mathbbm{1}}
\newcommand{\OO}{\mathcal{O}}
\newcommand{\Sp}{\operatorname{sp}}
\newcommand{\chic}{\operatorname{\chi_c}}
\renewcommand{\P}{{\mathbb{P}}}
\newcommand{\Tr}{\operatorname{Tr}}
\newcommand{\Dim}{\operatorname{dim}}

\newcommand{\Sm}{\operatorname{\mathbf{Sm}}}

\newcommand{\id}{\operatorname{id}}

\newcommand{\rk}{\operatorname{rk}}

\numberwithin{equation}{section}
\usepackage[colorlinks=false,plainpages,urlcolor=blue,linktocpage=true, hidelinks]{hyperref}

\begin{document}

\nocite{*}

\author{\textsc{Ran Azouri}}

\date{ }

\pagenumbering{arabic}

\title{\textsc{Motivic Euler Characteristic of Nearby Cycles and a Generalized Quadratic Conductor Formula}}

\maketitle

\begin{abstract}
	We compute the motivic Euler characteristic of Ayoub's nearby cycles spectrum in terms of strata of a semi-stable reduction, for a degeneration to multiple semi-quasi-homogeneous singularities. This allows us to compare the local picture at the singularities with the global conductor formula for hypersurfaces developed by Levine, Pepin Lehalleur and Srinivas, revealing that the formula is local in nature, thus extending it to the more general setting considered in this paper. The result is a quadratic refinement to the Milnor number formula with multiple singularities.
\end{abstract}
	
\setcounter{tocdepth}{1}
\tableofcontents

\section{Introduction}

\subsection{The Milnor fibre, nearby cycles and the Euler characteristic}

{Let $X$ be a complex manifold of dimension $n+1$ and $f: X \to D$} be a {non-constant} holomorphic function to the open unit disk $D$. {Suppose that $X_t:=f^{-1}(t)$ is smooth for $0<|t|<1$, and that $X_0$ has an isolated singularity $p$.} Take a small $\epsilon > 0$, and even smaller $t$, and consider $B_{p,\epsilon}$, the open ball with radius $\epsilon$. The Milnor fibre $M_{f,p}$ is defined by the intersection
$B_{p,\epsilon} \cap X_t$. $M_{f,p}$ is homotopically equivalent to a wedge of spheres, the number of which is defined to be the Milnor number, $\mu_{f,p}$ {\cite[Theorem 6.5]{Mil}}, an invariant of the singularity $p$; we may also consider the sheaf on $X_0$ defined by
$ x \mapsto H^*(M_{f,p}, \mathbb{Z}) $; as an alternative invariant we could compare the Euler characteristics of the fibres $\chi^{top}(X_t)$ and $\chi^{top}(X_0)$. Conductor formulas express the difference of Euler characteristics in the case of a proper map $f$ in terms of local invariants around the singular points of the special fibre:
\[ {\chi(X_t) - \chi(X_0) = \textit{invariants related to the singular points of } X_0.} \] 
{In the setting of complex geometry, this was investigated by Milnor. Suppose now that $f$ is a submersion outside a finite subset $\{p_1,\ldots, p_s\}$ of $X_0$. 
	At each singular point $p$, a choice of local coordinates $s_0,...,s_n$ for a neighbourhood of $p$ gives a description of the Milnor number $\mu_{f,p}$ by local terms, as the dimension of the Jacobian ring of $f$ at $p$, that is (\cite[Theorem 7.2]{Mil}),
	\[
	\mu_{f,p} = \Dim \OO_{X,p}/(\partial f / \partial s_0,..., \partial f / \partial s_n).
	\] 
	As an immediate consequence, we have the conductor formula with \emph{multiple} isolated singularities},
\begin{equation} \label{MilnorFormula}
	\chi^{top}(X_t)-\chi^{top}(X_0) = (-1)^{n} \sum_i   \mu_{f,p_i} .	
\end{equation}
In \cite{MO} Milnor and Orlik compute the value of $\mu_{f,p}$ explicitly  for the case of $f$ being a weighted homogeneous polynomial. 

These concepts around the Milnor fibre in the complex setting can all be developed in the world of algebraic geometry and \'etale cohomology. Let $f: X \to S$ be a flat family of schemes. We assume that we have a distinguished {closed} point $\sigma \in S$, with complement $\eta = S \setminus \sigma \hookrightarrow S$; so that the fibre of $f$ over $\eta$ is a smooth generic fibre $X_\eta $; and the fibre over $\sigma$ is the special fibre $X_\sigma$, which may be singular. The definition of the Milnor number in terms of the Jacobian ring carries naturally to this case, the Euler characteristics can be defined as well, using $l$-adic \'etale cohomology, and the cohomology of the Milnor fibre can be realised through the formalism of the nearby cycles functor 
 $$ {\Psi_f : \mathbf{D}_{cons}^b(X_\eta) \rightarrow \mathbf{D}_{cons}^b(X_\sigma)},$$
defined in {\cite[Expos\'e I, 2]{SGA7}}. {The Deligne-Milnor conjecture \cite[Expos\'e XVI, Conjecture 1.9]{SGA7II} is concerned with an algebraic version of Milnor's formula,  without restriction to characteristic zero.} Let $f: X \to S$ be a separated, finite type, flat morphism of relative dimension $n$, where $S$ is a henselian trait. Suppose that $X$ is {regular, that the general fibre $X_\eta$ is smooth over $\eta$,} and that $X_\sigma$ has a unique singular closed point $p$. Let $l$ be a prime number which is invertible on $\OO_S$. Then
\begin{equation} \label{DeligneFormula} \chi ^{l-adic} (X_ \eta) - \chi ^{l-adic} (X_\sigma) + \Dim Sw(\Phi^n(\mathbb{F}_\ell)_p) = (-1)^n \mu_{f,p} 
\end{equation}

 with the \textit{Swan conductor} $Sw(\Phi)$ being an additional term, adjusting for the case of positive characteristic. The formula is proven in the case of equal characteristics (\cite[Expos\'e XVI, Th\'eor\`eme 2.4]{SGA7II}),  and in the cases of relative dimension 1 and of $(X_\sigma)_{red}$  a simple normal crossing divisor (\cite{Blo}, \cite[Theorem 6.2.3]{KS}, and \cite[Th\'eor\`eme 0.8]{Or} for the statement with Milnor number as appearing here), most recently in full generality in \cite{BP}. The global difference of Euler characteristics is related to considering $\Dim \Phi^n(\mathbb{F}_\ell)_p$ at the singularity $p$, where $\Phi$ is the vanishing cycles functor.
 The local formula yields, as in the complex analytic case, a conductor formula for a flat proper map $f:X\to S$ as above, and allowing the special fibre to have \emph{multiple} isolated singularities. One may seek \emph{a quadratic refinement} for formulas \ref{MilnorFormula} and \ref{DeligneFormula}, that is, an identity of quadratic forms over a common base field, instead of an identity of integers. These quadratic refined invariants, which contain at once info about the schemes varying the base field, arise when considering motivic analogues for the relevant concepts in algebraic geometry.

\subsection{Motivic refinements}

In the context of motivic homotopy theory, the nearby functor cycles formalism has been developed by Ayoub in \cite{Ay07a}. Here the bounded derived category is replaced by the $\Aa^1$-homotopy category $\SH(-)$, constructing a functor
\[
  {\Psi_f : \SH(X_\eta) \rightarrow \SH(X_\sigma)}.
 \]
One may consider \emph{the motivic nearby cycles spectrum} $\Psi_f\one \in \SH(X_\sigma)$ as an object of study, as well as its restriction to different subschemes or points of $X_\sigma$, e.g if $p \in X_\sigma$, we may consider the invariant $(\Psi_f\one)|_p \in \SH(k(p))$ (the only non trivial part of $\Psi_f\one$ is at the singular locus). 
As a somewhat parallel concept, but from a motivic integration approach, Denef and Loeser \cite{DL}, \cite {DL00} constructed a motivic Milnor fibre defined by a class $[S_f]$ in the Grothendieck ring of varieties. It is expressed in terms of certain \'etale coverings for strata of the special fibre. Using rigid analytic motives Ayoub, Ivorra and Sebag \cite{AIS} show that the class of the motivic nearby cycles spectrum in $K_0(\SH(X_\sigma))$ is equal to the one computed by those covers, that is
$[\Psi_f \one] = [S_f]$.

Within the setting of stable $\Aa^1$-homotopy theory, we can refine the topological Euler characteristic as well to a motivic setting. The \emph{motivic, or quadratic, Euler characteristic} of a smooth and proper scheme is defined as the categorical trace of the identity morphism of the motive of the scheme in the category of motivic spectra $\SH(k)$. A variant definable over singular schemes is the compactly supported Euler characteristic. Working over a perfect field $k$, for every finite type $k$-scheme $X$ we get an element $\chic(X/k)$ in the Grothendieck-Witt group $\GW(k)$.
We may also consider $\chic(-)$ of any dualisable object in $\SH(k)$, such as $\Psi_f\one$ considered over $k$, and so we introduce the main invariant studied in this paper, \emph{the quadratic Euler characteristic of the nearby cycles spectrum}, \[ \chic(\Psi_f \one)  \in \GW(k) ; \]
we may also restrict the nearby cycles spectrum to a point $p$ in the special fibre, giving a \emph {motivic version for the Euler characteristic of the Milnor fibre}, $\chic(\Psi_f \one|_p)  \in \GW(k(p))$. This gives an invariant on quadratic forms for an isolated singularity on a scheme.

Let $\OO$ be a discrete valuation ring with residue field $k$,  fraction field $K$ and a fixed uniformizer $t\in \OO$. {Let $F(T_0,\ldots, T_{n})\in k[T_0,\ldots, T_n]$ be a homogeneous (or weighted-homogeneous) polynomial of degree $e$, defining a smooth projective (or weighted projective) hypersurface. The hypersurface $H^F$ defined in $\P^{n+1}_\OO$ by $F(T_0,\ldots, T_{n}) - tT_{n+1}^e$ thus gives a family of hypersurfaces that degenerates to the cone over the section defined by $F$. With this setup, Levine, Pepin Lehalleur and Srinivas \cite[Theorem 5.6]{Le20b} develop a quadratic conductor formula that takes the form}
\[ \Delta_t(F/k) := \Sp_t \chic (H^F_t) - \chic (H^F_0) =  \< e\> - \<1\> + (-\<e\>)^{n} \cdot \mu_{F,0}^q  \] in the homogeneous case. They also develop a similar formula for a weighted homogeneous $F$ \cite[Theorem 5.3]{Le20b}.
Since $\chic (H^F_t) \in \GW(K)$, $\chic (H^F_0) \in \GW(k)$ live in different rings, one has to use the specialization map $\Sp_t: \GW(K) \rightarrow \GW(k) $ to compare them; the term $\mu_{F,0}^q \in \GW(k)$ in the right hand side is a quadratic refinement of the Milnor number $\mu_{F,0} \in \mathbb{Z}$. It can be defined in algebraic terms by a certain quadratic form on the Jacobian ring $J(F,0)$, corresponding to a distinguished element in this ring defined by Scheja-Storch.  
 The main goal of this paper is to formulate and prove a generalization of this result, for a more general scheme, and with multiple singularities. Our first main result is  Theorem \ref{introthm}, a formula for the quadratic Euler characteristics of the motivic nearby cycles spectrum at a semi-quasi-homogeneous singular point $p$, $\chic(\Psi_f\one|_p)$, in terms of invariants of the defining polynomial $F$. This, combined with the conductor formula of \cite{Le20b} for projective hypersurfaces, and the functoriality of nearby cycles, provides a conductor formula for a scheme with several isolated singularities, Theorem \ref{introformula}.

\subsection{Outline and main results}

After this introduction, in section 2 we first review some basic facts about the quadratic Euler characteristic with compact supports $\chic(-)$, including its behaviour with respect to open-closed decomposition of a scheme.

In section 3 we proceed to discuss the invariant $\chic(\Psi_f \one)$ for a flat morphism $f:X \to S$. In the case of a special fibre which is supported on a normal crossing divisor, $X_\sigma = \sum a_i D_i$, where there are no triple intersections of the $D_i$, we show that a certain geometric construction gives a semistable reduction for $X$ from which we conclude (Proposition \ref{AIS}, with an assumption on the characteristic of the base field), that \[\chic(\Psi_f \one) = \chic ([S_f]) =  \sum_i \chic(\widetilde{D_{i}^\circ}) - \sum_{i < j}  \chic(\mathbb{G}_m \times \widetilde{D_{ij}}) ,\]
 where $ \widetilde{D_{i}^\circ} $, $\widetilde{D_{ij}}$ are certain étale covering of the strata of $X_\sigma$ defined in Section~\ref{DLDef}. This reproves a special case of the more general formula of Ayoub-Ivorra-Sebag mentioned above, \cite[Theorem 8.6]{AIS}. It is the same geometric construction that we present here, with which we proceed to the results in the rest of the paper. Using our method of proof we can get the same formula also in some cases in which the $D_i$ are not smooth, to treat the quasi-homogeneous case, see Remark \ref{qAIS} and Remark \ref{qRem}.

Next we proceed to computing our invariant at an isolated singular point $p$, that is we compute \[ \chic(\Psi_f \one |_p )  \in \GW(k(p)).\]  In section 4 we deal with \emph{the homogeneous case}.  Our setup is as follows:

\begin{Def} \label{lookslike} Let $f: X \rightarrow \Spec \OO $ be a flat quasi-projective morphism of schemes over a discrete valuation ring $\OO$ with quotient field $K$, residue field $k$ and uniformizer $t$, with $X$ being a regular scheme and with $X_\eta / K$ smooth. Let $p\in X_\sigma$ be an isolated singular point and let $F\in k(p)[T_0,\ldots, T_n]$ be a homogeneous polynomial of degree  $e$; let $\OO_{X,p}$ be the stalk at $p$, and $m_p \subset \OO_{X,p}$ the maximal ideal. We say that $X_\sigma$ {\em looks like the homogeneous singularity defined by $F$ at $p$} if there is a regular sequence of generators $s_0,\ldots,s_n $
for $m_p$ such that
\[
f^*(t)\equiv F(s_0,\ldots, s_n)\mod  m_p^{e+1}.
\]
\end{Def}

We then construct semi-stable reduction $Y$ for $X$ by a blow-up, followed by base change and normalisation. Using the key result by Ayoub that the functor $\Psi_{(-)}$ is computable on strata of semi-stable schemes, we obtain a formula for our invariant, the Euler characteristic of nearby cycles at the singular point.
\begin{Th}[Corollary~\ref{CorLocalHomogFormula}] \label{introthm}
	Let $f: X  \to \Spec \OO $ be as in definition~\ref{lookslike}, with $p \in X_\sigma$ an isolated singularity of the special fibre $X_\sigma$, on which  $f$ looks like the singularity defined by a homogeneous polynomial  $F\in k(p)[T_0,\ldots, T_n]$ of degree $e$,  with $V(F)\subset \P^n_{k(p)}$ a smooth hypersurface, and with $e$ prime to the exponential characteristic of $k$; assume $\Psi_f \one$ is dualisable (e.g. in characteristic $0$). Then
	\[
	\chic(\Psi_f\one|_p)= \chic(V_{\P^{n+1}}(F-T_{n+1}^{e}))-
	\<-1\> \chic(V_{\P^n}(F)).
	\]
\end{Th}

In Section 5 we treat the more general {\em quasi-homogeneous} case, where the defining polynomial $F$ at each singular point is a {\em weighted homogenous} polynomial with respect to a sequence of positive integer weights $a_*=(a_0,\ldots,a_n)$. The projective space $\mathbb{P}^n$ is replaced by the $a_*$-weighted projective space $\mathbb{P}(a_*)$ and its presentation as a finite group quotient of $\P^n$ is used to lift to the homogeneous case. For the precise definition of when $f$ looks like a quasi-homogeneous singularity at $p \in X_\sigma$ see Definition \ref{qdef}. For the precise assumption on the special fibre in this case, see Assumption~\ref{qassumption}. We then get the same result as that of Theorem~\ref{introthm} for this more general case, Corollary~\ref{CorLocalqHomogFormula}.

\begin{Rem} The notion of singularity discussed here includes the case of a quasi-homogeneous singularity, but allows for additional higher degree terms in the local expansion. It is closely related to the notion of \emph{semi-quasi-homogeneous singularity} appearing in the literature, see e.g. \cite[Definition 2.17]{GLS}.	\end{Rem}

The quadratic Milnor number $\mu_{f,p}^q$ is the same as the $\Aa^1$-local Euler class for $X$ at $p$, $e_{p}( \Omega _{X/k} , dt )$. This class is the same as the local Euler class for $H^{F}$, which also equals to the quadratic Milnor number $\mu_{F,0}^q$, defined purely in algebraic terms terms depending on $F$. This is dealt with in section 6, using an $\Aa^1$-homotopy invariance argument (Corollary~\ref{Eulerclass}).

We then have the components needed to deduce the main theorems in section 7. First we have the following formula (Theorem~\ref{mainthm}, for simplicity assuming $k(p)=k$)
One may think of the left hand side as enumerating vanishing cycles for $X$ around $p$, and the right hand side as doing the same for the hypersurface $H^F$. So this gives us a comparison between $X$ and $H^F$, and allows us to use the main result of \cite{Le20b} for $H^F$, in order to get a formula for the scheme $X$ at $p$. Using the formalism of Ayoub's functor, we can consequently extend it to a global formula on a scheme $X$ with several semi-quasi-homogeneous singularities. 
\begin{Th}[Generalized quadratic conductor formula for quasi-homogeneous singularities, Corollary~\ref{conductor}] \label{introformula}
	Let $f: X \rightarrow \Spec \OO$ be as in Definition~\ref{lookslike}, of relative dimension $n$ {with $f$ proper} and $k$ of characteristic $0$. Suppose that $X_\sigma$ satisfies Assumption~\ref{qassumption}, with singular points $\{p_1,...,p_s\}$. Let $e_i$ denote the weighted-homogeneous degree of the corresponding polynomial $F_i$. Then
	\[
	\Sp_t\chic(X_\eta)-\chic(X_\sigma) = \sum_i \Tr_{k(p_i)/k} \left( \< \prod_j a_j^{(i)} \cdot e_i\> - \<1\> + (-\<e_i\>)^{n} \cdot \mu_{f,p_i}^q \right) .\]
\end{Th} 

This settles Conjecture 5.7 in \cite{Le20b} for the case of characteristic zero and singularities resolved by a single blow-up with a smooth exceptional divisor (satisfying Assumption \ref{assumption} or \ref{qassumption}); in fact, our result handles cases not covered by Conjecture 5.7, as the types of singularities treated above are not necessarily homogeneous or weighted-homogeneous in the sense of {\it loc. cit.} This is a generalization of the formula in \cite[Theorem 5.6]{Le20b} even for the case of a single singularity, as it does not assume $X$ is the hypersurface $H^F$. An interesting aspect in the quadratic formula, is that besides generalizing the classical formulas over the complex and real numbers, the summands $\Tr_{k(p_i)/k} ( \< \prod_j a_j^{(i)} \cdot e_i\> - \<1\> )$ for each $p_i$ vanish in the classical cases and so make appearance only 'motivically'. For more on that last point see the discussion after Corollary~\ref{RefMilnor}.

In section 8,  we deduce a quadratic formula for curves on a surface, refining the Jung-Milnor formula for curves, Corollary~\ref{CurveFormula}; we also deduce an identity on the Witt ring, Corollary~\ref{WittIdentity}.

\section*{Acknowledgements}

This work was made in partial fulfilment of the author's PhD thesis. The author would like to thank his PhD supervisor Marc Levine for his continuous support and for his willingness to share his mathematical understanding throughout the work on this paper. The author would also like to thank Florian Ivorra and the anonymous referee for useful comments and corrections on earlier versions of this paper. The author was funded from the DFG through the SPP 1786 and from the ERC through the project QUADAG. The paper is part of a project that has received funding from the European Research Council (ERC) under the European Union’s Horizon 2020 research and innovation programme (grant agreement No 832833).\\
\includegraphics[scale=0.08]{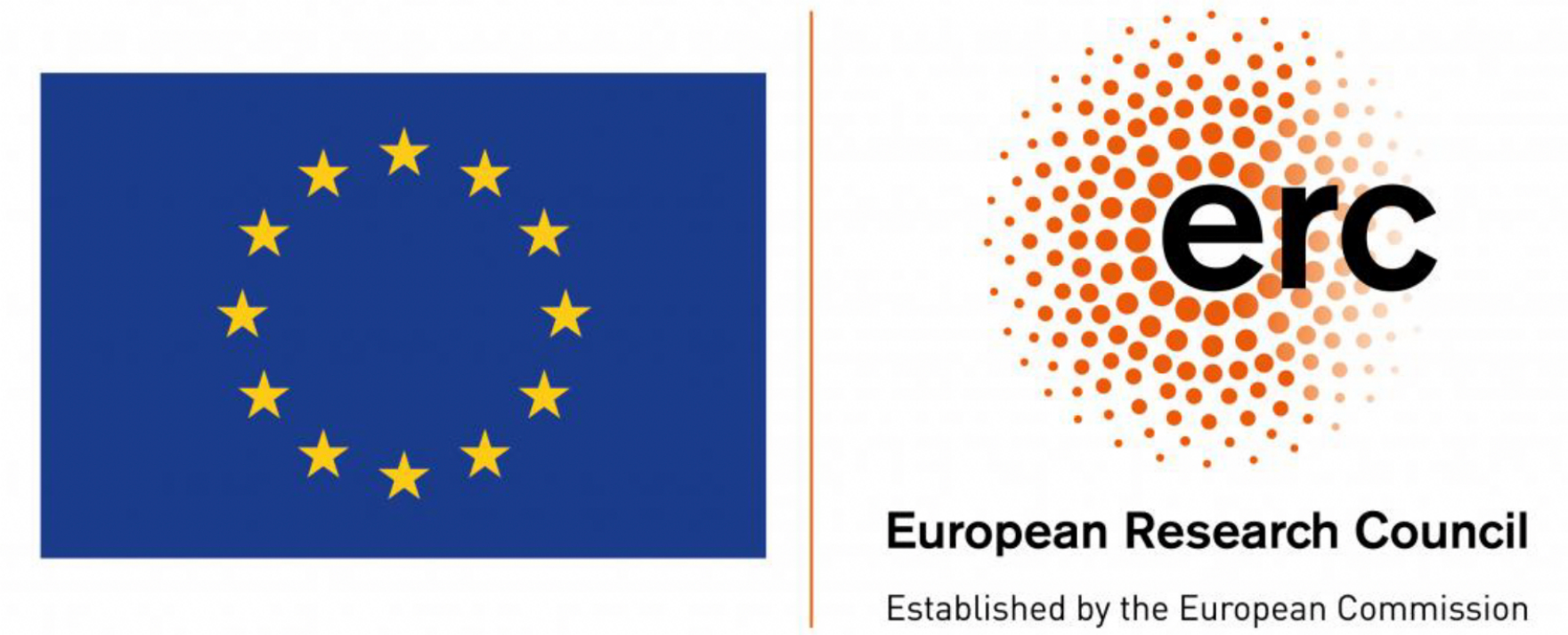}

\section{The motivic Euler characteristic with compact supports}\label{SectionMotivicEulerChar}

A construction of central importance in this paper is the {\em motivic Euler characteristic with compact supports}. For a finite type separated $k$-scheme $X$, $\chic(X/k)$ is an element in the Grothendieck-Witt ring $\GW(k)$ of $k$. Before going into a detailed discussion, we first give sketch of the main ideas that go into its construction. We use the notation and properties of the unstable and stable motivic homotopy categories to be found in \cite{Ay07a}, \cite{CD},  and \cite{Hoy}, including the six-functor formalism for $\SH(-)$.
 
Let $p: X \rightarrow \Spec k$ be a smooth and proper scheme over a  field $k$. {As we shall see below,} its motive with compact supports $p_! \one_X$ is a strongly dualisable object in the symmetric monoidal category $(\SH(F), \otimes)$, with dual $(p_! \one_X)^\vee = p_\# \one_X$. The Euler characteristic with compact supports of $X/k$ is   the trace of the identity endomorphism {for  $p_! \one_X \in \SH(k)$.
This yields an{ element in the ring $\End_{\SH(k)}(\one_k)$, which} is isomorphic to $\GW(k)$ via the Morel isomorphism {\cite[Lemma 6.3.8, Theorem 6.4.1]{Mo}. We} denote {the corresponding element of $\GW(F)$} by {$\chic(X/k)$; we omit $k$ when it is obvious from the context.}  For more details on the motivic Euler characteristic see \cite[\S 2]{Le20a}, and \cite[\S 1]{AMBOWZ} for a nice introduction on the compactly supported version. 

Here are some useful notations and definitions.

\begin{Not}
For a field $k$, we usually let $p$ denote the exponential characteristic of $k$, that is, $p$ is the characteristic of $k$ if this is positive, and is $1$ if the characteristic is zero. We will always assume that the characteristic is different than $2$. 

By $\GW(k)$ we denote the Grothendieck-Witt ring of $k$, i.e. the Grothendieck completion of the monoid of non-degenerate quadratic forms on $k$, with multiplication induced by tensor product of quadratic forms. For $a \in k^\times$, we denote by $\<a\> \in \GW(k)$ the class corresponding to the quadratic form $x \mapsto ax^2$
 
 \end{Not}

\begin{Def} Let $(\mathcal{C}, \otimes, \one_\mathcal{C} )$ be a symmetric monoidal {category, and take $x \in Ob(\mathcal{C})$.}  We say that $x$ is strongly dualisable if there exists an object $x^\vee \in Ob(\mathcal{C})$ and morphisms $ \delta_x: \one_\mathcal{C} \rightarrow x \otimes x^\vee$ and {$ev_x:  x^\vee \otimes x \rightarrow \one_\mathcal{C}  $, called }respectively co-evaluation and evaluation, such that
	\[ x \simeq \one_\mathcal{C} \otimes x \xrightarrow{\delta_x \otimes id} x \otimes x^\vee \otimes x \xrightarrow{id \otimes ev_x} x \otimes \one_\mathcal{C} \simeq x\]
	and
	\[ x^\vee \simeq x^\vee \otimes \one_\mathcal{C} \xrightarrow{id \otimes \delta_x }  x^\vee \otimes x \otimes x^\vee \xrightarrow{ev_x \otimes id} \one_\mathcal{C} \otimes x^\vee \simeq x^\vee \]
	are the identity morphisms. We call the object $x^\vee$ the dual of $x$.  \\
For $x$ a strongly dualisable object of $\mathcal{C}$ and $f:x\to x$ an endomorphism, the {\em categorical trace} of $f$ is the element $tr(f)\in \End_\mathcal{C}(\one_\mathcal{C})$ defined as the composition 
\[
\one_\mathcal{C}\xrightarrow{\delta_x}x \otimes x^\vee \xrightarrow{f\otimes id} x \otimes x^\vee \xrightarrow{\tau_{x, x^\vee}}x^\vee \otimes x\xrightarrow{ev_x}\one_\mathcal{C}.
\]
In particular, taking $f=id_x$, we have the {\em categorical Euler characteristic} $\chi_\mathcal{C}(x):=tr_x(id_x)$.
\end{Def}

\begin{Rem}\label{ChiMult}
It follows directly from the definitions that for $x, y$ strongly dualisable objects of $\mathcal{C}$, we have
\[
\chi_\mathcal{C}(x\otimes y)=\chi_\mathcal{C}(x)\otimes\chi_\mathcal{C}(y).
\]
\end{Rem}

\begin{Def}[\hbox{\cite[Definition 4.2.1]{CD}}] Define 
{$\SH_c(X)$, the subcategory of constructible objects in $\SH(X)$, as the thick triangulated subcategory generated by the objects $ \Sigma_{\P^1}^n f_{\#}  \one_Y$, where $f: Y \rightarrow X$ is a smooth $X$-scheme and $n \in \mathbb{Z}$.} An object in this category is called a constructible object.
\end{Def}

\begin{Prop} \label{constfunc}
	Constructible objects are stable under $f^*$ for any morphism $f$, under $f_{\#}$ for a smooth $f$, under $f^!$ for a proper $f$, and under $f_!$ for a separated $f$ of finite type (\cite[Proposition 4.2.4, 4.2.11, 4.2.12]{CD}).
	
	In addition, for $i: Z \hookrightarrow X$ a closed immersion and $j: U \hookrightarrow X$ its open complement, an object $\alpha \in \SH(X)$ is constructible if and only if $i^*\alpha$ and $j^* \alpha$ are constructible (\cite[Proposition 4.2.10]{CD}).
\end{Prop}

\begin{Prop} \label{constdual}
Take $\alpha \in \SH(k)[1/p] $, with  $k$ a field of exponential characteristic $p$. If $\alpha$ is constructible then it is strongly dualisable.
\end{Prop}

\begin{proof}
 By \cite[Theorem 3.2.1]{EK} for every smooth, separated, and finite type morphism $Y \rightarrow k$ in $\Sm_k$, $\Sigma_{\P^1}^n f_{\#}  \one_Y$ is strongly dualisable in $\SH(k)$ (the result over a perfect $k$ is due to Riou in \cite[Appendix B, Cor. B2]{LYZR}). Since the subcategory of strongly dualisable objects is itself thick (see e.g. \cite[Theorem A.2.5]{HPS}) we get that all constructible objects are strongly dualisable.
\end{proof}
As a consequence we can now make the following definition.

\begin{Def} \label{chicdef}
	Let $k$ be a field of exponential characteristic $p$, let $q: X \rightarrow \Spec k$ be a $k$-scheme and $\alpha \in \SH(X)[1/p]$ a constructible object.
	Then $\chic(\alpha /k)$ is defined to be the categorical Euler characteristic of  $q_! \alpha $ in $\SH(k)[1/p]$:
	\[ \chic(\alpha /k):=\chi_{\SH(k)[1/p]}(q_! \alpha).
	 \]
This is well defined as $q_! \alpha \in \SH(k)$ is constructible by Proposition~\ref{constfunc} and is strongly dualisable (in $\SH(k)[1/p]$) by Proposition~\ref{constdual}.
		For a $k$-scheme $ X$ we denote by $\chic(X/k)$ the object $\chic(\one_X/k)$
	 We write $\chic(\alpha)$ for $\chic(\alpha/k)$ when the base field $k$ is clear from the context.
\end{Def}

\begin{Rem} \label{ComplexChi}
	In the case $k = \mathbb{C}$, the rank homomorphism of quadratic forms gives an isomorphism, $rk: \GW(\mathbb{C}) \cong \mathbb{Z} $. We recover the topological Euler characteristic under this identification \[ \chic({X/\mathbb{C}}) = \chi_c^{top}(X(\mathbb{C})) = \chi^{top}(X(\mathbb{C})) . \]
	For the first equality see \cite[remark 1.5]{Le20a}. The second equality is true for every complex algebraic variety.
\end{Rem}

A useful property of the compactly supported motivic Euler characteristic is the \emph{cut-and-paste property}, which is formulated in the following proposition.

\begin{Prop} \label{cutandpaste}
	{Let $q:X\to \Spec k$ be a $k$-scheme.  Let $\alpha \in \SH(X)$ be a constructible object,}  and let \[ Z \xhookrightarrow{i} X \xhookleftarrow{j} U \] be a closed embedding and its open complement.
	Then \[ \chic(\alpha) = \chic(i^*\alpha) + \chic(j^* \alpha)  \]
and
	 \[ \chic(\alpha) = \chic(i^! \alpha ) + \chic (j_{*}j^* \alpha). \]
\end{Prop}

\begin{proof}
	The distinguished triangle of endofunctors on $\SH(X)$	\[ j_{!}j^! \rightarrow id_{\SH(X) } \rightarrow i_{*}i^* \rightarrow  \]
	gives  a {distinguished triangle}  {of endofunctors on} $\SH(k)$ after composing with $q_!$,
	\[ q_!j_{!}j^! \rightarrow q_! \rightarrow q_! i_{*}i^* \rightarrow  \]
		Applying each of these terms to $\alpha$ gives a constructible object in $\SH(k)$ by Proposition \ref{constfunc}, which is therefore strongly dualisable in $\SH(k)[1/p]$  (Proposition \ref{constdual}). It follows from the result on the additivity of traces of May (\cite[Theorem 0.1]{May}) that the Euler characteristic $\chic(-)$ is additive on distinguished triangles. So we can apply $\chic$ here to get \[ \chic(\alpha) = \chic(j_{!}j^! \alpha ) + \chic ( i_{*}i^* \alpha) .\] 
		Since $ i_* = i_! $, $j^* = j^! $, we have
		\[ \chic(\alpha) = \chic(j^* \alpha ) + \chic (i^* \alpha). \]
		Similarly, by using the {distinguished} triangle
			\[ i_{!}i^! \rightarrow id_{\SH (X) } \rightarrow j_{*}j^* \rightarrow  \] we get
			\[ \chic(\alpha) = \chic(i^! \alpha ) + \chic (j_{*}j^* \alpha). \]
		\end{proof}

\begin{Rem}
Let $k$ be a field and let $X$ be a  $k$-scheme, $Y \subset X$ {a} closed subscheme and {$U$ the open complement $X\setminus Y$}, then from Proposition~\ref{cutandpaste} applied to $\alpha = \one_X$ we get 
$ \chic(X) = \chic (Y) + \chic(U)$.
From this relation it follows that the motivic Euler characteristic factorises through the Grothendieck ring of $k$-varieties $K_0(Var_k)$. In other words we have the following commutative diagram
	\[ \begin{tikzcd}
	\text{Var}_k \arrow[rr, "\chic(\cdot / k)"] \arrow[rd, "{X \mapsto [X]}"'] & & \GW(k)   \\
	&  K_0(Var_k) \arrow[ru, dashrightarrow]			
\end{tikzcd} .\]
\end{Rem}
This yields in the standard way a Mayer-Vietoris property with respect to a Zariski open cover for $\chic(-)$.

\begin{Prop} \label{MV} Let $k$ be a field,
	let $X$ be a $k$-scheme with a Zariski open cover $X=U_1\cup U_2$ and let $\alpha \in \SH(X)$ be a constructible object. Let $U_{12}=U_1\cap U_2$ and let  $j_1:U_1\to X$, $j_2:U_2\to X$,  and $j_{12}:U_{12}\to X$ be the inclusions. Then
\[
\chic(\alpha)=\chic(j_1^*\alpha)+\chic(j_2^*\alpha)-\chic(j_{12}^*\alpha).
\]
\end{Prop}

\begin{proof} Let $Z=X\setminus U_1=U_2\setminus U_{12}$, with reduced scheme structure, and with closed immersions $i:Z\to X$, $i_2:Z\to U_2$. We have the canonical isomorphism $i_2^*j_2^*\alpha\cong i^*\alpha$, whence the identity
$
\chic(i^*\alpha)=\chic(i_2^*j_2^*\alpha)
$.
Put together with Proposition~\ref{cutandpaste} this gives the desired result.
\end{proof}

\begin{Prop}[Purity] \label{Purity} Let $i:Z\to X$ be a closed immersion of smooth $k$-schemes, or pure codimension $c$, let $f:X\to \Spec k$, $g:Z\to \Spec k$ be the structure morphisms. Then for $\alpha$ a constructible object of $\SH(k)$, we have
\[
\chic(i^! f^*\alpha )=\<-1\>^c\cdot \chic(g^* \alpha )
\]
In particular,
\[
\chic(i^!\one_X)=\<-1\>^c\chic(Z/k).
\]
\end{Prop}

\begin{proof} The special case  follows from the main statement by taking $\alpha=\one_k$. We use the notation from \cite{Hoy}. Let $f:Z\to \Spec k$, $g:Z\to \Spec k$ be the structure morphisms, let $\Omega_f$, $\Omega_g$ be the respective sheaves of relative differentials, and let $\mathcal{N}_i$ be the conormal sheaf of $i$. We have the purity isomorphism (see \cite[Appendix A]{Hoy14})
\[
i^!\circ f^*\cong \Sigma^{-N_i}\circ  g^*.
\]
 Using the Mayer-Vietoris property Proposition~\ref{MV} for $\chic(-)$, we reduce to  the case of trivial conormal sheaf, $\mathcal{N}_i\cong\OO_Z^c$, inducing the natural isomorphism $\Sigma^{-\mathcal{N}_i}\cong \Sigma_{\P^1}^{-c}$.
We have the projection formula \cite[Theorem 6.18(7)]{Hoy}
\[
g_!(\Sigma^{-c}_{\P^1}\beta)\cong \Sigma^{-c}_{\P^1}g_!(\beta)
\]
for $\beta\in \SH(Z)$. Since $\Sigma^{-c}_{\P^1}\gamma\cong S^{-2c,-c}\otimes\gamma$ for 
$\gamma\in \SH(k)$, it follows from Remark~\ref{ChiMult} and \cite[Lemma 2.2]{Le20a} that
\[
\chi_{\SH(k)}(\Sigma^{-c}_{\P^1}\gamma)= \<-1\>^{-c}\cdot \chi_{\SH(k)}(\gamma)=\<-1\>^{c}\cdot \chi_{\SH(k)}(\gamma)
\]
for $\gamma\in \SH(k)$ strongly dualisable. Thus
\[
\chic(i^!f^*\alpha)=\chi_{\SH(k)}(\Sigma^{-c}_{\P^1}g_!(g^*\alpha))= 
\<-1\>^c\cdot \chic(g^*\alpha).
\]
\end{proof}

\begin{Rem}[Non-perfect fields] Let $F$ be a field of characteristic $p>2$,  and with perfect closure $F^{perf}\supset F$. Thanks to \cite[Theorem 3.2.1]{EK} we were able to define $\chic(-/F)$ in Definition~\ref{chicdef} over a non-perfect field and proceed with this definition to prove the properties. However, the base-extension $\GW(F)[1/p]\to \GW(F^{perf})[1/p]$ is an isomorphism, so we can compare $\chic(-/F)$ with $\chic(-/F^{perf})$ and use the latter to define the former. For a constructible object $\beta\in \SH(F)$, the base-extension $\beta^{perf}\in \SH(F^{perf})$ is constructible. Moreover, for an $F$-scheme $q:X\to \Spec F$ and an element $\alpha\in \SH(X)$, we have the base-change $q^{perf}:X\times_{\Spec F}\Spec F^{perf}\to\Spec F^{perf}$ and  $\alpha^{perf}\in  \SH(X)$, with $q^{perf}_!(\alpha^{perf})$ canonically isomorphic to the base-change $q_!(\alpha)^{perf}$ of $q_!(\alpha)$. Thus, we may define $\chic(\alpha/F)$ also by
\[
\chic(\alpha/F):=\chic(\alpha^{perf}/F^{perf})\in \GW(F^{perf})[1/p]=\GW(F)[1/p].
\]
This agrees with Definition~\ref{chicdef} through the isomorphism $\SH(F)[1/p] \simeq \SH(F^{perf})[1/p]$ of \cite[Theorem 2.1.1]{EK}; see also \cite[Section 5.1, p. 45]{Le20b} and \cite[Section 2, p. 2185]{Le20a} for similar discussions and for the passage from $\GW(F)[1/p]$ to $\GW(F)$.
\end{Rem}

{Another useful formula concerns change of base field. For $k_1\subset k_2$ a finite separable field extension, we have the transfer map on the Grothendieck-Witt rings
\[
\Tr_{k_2/k_1}:\GW(k_2)\to \GW(k_1).
\]
This is the so-called {\em Scharlau transfer}\footnote{The Scharlau transfer for the Witt groups is discussed, for example, in \cite[Chapter 2, Section 5]{Scharlau}; the same construction works for the Grothendieck-Witt groups.}  with respect to the trace map $\Tr_{k_2/k_1}:k_2\to k_1$ and is defined as follows.  For a finite-dimensional $k_2$-vector space $V$ and a non-degenerate symmetric $k_2$-bilinear map $b:V\times V\to k_2$, one considers $V$ as a (finite-dimensional) $k_1$-vector space, giving the
symmetric $k_1$-bilinear map  $\Tr_{k_2/k_1}\circ b:V\times V\to k_1$; the fact that $k_2$ is separable over $k_1$ implies that $\Tr_{k_2/k_1}$ is surjective and hence $\Tr_{k_2/k_1}\circ b$ is non-degenerate. Sending $b$ to $\Tr_{k_2/k_1}\circ b$ defines the map $\Tr_{k_2/k_1}:\GW(k_2)\to \GW(k_1)$. 
}

\begin{Prop} \label{Transfer} Let $k_1\subset k_2$ be a finite separable extension of fields, let $\pi:\Spec k_2\to \Spec k_1$ be the induced morphism, and let $f:X\to \Spec k_2$ be a $k_2$-scheme, which we consider as $k_1$-scheme via composition with $\pi$.  For a constructible object $\alpha\in \SH(X)$ we have
\[
\chic(\alpha/k_1)=\Tr_{k_2/k_1}(\chic(\alpha/k_2))\in \GW(k_1).
\]
\end{Prop}

\begin{proof} This is \cite[Proposition 5.2]{Hoy14} combined with the canonical isomorphism $(\pi\circ f)_!\cong \pi_!\circ f_!$.
\end{proof}

\section{Motivic nearby cycles and semi-stable reduction}

\subsection{Ayoub's motivic nearby cycles functor}\label{SectionNearbyCycles}
Throughout the paper we fix a discrete valuation ring $\OO$ with residue field $k$,  fraction field $K$ and fixed uniformizer $t\in \OO$;  $\sigma$ denotes the closed point $\Spec k$ and  $\eta$ the generic point $\Spec K$. We define $B$ to be $\Spec \OO$. We will assume in addition that $\OO$ contains a subfield $k_0$ such that $B$ is smooth and essentially of finite type over $k_0$, and  the field extension $k_0\to k$ is finite and separable.

{Let $f:X\to B$ be a flat, quasi-projective $B$-scheme. We have the open-closed embedding  $ \sigma \xhookrightarrow{i} B \xhookleftarrow{j} \eta $, with the closed immersion $i$ and the open immersion $j$. Denote the respective pullbacks {by} $X_{\sigma}$, $X_{\eta}$ ('the special and the generic fibre') and {denote} the maps {induced by $f$} according to the following diagram
\[
\begin{tikzcd}[]
X_{\sigma} \arrow[r, hook] \arrow[d, "f_\sigma"] & X  \arrow[d, "f"] & X_{\eta} \arrow[l, hook'] \arrow[d, "f_{\eta}"] \\
\sigma \arrow[r, hook, "i"] & B  & \eta \arrow[l, hook', "j" ']
\end{tikzcd} . \]
For the construction of the motivic nearby cycles functor 
\[ \Psi_f : \SH(X_{\eta}) \rightarrow \SH(X_{\sigma})\]
see \cite[3.2.1]{Ay07a}.  {Fixing the parameter} $t$ defines a map $t: \Spec \OO \rightarrow \Spec k_0[t]$. By abuse of notation we use $\Psi_f$ also to denote $\Psi_{t \circ f}$, with the base being $\Aa_{k_0}^1$. We will use some of the compatibility properties satisfied by $\Psi_{(-)}$, among which is the following.

\begin{Proper}[see \hbox{\cite[Definition 3.1.1]{Ay07a}}] \label{property}

For {each} morphism $ g: Y \rightarrow X $, {of flat quasi-projective $B$-schemes,} there are well-defined natural transformations \[ \alpha_g : g_{\sigma}^* \circ \Psi_f \rightarrow \Psi_{f\circ g} \circ g_{\eta}^* \]
and
\[ \beta_g : \Psi_f \circ g_{\eta *} \rightarrow g_{\sigma *} \circ \Psi_{f \circ g} \]
 such that:
\begin{enumerate}[label=(\alph*)]
	\item
  If $g$ is smooth $\alpha_g$ is {natural isomorphism}.
  \item
  if $g$ is projective then $\beta_g$ is an {natural isomorphism}.
 \end{enumerate}
\end{Proper}
 These natural transformations satisfy some compatibility conditions, for details check \cite[3.1.1, 3.1.2]{Ay07a}
{The next result is a very useful tool for computing $\Psi_f$.} {\begin{Not} Let $X$ be a smooth $k_0$-scheme, $D$ a simple normal crossing divisor on $X$ with 
irreducible components $D_1,\ldots, D_r$. For $I\subset \{1,\ldots, r\}$, let $D_I:=\cap_{i\in I}D_i$, $D_I^\circ:=\cap_{i\in I}D_i\setminus\cup_{j\not\in I}D_j$,  $D_{(I)}:=\cup_{i\in I}D_i$, and  $D_{(I)}^\circ:=D_{(I)}\setminus\cup_{j\not\in I}D_j$.  
\end{Not}}

\begin{Prop}[\hbox{\cite[Th\'eor\`eme 3.3.44]{Ay07a}}] \label{computability}
 Let $f : X \rightarrow B$ be a flat quasi-projective $B$-scheme. Suppose that $X$ is smooth over $k_0$ and that  $X_\sigma := f^{-1}(0)$ is a simple normal crossing divisor (in particular, reduced) with irreducible components $D_1, \ldots, D_r$. Fix a non-empty subset $I\subset \{1,\ldots, r\}$,   let $D_{(I)}^\circ \xhookrightarrow{v} D_{(I)} \xhookrightarrow{u} X_\sigma$ denote the respective open and closed immersions. \\
	Then composing $u^* \Psi_f f_\eta ^*$ with the unit map  $id \rightarrow v_* v^* $ of the adjunction, induces {a natural isomorphism}
	\[ u^* \Psi_f f_\eta ^* \simeq v_* v^* u^* \Psi_f f_\eta ^* .  \] 
\end{Prop}
For the rest of the section we fix $I$ and let $D:=D_{(I)}$, $D^\circ:=D_{(I)}^\circ$. {For $i:Z\to Y$ the inclusion of a locally closed subscheme, and $\alpha\in \SH(Y)$, we sometimes write $\alpha|_Z$ for $i^*(\alpha)\in \SH(Z)$.}
\begin{Rem}\label{Rem:computability} {We retain the notation from Proposition~\ref{computability}.
Evaluating at $\one_\eta \in \SH (\eta) $ and formulating the statement slightly differently, we have
\[ (\Psi_f(\one_{X_{\eta}}))|_D = v_* (\Psi_f(\one_{X_{\eta}})|_{D^\circ}). \]
Here $(\Psi_f(\one_{X_{\eta}}))|_D$ denotes the pullback  $u^*\Psi_f(\one_{X_{\eta}})\in \SH(D)$ via the inclusion $u:D\to X_\sigma$, and similarly $\Psi_f(\one_{X_{\eta}})|_{D^\circ}:=v^*u^* \Psi_f(\one_{X_{\eta}})\in \SH(D^\circ)$.} \\
Moreover, {taking $I=\{i\}$, we get}
	\[   
(\Psi_f(\one_{X_{\eta}}))|_{D^\circ} = w^* \Psi_{id}(\one_B)=w^* (\one_\sigma)=\one_{D^\circ} \] where $w:D^\circ\to \sigma$ is the structure morphism. This last statement follows from the compatibility of $\Psi_{(-)}$ with smooth pullback, Property~\ref{property}, applied to the open immersion $X\setminus\cup_{j\neq i}D_j\hookrightarrow X$ and then to the smooth morphism 
$X\setminus\cup_{j\neq i}D_j\to B$. In addition, the identity $\Psi_{id}(\one_B)=\one_\sigma$ follows from \cite[Proposition 3.4.9, Lemma 3.5.10]{Ay07a}.
	\end{Rem}

\begin{Rem}
The statement of the theorem appears in \cite[Th\'eor\`eme 3.3.10, Remarque 3.3.12]{Ay07a} for the case 
$
X= B[T_1,\ldots,T_k]/(T_1\cdot\ldots\cdot T_k - t)
$, and $f$ the obvious morphism to $B$. In \cite[Th\'eor\`eme 3.3.44]{Ay07a} the statement is essentially the same as in our Proposition~\ref{computability}, with the assumption $I=\{i\}$. This special case is in fact all we need to use later on.
 \end{Rem}

\subsection{The Euler characteristic of nearby cycles}\label{SectionNearbyCyclesEulerChar}

Retain the notation of $\OO$ and $B=\Spec\OO$ as  in Section~\ref{SectionNearbyCycles}. Let $f:X\to B$ be a flat quasi-projective morphism with  $X$ smooth over $k_0$  and $X_\eta$  smooth over $\eta$. We make here some first computations {of} $\chic(\Psi_f(\one_{X_\sigma}))$. In what follows we will assume that $k_0$ is of characteristic $0$ in order to have the result of the proposition below; alternatively we may assume the result, namely that $ f_{\sigma !}\Psi_f\one $ is a strongly dualisable object in $\SH(\sigma)$.

\begin{Prop} Assume that the base field $k_0$ is of characteristic $0$. Then:
\item[ (1)]	$ f_{\sigma !}\Psi_f(\one_{X_{\eta}}) $ is a strongly dualisable object in $\SH(k)$.
\item[ (2)] 	 $ \chic(\Psi_f(\one_{X_{\eta}})) \in \GW(k) $ is well-defined.
\end{Prop}

\begin{proof}
	For the first assertion, $\Psi_f$ sends constructible objects to constructible objects \cite[Th\'eor\`eme 3.5.14]{Ay07a} and constructibles are stable under the exceptional pushforward functor $(-)_{!}$ \cite[Corollaire 2.2.20]{Ay07a}, hence
	$ f_{\sigma !}\Psi_f(\one_{X_{\eta}}) $ is constructible and therefore strongly dualisable (Proposition \ref{constdual}). \\
		(2) follows from (1) and Definition \ref{chicdef}.
\end{proof}

By formal consequence of the properties of $\Psi_f$, the only non-trivial part of $\Psi_f \one$ is at the singular locus, and as the following proposition shows one can compute $\chic(\Psi_f \one)$ by just investigating $\Psi_f$ around isolated singularities.

\begin{Prop} \label{local} \cite[Proposition 8.3]{Le20b} 
	Assume $P=\{p_1,\ldots,p_s\}$ is the (finite) set of singular points in $X_\sigma $. Then
	\[ \chic(\Psi_f({\one_X}_\eta)) = \sum_i \chic(\Psi_f({\one_X}_\eta)|_{p_i}) + \chic(X_\sigma \setminus P)  \]
\end{Prop}

\begin{proof}
	 Denote by $j: X \setminus P  \hookrightarrow X $, then by Property \ref{property},
	\[\Psi_f ({\one_X}_{\eta})|_{X_\sigma \setminus P} \simeq \Psi_{f\circ j} ( j_\eta ^* \one_X )=\Psi_{f\circ j} ( ({\one_{X \setminus P})}_\eta ) = \one_{X_\sigma \setminus P } \]
	the last equality being since $X \setminus P $ is smooth (e.g. by Proposition \ref{computability}).	Then by cut-and paste (Proposition \ref{cutandpaste})
	\[ \chic(\Psi_f({\one_X}_\eta)) = \sum_i \chic(\Psi_f({\one_X}_\eta)|_{p_i}) + \chic(\Psi_f ({\one_X}_{\eta})|_{X_\sigma \setminus P})  \]
		and we get the desired result.
\end{proof}

The following example illustrates how we can use Proposition \ref{computability} to compute $\chic(\Psi_f \one)$ on a simple normal crossing divisor {stratum by stratum}.

\begin{Ex} \label{example}
	Suppose $X_\sigma$ is a simple normal crossing divisor on $X$ that can be written as $X_\sigma = D_1 + D_2$ with $D_1$ and $D_2$ smooth over $\sigma$ and with transverse intersection $D_{12}:=D_1\cap D_2$.  Let $D_i^\circ:=D_i\setminus D_{12}$, $i=1,2$.	We have the closed-open complements 
	\[ D_1 \xhookrightarrow{u_1} X_\sigma \xhookleftarrow{j} D_2^\circ . \]
Then by Proposition \ref{cutandpaste}
	\[ \chic(\Psi_f(\one_{X_{\eta}})) = \chic(\Psi_f(\one_{X_{\eta}})|_{D_1}) + \chic(\Psi_f(\one_{X_{\eta}})|_{D_2^\circ}).  \]
Using Proposition \ref{computability}
	\[\chic(\Psi_f(\one_{X_{\eta}})) = \chic (v_{1*} \one_{D_{1}^\circ}) + \chic(\one _{D_{2}^\circ}), \]
applying both equations of Proposition \ref{cutandpaste} to { $\one_{D_1}$ and} the close-open complements
	\[ D_{12} \xhookrightarrow{i} D_1 \xhookleftarrow{v_1} D_1^0 \]
	 gives
	\[ \chic(\one_{D_{1}}) = \chic(i^*\one_{D_{1}}) + \chic (v_1^* \one_{D_{1}}) = \chic(\one_{D_{12}}) + \chic(\one_{D_{1}^\circ}) \]
	and
	\[ \chic (v_{1*} \one_{D_{1}^\circ}) = \chic (v_{1*} v_1^* \one_{D_{1}}) = 
	\chic(\one_{D_1}) -\chic(i^!\one_{D_{1}}).\]
Applying Proposition~\ref{Purity}, we have
\[ \chic (v_{1*} \one_{D_{1}^\circ}) =
\chic(\one_{D_1}) - \<-1\> \chic(\one_{D_{12}}). 
\]
	Combining the equations we get the formula
	\[ \chic(\Psi_f(\one_{X_{\eta}})) = \chic(\one_{D_{12}}) + \chic(\one_{D_{1}^\circ}) -  \<-1\> \chic( \one_{D_{12}}) + \chic(\one _{D_{2}^\circ}). \]
	We obtain the nice formulas
		\[ \chic(\Psi_f(\one_{X_{\eta}})|_{D_1}) = \chic(\one_{D_{1}}) -  \<-1\> \chic( \one_{D_{12}}) , \]
	and
	\[ \chic(\Psi_f(\one_{X_{\eta}})) = \chic(D_{1}^\circ)   + \chic(D_{2}^\circ) - (\<-1\>-\<1\>) \cdot \chic(D_{12}). \]
\end{Ex}
This exhibits how Proposition \ref{computability} enables us to compute the Euler characteristic of the nearby cycles functor of the unit when the special fibre is a simple normal crossing divisor. We would like to be able to reduce the general case to that case, also when the special fibre is \emph{not reduced}.

\subsection{Semi-stable reduction}

Let $f:X \rightarrow B$ be as in Section~\ref{SectionNearbyCyclesEulerChar} a flat quasi-projective morphism with $X$ smooth over $k_0$ and $X_\eta$ smooth over $\eta$, with $B=\Spec\OO$ as in Section~\ref{SectionNearbyCycles}.
 Let $\OO_e:=\OO[s]/(s^e-t)$, $B_e:=\Spec\OO_e$  and $b_e: B_e \rightarrow B$ the projection. Let $X_e:=X\times_B B_e$. Note that $ \sigma_e = \sigma $ as the residue field does not change by adding a root, but $\eta_e \rightarrow \eta$ may not be trivial. 

\begin{Def}
A semi-stable reduction datum for $f$ consists of a natural number $e$ and a projective birational map $p_e: Y \rightarrow X_{e} $, such that {$Y$ is smooth over $k_0$,} $Y_\sigma$ is a (\textbf{reduced}) simple normal crossing divisor and  $p_{e\eta}: Y_{\eta} \rightarrow X_{\eta_e} $ is an isomorphism. In addition, we will require that the cover $B_e\to B$ is {\em tame}, that is, that $e$ is prime to the exponential characteristic of $k$. 
\end{Def}
A theorem by Kempf, Knudsen, Mumford, and Saint-Donat \cite{KKMSD} asserts that over a field of characteristic $0$, and base $B$ a smooth curve, every variety $X$ admits a semi-stable reduction.

\begin{Prop}\label{prop:SStableRed}
	Assume $f:X \rightarrow B$ admits a semi-stable reduction $Y \xrightarrow{p_e} X_e \xrightarrow{f_e} B_e$ for some $e$. {Let $\pi:X_e\to X$ be the projection, and let  $f_Y=f_e \circ p_e $.} Then \[ \Psi_f(\one_{X_\eta}) \simeq (\pi\circ p_e)_{\sigma*}\circ \Psi_{f_Y} (\one_{Y_\eta})   \]
\end{Prop}

\begin{proof}   	By \cite[Proposition 3.5.9]{Ay07a}, we have the natural isomorphism $\Psi_f \simeq \pi_{\sigma*}\circ\Psi_{f_e} \circ {\pi}_\eta^* $. Since $p_{e\eta}$ is an isomorphism, the natural map $\id_{\SH(X_{e\eta})}\to p_{e\eta*}\circ p_{e\eta}^*$ is an isomorphism. This together with the pushforward property of $\Psi$ for projective maps, Property \ref{property}(b), gives the sequence of isomorphisms
\[
\Psi_f(\one_{X_\eta}) \simeq p_{\sigma*}\circ\Psi_{f_e} (\one_{X_{e\eta}}) \simeq
	 \pi_{\sigma*}\circ\Psi_{f_e} \circ p_{e\eta*}\circ \pi_{e\eta}^*(\one_{X_{e\eta}})
	 \simeq
	  \pi_{\sigma*}\circ p_{e\sigma*}\circ \Psi_{f_Y}  (\one_{Y_\eta}) \simeq   (\pi \circ p_e)_{\sigma*}\circ \Psi_{f_Y} (\one_{Y_\eta}). 
\]
	   
\end{proof}

As a consequence we can compute $\chic(\Psi_f)$ on a semi-stable reduction.

\begin{Cor} \label{semistablered}
 With the above notation	\[\chic(\Psi_f(\one_{X_\eta})) = \chic(\Psi_{f_Y} (\one_{f_Y})) .\]
 In addition if $D \subset X_\sigma$ is a closed subscheme and $E : = (\pi \circ p_e)^{-1} (D) \subset Y_\sigma$, then  	$\chic(\Psi_f(\one_{X_\eta})|_D) = \chic(\Psi_{f_Y} (\one_{f_Y})|_E) $,
\end{Cor}

\begin{proof}
	$(\pi \circ p_e)_{\sigma}$ is proper, so $ (\pi \circ p_e)_{\sigma *}=(\pi \circ p_e)_{\sigma !}$. Since $\sigma_e=\sigma$, we thus have
\begin{align*}
\chic(\Psi_f(\one_{X_\eta}))=\chi_{\SH(k)}(f_{\sigma!}\circ\Psi_f(\one_{X_\eta}))=&
\chi_{\SH(k)}(f_{\sigma!}\circ(p\circ p_e)_{\sigma!}\circ \Psi_{f_Y} (\one_{Y_\eta}))\\=&
\chi_{\SH(k)}(f_{Y\sigma!}\circ \Psi_{f_Y} (\one_{Y_\eta}))=
\chic(\Psi_{f_Y} (\one_{f_Y})).
 \end{align*}
 The second assertion follows by the same argument replacing $\Psi_f(\one_{X_\eta})$ by $\Psi_f(\one_{X_\eta})|_D$, and using proper base change.
\end{proof}	

\begin{Prop} \label{example2}
	Let $f : X \to B$ be a morphism as in \ref{SectionNearbyCycles}. Assume that that $X_\sigma = D_1\cup D_2$ is the decomposition of the special fibre to irreducibles, with $(X_\sigma)_{red}$ not necessarily a normal crossing divisor, and that we have a semi-stable reduction $Y \to \Spec \OO_e$ for some $e$, with $Y$ smooth over $k_0$, and $Y_\sigma= \widetilde{D_1} + \widetilde{D_2}$ a normal crossing divisor, and $\widetilde{D_1}$, $\widetilde{D_2}$ being the preimages of $D_1$, $D_2$ under the construction, respectively; let $\widetilde{D_{12}} := \widetilde{D_{1}}\cap \widetilde{D_{2}}$. Then we have 
\begin{enumerate}
	\item
		$  \chic(\Psi_f(\one_{X_\eta})) = \chic(\widetilde{D_{1}^\circ})   + \chic(\widetilde{D_{2}^\circ}) -  \chic(\mathbb{G}_m \times \widetilde{D_{12}}). $
	\item
		$  \chic(\Psi_f(\one_{X_\eta})|_{D_1}) = \chic(\widetilde{D_{1}}) - \chic(\mathbb{A}^1 \times \widetilde{D_{12}}). $
\end{enumerate}
\end{Prop}
\begin{proof} 
	By combining Example \ref{example} and Corollary \ref{semistablered} we get
	\[ \chic(\Psi_f(\one_{X_\eta})) = \chic(\Psi_{f_Y} (\one_{f_Y})) =  \chic(\widetilde{D_{1}^\circ})   + \chic(\widetilde{D_{2}^\circ}) - (\<-1\>-\<1\>) \cdot \chic(\widetilde{D_{12}}) \]
	and
	\[ \chic(\Psi_f(\one_{X_\eta})|_{D_1}) = \chic(\Psi_{f_Y} (\one_{f_Y})|_{\widetilde{D_1}}) =  \chic(\widetilde{D_{1}})   + \chic(\widetilde{D_{2}^\circ}) - \<-1\>\cdot \chic(\widetilde{D_{12}}) . \]
	Since $\chic(\Aa^1)=\<-1\>$, $\chic(\mathbb{G}_m)= \chic(\mathbb{A}^1)- \chic(pt) =\<-1\>-\<1\>$, the first formulas can be rewritten as
	\[  \chic(\Psi_f(\one_{X_\eta})) =  \chic(\widetilde{D_{1}^\circ})   + \chic(\widetilde{D_{2}^\circ}) -  \chic(\mathbb{G}_m \times \widetilde{D_{12}}) ,\]
	and
	\[
	\chic(\Psi_f(\one_{X_\eta})|_{D_1}) = \chic(\widetilde{D_{1}}) - \chic(\mathbb{A}^1 \times \widetilde{D_{12}}).
	 \]
\end{proof}
	
Proposition~\ref{example2} can be extended to a special fibre that has more than two components, with no triple intersections. In the next section we describe how to construct a semi-stable reduction in such case.

\subsection{Expressing $\chic (\Psi_f\one)$ by coverings of the strata}

	In the course of their work on motivic integration and motivic Zeta functions, Denef and Loeser define a motivic Milnor fibre {of morphism $f:X\to \Aa^1$} \cite[3.3]{DL00}, \cite[4]{DL} as an element in the Grothendieck ring of varieties, defined by certain coverings of {the strata} of the special fibre {of a resolution of $f$}. Ayoub, Ivorra and Sebag proved that the class defined by Ayoub's functor in this ring can be computed as an alternating sum involving these coverings \cite[Thm. 8.6]{AIS}; their proof relies on the use of motivic stable homotopy category for rigid analytic spaces. We treat here a simple case in which semi-stable reduction can be achieved by a simple construction, and then the formula can be proven by purely geometric means, relying on the properties of the nearby cycles functor mentioned in the previous subsections. 
	
We recall the construction of the covering maps following the description in \cite[3.1]{IS}, that we call here \emph{the Denef-Loeser covers}: Let $\sigma\hookrightarrow B=\Spec\OO \hookleftarrow \eta$ be as in Section~\ref{SectionNearbyCycles}.} Let $f:X\to B$ be a flat quasi-projective morphism with $X$ smooth over $k_0$ and $X_\eta$ smooth over $\eta$, and suppose $(X_\sigma)_{red}$ is a simple normal crossing divisor. We write $X_\sigma = a_1 D_1 + \ldots +a_r D_r $ with $D_1,\ldots,D_r$ the reduced irreducible components and assume that if $char k =p > 1$, then $p \nmid a_i$ for each $i$. Let $I$ be a non-empty subset of $ \{1,\ldots,r\}$, giving the closed stratum $u_I:D_I\to X_\sigma$ and open substratum $v_I:D_I^\circ\to D_I$. 
$f$ may be described on some affine open neighbourhood $U$ of some point of $D_I$ as 
\[f=u \cdot \prod_{i\in I} t_i^{a_i} 
\] 
with $t_i \in \OO_X(U)$, $u \in \OO_X(U)^\times$, and $D_i$ being $V(t_i)$ in $U$. Let $N_I=gcd_{i\in I}(a_i)$. We have the finite \'etale cover
 \[ \widetilde{D_{I,U}}:= \Spec(\OO_{ D_I^{\circ} \cap U} [T] / (T^{N_I} -u)) \rightarrow D_I^\circ \cap U. \]
The finite morphism 
$\widetilde{D_I} \rightarrow D_I $ is defined as the normalisation of $D_I$ in $\widetilde{D_{I, U}}$, and $\widetilde{D^\circ_I}\subset D_I$ is defined to be the open subscheme $\widetilde{D_I}\times_{D_i}D^\circ_I$ of $\widetilde{D_I}$. One shows that this construction is independent of the choice of $U$ and that $\widetilde{D^\circ_I}\to D^\circ_I$ is \'etale. We call the coverings $\widetilde{D_I} \rightarrow D_I $, $\widetilde{D_I^\circ} \rightarrow D_I ^\circ$ the Denef-Loeser coverings of  {$D_I$, $D_I^\circ$, respectively}. {These} coverings are well-defined up to isomorphism and do not depend on the choice of open neighbourhood and local coordinates. \label{DLDef}

In some cases semi-stable reduction can be achieved by {taking} $p: Y \rightarrow X_e$ to be the normalisation of a base change $X_e$ of $X$, and the components of the special fibre $Y_\sigma=  \widetilde{D_1} + \ldots + \widetilde{D_r}$ which lie above $D_1, \ldots, D_r$  give indeed the Denef-Loeser coverings described here. We address such a situation in the following proposition.

\begin{Prop} \label{DL}Let  $f: X \rightarrow B = \Spec \OO$ be a flat morphism, we assume that $X$ is smooth over $k_0$,  with the generic fibre $X_\eta$  smooth over $\eta$. Suppose $(X_\sigma)_{red}$ is a normal crossing divisor, $X_\sigma = aD_1 + b D_2 $, with each $D_i$ smooth. Suppose in addition that $gcd(a,b) =1 $, and if $char k =p > 0$ then $p \nmid a,b$. Let $e=ab$. \\
	Form the base-change $X_e$ as defined above and let $Y\to \Spec \OO_e$ be the normalisation of $X_e$, with the induced morphism $h:Y\to X$. Let $E_i=h^{-1}(D_i)_{red}$, $i=1,2$. Then
	\item[ (1)] $Y$ is a smooth $k_0$-scheme.
	\item[ (2)]  $E_1$ and $E_2$ are smooth divisors on Y, intersecting transversally. In particular, $Y_\sigma=E_1+E_2$ is a simple normal crossing divisor and $Y$ is a semi-stable reduction of $X$. 
	\item[ (3)]  The maps $E_I\to D_I$, $ \emptyset \neq I \subset \{1,2\}$, are isomorphic to the Denef-Loeser covers $\widetilde{D}_I\to D_I$. 
\end{Prop}

\begin{proof}
 
	Let $m,n$ be integers such that $1=ma+nb$.  \\
		For the first assertion, take $q\in Y$, we will show that $Y$ is smooth over $k_0$ at $q$. If $q\in Y_{\eta_e} \simeq X_{\eta_e} $, then as $B$ is smooth over $k_0$ and $B_e\to B$ is tame, $B_e$ is also smooth over $k_0$. Since $X_{\eta_e}$ is smooth over $\eta_e$, we see that $Y$ is smooth over $k_0$ at $q$. \\
		If $q$ is a point of $Y_\sigma$, let $p=h(q)$. We deal separately with the cases $p \in D_{12} $, $p \in {D_{1}^\circ} $, $p \in {D_{2}^\circ} $.
	
\item[ \textbf{Case 1}]	{For $p\in D_{12}$,} $f$ may be locally described on some affine open $U \ni p$ by $t = u x^a y^b $, $x,y \in \OO_X(U)$ local coordinates on $ U $ with $V(x) = D_{1} \cap U$,  $V(y) = D_{2} \cap U$ and $u \in \OO_X(U)^\times$.	We may assume $u=1$ as $u x^a y^b = u^{ma+nb} x^a y^b  = (u^m x)^a (u^n y)^b $ and we may replace $x$ and $y$ by unit multiples. \\
		In the $e$-base change scheme $X_e$, where we take $s$ with $s^e = t$, the defining equation on $U_e$  becomes   $s^e = x^a y^b  $. \\
		Normalisation can be achieved by adjoining roots $z^b=x $, $w^a = y $ as follows: Set $z = \frac{s^{am} x^n}{y^m}$, $w = \frac{s^{bn} y^m}{x^n}$ and let $V = h^{-1} (U)$.  Then $z$ and $w$ are in  $Frac (\OO_{X_e}(U_e)) $ and satisfy the integral equations above, so $z$ and $w$ are in the normalisation $\OO_Y (V)$, and in addition satisfy the equation $z \cdot w = s$.
	
	Now consider the ring $ \OO_X(U)[s,z,w] \subset \OO_Y(V)$. We claim that in fact $ \OO_X(U)[s,z,w] = \OO_Y(V)$ and that $V$ is smooth over $k$. {Indeed, as} local coordinates $x,y$ define an \'etale map $ \Spec \OO_X (U) \rightarrow \Aa^2_{k_0}$, or equivalently an \'etale ring extension $ k_0[X,Y] \rightarrow \OO_X(U) $. The algebraic picture after adjoining $s,z,w$ to the ring $\OO_X(U)$ is described by the following {commutative} diagram:
			\[ \begin{tikzcd}
			k_0[X,Y]  \arrow[r, ] \arrow[d]
			& \OO_X(U) \arrow[d]  \\
			k_0[X,Y,S,Z,W] / (S-ZW, S^e - X^a Y^b, Z^b - X, W^b - Y) \arrow[r] 
			&  \OO_X (U)[s,z,w]			
		\end{tikzcd} \]
{which induces a surjective  homomorphism 
\[
\phi:\OO_X(U)\otimes_{k_0[X,Y] }k_0[X,Y,S,Z,W] / (S-ZW, S^e - X^a Y^b, Z^b - X, W^a - Y)
\to \OO_X (U)[s,z,w].
\]
We claim that $\phi$ is an isomorphism. To see this, denote the  quotient ring in the left lower corner  by $C$.  Of the equations defining $C$,  the second is redundant as it follows from the other three, the first makes the variable $S$ redundant, and the last two make $X$ and $Y$ redundant, so we can write $C \simeq k_0[Z,W]$.
Since $k_0[X,Y]\to \OO_X(U)$ is  smooth, the homomorphism $k_0[Z,W]\to \OO_X(U)\otimes_{k_0[X,Y] }k_0[Z,W]$ is  smooth as well, hence $\OO_X(U)\otimes_{k_0[X,Y] }k_0[Z,W]$ is smooth over $k_0$, of Krull dimension equal to the Krull dimension of $ \OO_X(U)$. From the equations defining $C$ we can deduce  the further relations
\begin{equation}\label{FurtherRelations}
ZY^m=X^nS^{am}, WX^n=S^{bn}Y^m.
\end{equation}
}
{From the relations $S^e=X^a Y^b$, $S=ZW$, and $t=x^ay^b$, we see that canonical map $\OO_X(U)\to \OO_X(U)\otimes_{k_0[X,Y] }k_0[Z,W]$ extends to 
$\OO_X(U)[s]/(s^e-t)\to  \OO_X(U)\otimes_{k_0[X,Y] }k_0[Z,W]$ by sending $s$ to $1\otimes ZW$.
After inverting $x$ and $y$, the relations \eqref{FurtherRelations} and the universal property of the localization yield an extension of this homomorphism to
\[
\psi:\OO_X(U)[x^{-1}, y^{-1}][s, z,w]\to \OO_X(U)[x^{-1}, y^{-1}]\otimes_{k_0[X,Y] }k_0[Z,W]
\]
sending $z$ to $1\otimes Z$, $w$ to $1\otimes  W$; $\psi$ then defines an inverse to $\phi$, after inverting $x$ and $y$. Furthermore, the extension $k_0[X,Y]\to k_0[Z,W]$ is flat, so $\OO_X(U)\to \OO_X(U)\otimes_{k_0[X,Y] }k_0[Z,W]$ is flat as well, and thus $x$ and $y$ are non-zero divisors on $\OO_X(U)\otimes_{k_0[X,Y] }k_0[Z,W]$. As $\OO_X(U)\otimes_{k_0[X,Y] }k_0[Z,W]$ and $ \OO_X (U)[s,z,w]$ have the same Krull dimension and  both rings are   finite type $k_0$-algebras, the surjective, birational $k_0$-algebra homomorphism $\phi$ has zero kernel (by Krull's principal ideal theorem), hence is an isomorphism, as claimed. }

In addition, this shows that $\OO_X (U) [s,z,w] $ is a smooth $k_0$-algebra. Since $\OO_X (U) [s,z,w] $  contains $\OO_{X_e} (U_e)$ and is contained in the normalisation $\OO_Y(V)$ we have the desired equality  $ \OO_X(U)[s,z,w] = \OO_Y(V)$ and hence $V \subset Y$ is smooth. This also verifies that  $Y_\sigma \cap U$, defined by $s=0 = z \cdot w$, is a reduced divisor, $Y_\sigma \cap U \simeq \Spec  \OO_X(U)[z,w] /(zw) $, with $V \cap \widetilde{D_1} = V(z) $ and $V \cap \widetilde{D_2} = V(w) $.
	
	By definition of the Denef-Loeser covers, since $gcd(a,b) =1$, $ \widetilde{D_{12}}\simeq D_{12} \xrightarrow{id} {D_{12}} $. But also 
\[
	{E_{12}} \cap V \simeq \Spec \OO_X (U) [z,w] /(z,w)  \simeq \Spec (\OO_X(U) \otimes _{k[X,Y]} k[Z,W])/(Z,W) \simeq \Spec \OO_X(U)/(x,y) \simeq  D_{12} \cap U .
\]
		Thus $E_{12}$ coincides with the Denef-Loeser cover $ \widetilde{D_{12}} \simeq D_{12} \simeq E_{12} $.
	
\item[ \textbf{Case 2}]	Consider the case $p \in {D_{1}^\circ} $; the case $p \in {D_{2}^\circ} $ is handled the same way. There is a neighbourhood $U \ni p$ on which $f$ is described as $f^*(t) = t  = u \cdot x ^a$ with $u \in \Gamma (U, \OO_X)^\times $ and $U \cap \widetilde{D_1} = V(x)$. After $e$-base change we have the equation $s^e = u \cdot x^a$. Set $v = \frac{s^b}{x}$, then $v^a = u$, so {$v$ is in $\OO_Y(V)$.}
		In a similar manner to the previous case we wish to describe the ring $\OO_Y(V)$, to ascertain that $V \subset Y$ is smooth. We have to show that the inclusion $  \OO_X(U)[s,v] \subset \OO_Y(V) $ is an equality. {For this, we define a commutative square} 
		\[ \begin{tikzcd}
		k_0[W, W^{-1}, X]  \arrow[r, ] \arrow[d]
		& \OO_X(U) \arrow[d]  \\
		k_0[W, W^{-1}, X, V, S] / (V^a - W, S^b - VX) \arrow[r] 
		&  \OO_X (U)[v, s]			
	\end{tikzcd} \]
{where the upper horizontal morphism is defined by $W \mapsto u $, $X \mapsto x$, and the lower one by $V \mapsto v$, $S \mapsto s$. } We have the isomorphism
	\[
	k_0[W, W^{-1}, X, V, S] / (V^a - W, S^b - VX) \simeq k_0[V, V^{-1}, X, S]/({S^b}{V^{-1}} - X) \simeq k_0[V, V^{-1}, S].
	\]
{As in the previous case, one shows that the square induces an isomorphism
\[
\OO_X(U)\otimes_{k_0[W, W^{-1}, X]}k_0[V, V^{-1}, S]\xrightarrow{\sim} \OO_X (U)[v, s],
\]
so $ \OO_X (U)[v, s]$ is a smooth $k_0$-algebra and is therefore equal to the normalisation $ \OO_Y(V)$.} Thus $V \subset Y$ is smooth and $Y_\sigma \cap V$, being defined by $s=0$, is a smooth divisor on $V$.
	
	We can now show that $\widetilde{D_{1}} \simeq E_1$ over $D_1$. Let $ \pi : \widetilde{D_{1}} \rightarrow D_1$ be the Denef-Loser covering, $U$ {being} the same neighbourhood of $p \in D_1^\circ$ {as above}.  Then by definition $\pi^{-1} (D_1 \cap U) = \Spec (\OO_X (U) [T] /(T^a - u) )/(x) \simeq \OO_X(U)(v) /(x) $.
	On the other hand \[{E_1} \cap V = \Spec \OO_X(U)[v,s] / (s) \simeq \Spec \OO_X(U)[v] / (x) . \]
	We get $ {E_1} \cap V \simeq \pi^{-1} (D_1 \cap U)$. Since $ {E_1} $ is normal and $\widetilde{D_{1}}$ is the normalisation of ${D_{1}}$ in $\pi^{-1}(D_1 \cap U)$, we get $ \widetilde{D_{1}} \simeq E_1 $.
	In the same way $\widetilde{D_2} \simeq E_2$. {This completes the proof of (1), (2) and (3).} 	
\end{proof}

\begin{Rem}\label{RemarkIntegral} With  $f: X \rightarrow B = \Spec \OO$, $a, b$, and $e=ab$ as in 
Proposition~\ref{DL}, suppose that $X$ is irreducible and that $a=1$. We retain the notation of  
Proposition~\ref{DL}.  We claim that the base-change
  $X_e$ is integral. To see this, let $x$ be a generic point of $D_1$. Since $X$ is smooth, $D_1$ is a Cartier divisor on $X$ and thus the local ring $\OO_{X,x}$ is a dvr. Moreover, since $a=1$, $t$ is a parameter for $\OO_{X,x}$. Let $y\in X_e$ be the unique point lying over $x$. Then 
 \[
  \OO_{X_e,y}=\OO_{X,x}\otimes_\OO\OO[s]/s^e-t=\OO_{X,x}[s]/s^e-t.
  \]
 Since $e$ is prime to the characteristic, $\OO_{X_e,y}$ is smooth over $k$, so $\OO_{X_e,y}$  is a normal local ring, hence integral. Since $X_e\to X$ is finite and flat, each irreducible component of $X_e$ dominates $X$, and thus $X_e$ is irreducible and is also reduced in a neighbourhood of $y$. Since $X_e$ is a hypersurface in the smooth $k$-scheme $X\times_k\Spec k[s]$, $X$ is Cohen-Macaulay, and the fact that $X_e$ is irreducible and generically reduced then implies that $X_e$ is integral. 
\end{Rem}

\begin{Prop} \label{AIS} \cite[Theorem 8.6]{AIS}
	Let $f: X \rightarrow \Spec \OO$ be a flat, quasi-projective morphism, with $X$ smooth over $k_0$, and with generic fibre $X_\eta$ smooth over $\eta$. Suppose that the special fibre $(X_\sigma)_{red}$ is a normal crossing divisor, and $X_\sigma = \sum a_i D_i $; if $char k =p >0$, we suppose in addition that $\Psi_f \one$ is a dualisable object, and that $p \nmid \prod_ia_i$. Assume that for all $i \neq j$ $\text{gcd}(a_i, a_j)=1$, and that there are no triple intersections, i.e. for each triple of distinct indices $i, j, k$, $D_i\cap D_j \cap D_k = \emptyset $.	Denote by $\widetilde{D_i}$, $ \widetilde{D_{i}^\circ} $, $\widetilde{D_{ij}}$ the Denef-Loeser coverings. Then:
	\begin{enumerate}
		\item
	If for some $i$, $D_i$ intersects only a single additional stratum $D_j$, we have	\[  \chic(\Psi_f(\one_{X_\eta})|_{D_i}) = \chic(\widetilde{D_{i}}) - \chic(\mathbb{A}^1 \times \widetilde{D_{ij}}). \]
	\item
	\[  \chic(\Psi_f(\one_{X_\eta})) = \sum_i \chic(\widetilde{D_{i}^\circ}) - \sum_{i < j}  \chic(\mathbb{G}_m \times \widetilde{D_{ij}}). \]
		\end{enumerate}
\end{Prop}	

\begin{proof}
	To analyse each intersection separately consider \[ X_{ij}:= X \setminus \bigcup_{k \neq i,j}  D_k, \]
and set $D_\ell':=D_\ell\setminus \bigcup_{k \neq i,j}  D_k$. Then ${(X_{ij})}_\sigma = a_i D'_i + a_j D'_j $. Define $f_{ij} = f|_{X_{ij}} : X_{ij} \rightarrow B$. \\
		By Proposition~\ref{DL}, $X_{ij}$ admits a semi-stable reduction $Y_{ij}$ with components of the special fibre giving the Denef-Loeser coverings $\widetilde{D_i'}\to D_i'$, $\widetilde{D_j'}\to D_j'$ and $\widetilde{D_{ij}}=D_{ij}$. Note that  $(D_i')^\circ=D_i^\circ$ and $(D_j')^\circ=D_j^\circ$, so $\widetilde{D_i^{\prime\circ}}=\widetilde{D_i^\circ}$ and $\widetilde{D_j^{\prime\circ}}=\widetilde{D_j^\circ}$. We can use Proposition \ref{example2} to get 	
	\[ \chic(\Psi_{f_{ij}}({\one_{X_{ij}}}_\eta)) =  \chic(\widetilde{D_{i}^\circ})   + \chic(\widetilde{D_{j}^\circ}) -  \chic( \mathbb{G}_m \times \widetilde{D_{ij}}). \]
	{By the same argument applied to $X_{ij}\setminus D_{ij}$, we find 
\[ \chic(\Psi_{f_{ij}}({\one_{X_{ij}}}_\eta)|_{X_{ij} \setminus D_{ij}}) =  \chic(\widetilde{D_i^\circ})   + \chic(\widetilde{D_j^\circ}), \]
and by cut and paste, we have 
\[ \chic(\Psi_{f_{ij}}({\one_{X_{ij}}}_\eta)|_{D_{ij}}) = \chic(\Psi_{f_{ij}}({\one_{X_{ij}}}_\eta)) - \chic(\Psi_{f_{ij}}({\one_{X_{ij}}}_\eta)|_{X_{ij} \setminus D_{ij}}), \]
so 
\[ \chic(\Psi_{f_{ij}}({\one_{X_{ij}}}_\eta)|_{D_{ij}}) = -  \chic( \mathbb{G}_m \times \widetilde{D_{ij}}).
\]
Similarly,
\[
\chic(\Psi_{f_{ij}}({\one_{X_{ij}}}_\eta)|_{D_i^\circ})=\chic(\widetilde{D_i^\circ}).
\]}
		{Since $X_{ij}$ is an open neighbourhood of $D_{ij}$ in $X$, the compatibility of $\Psi_{(-)}$ with respect to the smooth morphism $X_{ij}\hookrightarrow X$ (Property~\ref{property}) implies}
	\[ \Psi_f({\one_{X}}_\eta)|_{D_{ij}} = \Psi_{f_{ij}}({\one_{X_{ij}}}_\eta)|_{D_{ij}}. \]
Similarly,
	\[ \Psi_f({\one_{X}}_\eta)|_{D_i^\circ} = \Psi_{f_{ij}}({\one_{X_{ij}}}_\eta)|_{D_i^\circ}. \]
	Now for the first identity use the cut and paste relation along $D_i = D_i^\circ \cup D_{ij}$,
	\begin{align*} \chic(\Psi_f({\one_{X}}_\eta)|_{D_i}) =  \chic(\widetilde{D_i^\circ}) -  \chic( \mathbb{G}_m \times \widetilde{D_{ij}})  = & \chic(\widetilde{D_i^\circ}) + \chic(\widetilde{D_{ij}}) -  [ \chic(\widetilde{D_{ij}}) + \chic( \mathbb{G}_m \times \widetilde{D_{ij}}) ]
		\\ = & \chic(\widetilde{D_i}) - \chic( \mathbb{A}^1 \times \widetilde{D_{ij}}).  
	\end{align*}
For the second identity use cut and paste along $X_\sigma = \coprod_i D_i^\circ \coprod \coprod_{i < j} D_{ij} $ to get
	\[ \chic(\Psi_f({\one_{X}}_\eta)) = \sum_i \chic(\Psi_f({\one_{X}}_\eta)|_{D_{i}^\circ}) + \sum_{ij} \chic(\Psi_f({\one_{X}}_\eta)|_{D_{ij}}) = \sum \chic(\widetilde{D_{i}^\circ}) - \sum  \chic(\mathbb{G}_m \times \widetilde{D_{ij}}) . \]
	\end{proof}
	
	\begin{Rem}  \label{qAIS}
		The second formula is in fact a special case of a result by Ayoub-Ivorra-Sebag \cite[Theorem 8.6]{AIS} which is proven for a general simple normal crossing divisor, relying on the theory of rigid analytic motives. The case considered here suffices for our use in this paper and follows from the same construction we use for our main result so our proof here. We also obtain a formula similar to the one above in Proposition \ref{AIS} in our quasi-homogeneous case, in that case $X_\sigma$ is not a normal crossing divisor, though we can still obtain semi-stable reduction and strata which are essentially the same as the Denef-Loeser coverings, see Remark \ref{qRem}.
	\end{Rem}	

\subsection{Nearby cycles at the base}
We continue with the notation of a discrete valuation ring $\OO$, with subfield $k_0$,  residue field $k$, fraction field $K$ and parameter $t$, and let $B=\Spec\OO$, as in Section~\ref{SectionNearbyCycles}; in this section, however, we assume in addition that $k_0$ has characteristic zero. We have a ring homomorphism $\Sp_t$, {(see \cite[Section 5.1]{Le20b})} from the Grothendieck-Witt ring of the fraction field $K$ to that of the residue field $k$, characterised as the unique map 
\[ \Sp_t : \GW(K) \rightarrow \GW(k) \]
satisfying:

(1) $\Sp_t(t) =\< 1 \> $ for the uniformizer $t$.

(2) $ \Sp_t (u) = \<\bar{u} \> $ for all invertible elements $u \in O^{\times} $ where $\bar{u}$ denotes the image of $u$ under the quotient map $ \OO \rightarrow k $.

Given a strongly dualisable object $\alpha \in \SH(K) $, the motivic Euler characteristic $ \chi(\alpha) $ is an endomorphism of $\SH(K) $, and so the functor {$\Psi_{id}:\SH(K)\to \SH(k)$} can be applied to it and produce an endomorphism of the unit in $\SH(k)$. Via the Morel isomorphism we get an object in $\GW(k)$. We state results from \cite{Le20b}, which follow from the fact that $\Psi_{id}$ is a monoidal functor in characteristic $0$ \cite[Corollaire 3.5.19]{Ay07a}.

\begin{Prop}[\hbox{\cite[Lemma 8.1]{Le20b}}] \label{PsiId}	For $\alpha\in \SH(K)$, we have
$ \Psi_{id *}(\chi(\alpha)) = \chi (\Psi_{id} (\alpha)).$
\end{Prop}

In fact,  $\Sp_t$ computes Ayoub's functor $ \Psi_{id} $.

\begin{Prop}[\hbox{\cite[Proposition 8.2]{Le20b}}] \label{simon} 	The following diagram commutes.
	\[
	\begin{tikzcd}[row sep=small ]
		\End_{\SH(K)}(\one_K) \arrow[r, "\Psi_{id *}"] \arrow[d, "\sim"] & \End_{\SH(k)}(\one_k)  \arrow[d, "\sim"] \\
		\GW(K) \arrow[r, "sp_t"] & \GW(k) \\
	\end{tikzcd}
	. \]
	Here the vertical arrows are Morel's isomorphisms.
\end{Prop}

\section{The homogeneous case}

{We continue to use our discrete valuation ring $\OO$ and base-scheme $B:=\Spec\OO$, and retain the notations and assumptions from Section~\ref{SectionNearbyCycles}.} We make the following assumption on the special fibre.}

\begin{Ass} \label{assumption}
	The reduced special fibre $X_\sigma$ has only isolated singularities $p_1, \ldots, p_r$. Moreover, if  
	  $\hat{X}=Bl_P(X) $ is the blow up of $X$ at $P:=\{p_1,\ldots, p_r\}$, $E=E_{p_1}\amalg\ldots\amalg E_r$ the exceptional divisor and $\pi^{-1}[X_\sigma]:= \overline{\pi^{-1}(X_\sigma\setminus \{p_1,\ldots, p_r\})} $ the proper transform, then $\pi^{-1}[X_\sigma]$ is smooth over $k$ and intersects each $E_i$ transversally.  
\end{Ass}

Let $p$ be a singular point of $X_\sigma$. Now consider function $f^*_p : \OO \to \OO_{X,p}$ (which we sometimes denote just by $f^*$) defined by the composition of $f^* : \OO \to \OO_X(X)$ and the localisation $\OO_X(X)\to \OO_{X,p} $ . We show that Assumption \ref{assumption} is equivalent to having an 'analytic expansion' of $f$ at each singular point $p$ of the form
 \[f_p^*(t) = F(s_0,\ldots,s_n) + h \] 
 with {$s_0,\ldots,s_n$ local coordinates at $p$, } $F$ a homogeneous polynomial of degree $e$ defining a smooth projective hypersurface over $k(p)$, and $h \in m_p^{e+1}$, where $m_p$ is the maximal ideal in $\OO_{X,p}$. We say then that at $p$, \emph{$f$ looks like the homogeneous singularity defined by $F$ }(see Definition \ref{lookslike}).

\begin{Prop}  \label{firstprop}
	Assumption \ref{assumption} above is equivalent to the following two conditions:
	
	(1) The special fibre $X_\sigma$ has only isolated singularities.
	
	(2) At each singular point $p$, let  $\OO_{X,p}$ denote the local ring at $p$, with maximal ideal $m_p$, let $e_p$ be the maximal integer with $f^*(t) \in m_p^{e_p}$, and let $\overline{f^*(t)}_p$ be the image of $f^*(t)$ in $m^{e_p}_p/m^{e_p+1}_p$.  Then $\overline{f^*(t)}_p$ defines a smooth hypersurface in $\Proj \Sym^*(m_p/m^2_p) \simeq \mathbb{P}^n_{k(p)}$.
	
	{Moreover, if Assumption \ref{assumption}} is satisfied then for each singular point $p$ there is a neighbourhood $U$ such that, letting $\hat{U}\to U$ denote the blow-up of $U$ at $p$, the special fibre $\hat{U}_\sigma$ decomposes as $\hat{U}_\sigma = e_pD_1 + D_2 $ with  $D_1 \simeq \P_{k(p)}^n$  the reduced exceptional divisor and $D_2 =\pi^{-1}[U_\sigma] $  the strict transform of $U_ \sigma$. Both $D_1$ and $D_2$ are smooth and intersect transversely, with $D_1\cap D_2\subset D_1$ the hypersurface defined by $\overline{f^*(t)}_p$.
	
\end{Prop}

\begin{proof} 
Let $p$ be a singularity and let $(s_0,\ldots,s_n)=m_p$ be a regular sequence of parameters on the maximal ideal $m_p$ of $\OO_{X,p}$. We write
\[ f^*(t)=F(s_0,\ldots,s_n) +h \] with $F$ a homogeneous polynomial of degree $e$ with coefficients in $\OO_{X,p}$, and $h\in m_p^{e+1}$.
	
	 $\overline {f^*(t)} = F( \bar{s_0}, \ldots, \bar{s}_n ) $ is a homogeneous equation defining an hypersurface in $\mathbb{P}^n_{k(p)}$, $k(p)$ the residue field of $\OO_{X,p}$. We show that this hypersurface is isomorphic to the intersection $D_{12}$.
	 
Define
	\[
	\hat{X}=Bl_p(\Spec \OO_{X,p})= \Proj \OO_{X,p}[T_0,\ldots , T_n]/(s_iT_j-s_jT_i)_{i < j}
	\]
	
		Let $\hat{X}= \bigcup U_i $ be the standard covering of the blow up, where $U_i$ is defined by $T_i \neq 0$. 
	
	For simplicity of notation we describe $U_0$ but the argument is similar for each of the $U_i$.  
	Use $s_0$, $t_1=T_1/T_0, \ldots, t_n=T_n/T_0$ as coordinates on $U_0$.
	\[
	U_0=\Spec \OO_{X.p}[T_1/T_0,.. ,T_n/T_0]/(s_iT_j-s_jT_i)_{i,j} = \Spec \OO_{X,p}[t_1,\ldots,t_n]/(s_0 t_1- s_1,\ldots,s_0 t_n - s_n)
	\]
	We may write now   
	\[ f^*(t)=s_0^e \cdot (F(1,t_1,\ldots,t_n) +s_0\tilde{h})=:s_0^e \cdot g_0 \] with $\tilde{h}\in m_p$.
	Then $D_1 \cap U_0 = V_{U_0}(s_0)$, $D_2 \cap U_0 = V_{U_0}(g_0)$ and $D_{12}= V_{U_0}(s_0, g_0)$; We have $(U_{0})_ \sigma = e\cdot(D_1 \cap U_0) + D_2 \cap U_0 $ and similarly for all $i$, so $\hat{X}_\sigma = e \cdot D_1 + D_2$.

So  $D_1 \cap U_0 \simeq \Spec \OO_{X,p}[t_1,\ldots,t_n]/(s_0, s_1,.., s_n)  \simeq \Spec k(p)[t_1,\ldots,t_n]$. We have a similar computation for each $i$. This shows that the $D_1\cap U_i$ form the standard affine chart for the projective space $\P^n_{k(p )}$, giving the isomorphism $D_1 \simeq \mathbb{P}^n_{k(p)}= \Proj k(p) [T_0,\ldots,T_n]$, with $D_1 \cap U_i$ defined as usual as the open subscheme $T_i \neq 0$.
	
	 $D_{12} \cap U_0$ is defined then by $F(1,t_1,\ldots,t_n) = 0$ inside $D_1 \cap U_0$; making the same construction for general $i$ shows that $D_{12}\cap U_i$ is defined by $F(t_1,\ldots, t_{i-1}, 1, t_i,\ldots, t_n)=0$ inside $D_1\cap U_i=\Spec k(p)[t_1,\ldots, t_n]$, with $t_j=T_{j-1}/T_i$ for $j=1,\ldots,i$ and 
	 $t_j=T_j/T_i$ for $j=i+1,\ldots, n$. This shows that $D_{12}$ is globally defined in $D_1\simeq \P^n_{k(p )}$ by $F$, as claimed. Thus the condition in the statement of the proposition is equivalent to the smoothness of $D_{12}$
	 
	 Now, since the blow-up of $X$ is smooth, $\text{codim}(D_1)=\text{codim}(D_2)=1$ in the blow-up, and $\text{codim}(D_{12})=2$ being a hypersurface in $D_1$, the condition of the proposition is equivalent to Assumption \ref{assumption}. 
\end{proof}	
In the following theorem we compute explicitly the strata of the special fibre of a semi-stable reduction, constructed according to Proposition \ref{DL}.

\begin{Th}
 \label{mainprop}   Let $f:X\to \Spec\OO$ be a flat quasi-projective morphism with $X$ smooth of dimension $n+1$ over $k_0$.   Suppose that $X_\sigma$ has a single singular point $p$ and that at $p$, $f$ looks like the 
 homogeneous singularity defined by $F\in k(p )[T_0,\ldots, T_n]$ of degree $e$, and that $V(F)\subset \P^n_{k(p )}$ is a smooth hypersurface. We suppose in addition that $e$ is prime to the exponential characteristic of $k_0$.  \\
   Let $q:\hat{X}\to X$ be the blow-up of $X$ at $p$ and let $\OO_e=\OO[s]/(s^e-t)$.  Let $D_1\subset \hat{X}$ be the exceptional divisor and let $D_2\subset \hat{X}$ be the proper transform of $X_\sigma$.  Then  there exists a quasi-projective morphism  $Y\to \Spec\OO_e$ and a morphism $\pi:Y\to \hat{X}$ over $\Spec\OO_e\to \Spec\OO$ such that 
	 \item[ (1)] $\pi$ defines a semi-stable reduction of $X$.
 \item[ (2)] The special fibre $Y_\sigma$ is of the form  $\widetilde{D_{1}}  + \widetilde{D_2}$ with $\widetilde{D_{1}}$  and $\widetilde{D_2}$ smooth,  with intersection $\widetilde{D_{12}}$, and with $\pi$ mapping  $\widetilde{D_{1}}$ to $D_1$,   and $\widetilde{D_2}$ to $D_2$.
 \item[ (3)] We have
\[
\widetilde{D_1} \simeq V(F-T_{n+1}^e) \subset \mathbb{P}^{n+1}_{k(p)};
\] and 
\[
\widetilde{D_{12}} \simeq V(F) \subset \mathbb{P}^n_{k(p )};
\]
the maps $ \pi:\widetilde{D_2}\to D_2,\ \pi:\widetilde{D_{12}}\to D_{12}:=D_1\cap D_2 $ are isomorphisms;
the morphism $\widetilde{D_1}\to D_1=\P^n_{k(p )}$ is the evident cyclic cover, induced by the projection  $\P^{n+1}_{k(p )} \setminus\{(0,\ldots,0,1) \}\to \P^n_{k(p )}$ from $(0,\ldots,0,1) $.
 \item[ (4)]  $\widetilde{D_1}\to D_1$, 
$\widetilde{D_2}\to D_2$ and $\widetilde{D_{12}}\to D_{12}$ are the Denef-Loeser coverings as in Definition \ref{DLDef}. 
\end{Th}

\begin{proof} 
		 
 By Proposition \ref{firstprop}, $\hat{X}_\sigma = e  D_1 + D_2$ {with $D_1 \simeq \P^n$ and $D_2\to X_\sigma$ a resolution of th singularity of $X_\sigma$,} and so $f \circ q : \hat{X} \rightarrow \Spec \OO$ satisfies the requirements of Proposition \ref{DL} (with $a=e$, $b=1$).
 Then we have the scheme $Y$ constructed by first forming the base-change { by $\OO\to \OO_e$,} and then taking the normalisation. By Proposition \ref{DL}, $Y$ is a semi-stable reduction for $\hat{X}$. That is, $Y$ is smooth over $k_0$ and  $Y_\sigma = \widetilde{D_{1}} + \widetilde{D_2} $ is a (reduced) simple normal crossing divisor. Also if we denote by $h$ the composition
  \[h : Y \rightarrow \hat{X}_e \rightarrow \hat{X} \rightarrow X, \]
  then $\widetilde{D_I} = h^{-1}(D_I) \rightarrow D_I$ are the Denef-Loeser coverings for all ${\emptyset\neq} I \subset \{1,2\}$. The only thing we have left to do is to give the explicit description of those coverings.

 By definition of Denef-Loeser covers and since $b=1$, $\widetilde{D_{12}} \simeq D_{12}$ and $\widetilde{D_2} \simeq D_2 $. By Proposition \ref{firstprop} then, $\widetilde{D_{12}} \simeq V(F) \subset \mathbb{P}^n_{k(p)}$. In the remaining part of the proof we shall describe $\widetilde{D_{1}}$.
 
We only need to check the explicit description of the covering  $\widetilde{D_1}  \rightarrow D_1$ after restriction over some neighbourhood of $p$ in $X$. Thus, we may replace $X$ with the local scheme $\Spec\OO_{X,p}$; we change notation and assume that $X=\Spec\OO_{X,p}$ is local.
Take the standard covering of the blow-up $\hat{X}= \bigcup U_i $, where $U_i$ is defined by $T_i \neq 0$.  Write again
	$ f^*(t)=F(s_0,\ldots,s_n) +h $ with $F$ a homogeneous polynomial of degree $e$  and $h\in m_p^{e+1}$.
	Take $s_0$, $t_1=T_1/T_0, \ldots, t_n=T_n/T_0$ as coordinates on $U_0$. {Then}
	\[
	U_0 \simeq \Spec \OO_{X,p}[t_1,\ldots,t_n]/(s_i-s_0  t_i) .
	\]
	On $U_0$,  
	$ f^*(t)=s_0^e \cdot (F(1,t_1,\ldots,t_n) +s_0\tilde{h})=:s_0^e \cdot g_0 $ with $\tilde{h}\in m_p$ {and $g_0=F(1,t_1,\ldots,t_n) +s_0\tilde{h}$}. After the base change, on $U_{0,e} = U_0 \times_\OO\Spec \OO[t']/(t'^e - t)$ we have
	\[ U_{0,e} \simeq \OO_{X,p}[t_1,\ldots,t_n, t']/(s_i-s_0 t_i, s_0^e \cdot g_0 - (t')^e ) . \]
	Normalising amounts to adjoining $t_{n+1} = t{'}/s_0 $, which is an integral element as $t_{n+1}^e= g_0$ [see the proof of Proposition \ref{DL}].
	So on $V_0$, the inverse image of $U_0$ in $Y$, we have
	\[
	V_0= \Spec( \OO_{X,p} [t_1,\ldots,t_n, t_{n+1} ] /( \{s_i- s_0  t_i\}_{1\le i\le n}, g_0 - (t_{n+1})^e )) .
	\]
The special fibre $Y_{\sigma}$ then is covered by the $V_i= h^{-1}(U_i)$. \\ The exceptional divisor $\widetilde{D_1}$ is the fibre along $\Spec k(p) \hookrightarrow \Spec \OO_{X,p} $, defined by $s_0 = 0$ on $V_0$, and so 
\[ 
	\widetilde{D_1} \cap {V_0} =  \Spec k(p)[t_1,\ldots,t_{n+1}]/({\bar{g}_0-t_{n+1}^e}) ,
	\]
	 where $ \bar{g}_0={F}(1,t_1,..,t_n) $. Set $\P^{n+1}_{k(p)} = \Proj k(p)[T_0,\ldots,T_{n+1}] = \bigcup_{i=0}^{n+1} W_i$ to be the standard affine covering, with $W_i$ corresponding to $T_i \neq 0$,  and identify $\widetilde{D_1} \cap {V_0}$ as embedded in the affine space $ W_0 = \Spec k(p)[t_1,\ldots,t_{n+1}] $   with $t_j = T_j/T_0 $. \\
	   In order to describe the cover $\widetilde{D_1}\cap V_0 \to D_1\cap U_0$, we also use the identification $D_1 = \P^n_{k(p)} = \Proj k(p)[T_0,\ldots,T_n] $ as in   Proposition \ref{firstprop}, with $D_1\cap U_0$ being  $\Spec k(p)[t_1,\ldots,t_n] $, still with $t_j = T_j/T_0$. 
We then get the restriction of the cover $\widetilde{D_{1}} \to D_1 $ to $V_0$ to be
	     \[	\begin{tikzcd}[row sep=small]
	   	\Spec k(p )[t_1,\ldots,t_n, t_{n+1}]/(F(1, t_1,\ldots,t_n)-t_{n+1}^e) \arrow[r, ] \arrow[d, hook] &  \Spec k(p )[t_1,\ldots, t_n]  \arrow[d, hook] \\
	   	V(F-T_{n+1}^e) \subset \P^{n+1}_{k(p)} \arrow [r] & \P^n_{k(p)} . 
	   \end{tikzcd}
 \]
This is the restriction of the cover $V_{\P_{k(p)}^{n+1}}(F-T_{n+1}^e)\to \P^n_{k(p )}$ over $W_0$. \\
	 	  Similarly,   for each $i = 0,..,n$,  the cover  $\widetilde{D_1}\cap V_i \to D_1\cap U_i$ is \[\Spec k(p )[t_1,\ldots,t_n, t_{n+1}]/(F(t_1 ,\ldots,t_{i-1}, 1, t_{i+1}, \ldots, t_n)-t_{n+1}^e)\to \Spec k(p )[t_1,\ldots, t_n] \] with $t_j:=T_j/T_i$ as in the proof of Proposition~\ref{firstprop}. Here we are  considering $V_i \cap \widetilde{D_1}$ as  a closed subscheme of  $W_i$.
	  We get $\widetilde{D_1}\cap V_i = V(F-T^e_{n+1})\cap W_i$ in  $\P^{n+1}_{k(p )}$. These restrictions of $\widetilde{D_{1}} \to D_1$ to $V_i$, patch together then to give exactly the desired cover \[V_{\P^{n+1}_{k(p)}}(F(T_0,\ldots,T_n)-T_{n+1}^e)\to \P^n_{k(p )}. \]
	  To be precise, the {open subschemes} we described here are $V(F-T_{n+1}^e) \cap W_i$ for $i=0,\ldots,n$, and in principle we should also consider the remaining {open} $V(F-T_{n+1}^e) \cap W_{n+1} $. This  {open} is defined by $F(y_0,\ldots,y_n) -1 =0$ on $W_{n+1}=\Spec k(p)[y_0,\ldots,y_n]$ with $y_i = T_i / T_{n+1}$, $i= 0,\ldots,n$. But since $F$ is homogeneous, $y_0,\ldots,y_n$ satisfying this equation cannot be all $0$, so at least one $T_i \neq 0$, $i < n+1 $, and the point falls in some $W_i$, $i < n+1$. So this remaining {open} is contained in the union of the others, and {is} therefore redundant for our covering describing $ V(F-T_{n+1}^e)$.
	  	 
	 $\widetilde{D_{12}}$ is given locally on $V_i$ by both $s_i=0$ and $t_{n+1}=0$, and so by the description of $\widetilde{D_{1}} \to D_1$ above it is contained in the $ \mathbb{P}^{n}_{k(p)}\subset \mathbb{P}^{n+1}_{k(p)} $ given by $T_{n+1} = 0$. We have \[ \widetilde{D_{12}} \simeq D_{12} \simeq V(F) \subset \mathbb{P}^{n}_{k(p)} \] as we saw in Proposition \ref{firstprop}.
\end{proof}

\begin{Cor}\label{CorLocalHomogFormula}
	Let $f:X\to \Spec\OO$ be a flat quasi-projective morphism with $X$ smooth over $k_0$ and with $X_\eta$ smooth over $\eta$. Suppose that the special fibre $X_\sigma$ has an isolated singular point $p$, at which $f$ looks at like the homogeneous singularity defined by a homogeneous polynomial  $F\in k(p)[T_0,\ldots, T_n]$ of degree $e$,  with $V(F)\subset \P^n_{k(p)}$ a smooth hypersurface; if char $k_0 > 0$ assume that $gcd(e,p)=1$ and that $\Psi_f \one $ is a strongly dualisable object. Then
	\[
	\chic(\Psi_f({\one_X}_{\eta})|_p)= \chic(V(F-T_{n+1}^{e}))-
	\<-1\> \chic(V(F)).
	\]
	\end{Cor}
	
\begin{proof}
	Notice that by the construction in the theorem above, the preimage of the point $p$ is $\widetilde{D_{1}}$. Therefore by Corollary \ref{semistablered}
	\[
	\chic(\Psi_f \one_{X_{\eta}} |_p) = \chic (\Psi_g \one_{Y_\eta}|_{\widetilde{D_{1}}}).
	\]
	Then we get the result by applying Proposition \ref{AIS} (1).
\end{proof}

 \begin{Cor}\label{CorHomogFormula}
  Let $f:X\to \Spec\OO$ be a flat quasi-projective morphism with $X$ smooth over $k_0$ and $X_\eta$ smooth over $\eta$. Suppose that the special fibre $X_\sigma$ has finitely many singular points $p_1,\ldots, p_r$, and for each $i$,  $f$ looks at $p_i$ like the homogeneous singularity defined by a homogeneous polynomial  $F_i\in k(p_i)[T_0,\ldots, T_n]$ of degree $e_i$,  with $V(F_i)\subset \P^n_{k(p_i)}$ a smooth hypersurface. If $char k =p >0$, suppose in addition that $\Psi_f \one$ is a dualisable object, and that $p \nmid \prod_ia_i$. Let $X_\sigma^\circ=X_\sigma\setminus\{p_1,\ldots, p_r\}$. Then
 \[
\chic(\Psi_f({\one_X}_{\eta}))=\chic(X_\sigma^\circ)+ \sum_{i=1}^r \chic(V(F_i-T_{n+1}^{e_i}))-
\<-1\> \sum_{i=1}^r \chic(V(F_i)).
\]
\end{Cor}

\begin{proof} 
This is a consequence of Proposition~\ref{local} and applying the previous Corollary for each $p_i$.
\end{proof}

\section{The quasi-homogeneous case}
We can extend the results of the previous section to larger class of singularities {for which} the defining polynomial is weighted homogeneous. The usual blow-up should be replaced by a weighted blow-up, but treating it is not as straightforward {as in the homogeneous case}. For example, the exceptional divisor of a weighted blow-up would be a weighted projective space, which is generally not smooth. Therefore the result of \cite[Theorem 8.6]{AIS}  (Proposition \ref{AIS}), cannot be applied as it is, as the special fibre is not a simple normal crossing divisor. However in view of Proposition \ref{example2} and using the construction of Proposition~\ref{DL}, if the covering strata are smooth we can still recover a similar semi-stable reduction construction. Presenting the scheme in the weighted case as a quotient of a scheme with a homogeneous singularity modulo a finite group allows {us to use the results} of the previous section; we show that the quotient defines a semi-stable reduction of our original degeneration with smooth strata at the special fibre. The end result is completely parallel to the homogeneous case, taking weights into account.

We retain our assumptions on the discrete valuation ring $\OO$ with residue field $k$ and parameter $t$ from Section~\ref{SectionNearbyCycles}.; as before,  we let $\sigma\hookrightarrow B:=\Spec\OO\hookleftarrow\eta$ denote the closed and generic points of $B:=\Spec\OO$, respectively, and we have the subfield $k_0$ of $\OO$, with $B$ smooth and essentially of finite type over $k_0$, and with $k_0\to k$ finite and separable. Let $f:X \to B$ be flat and quasi-projective, and $p\in X$ a closed point with stalk and maximal ideal $m_p\subset \OO_{X,p}$. Let $(s_0,\ldots, s_n)$ be a regular sequence generating 
$m_p$ and let $(a_0,\ldots, a_n)$ be a system of positive integral weights with $\text{gcd}(a_i, a_j)=1$ for every $i,j$. Define the ideal $m_{p,s_*, a_*}^{(\ell)}\subset \OO_{X,p}$ to be the ideal generated by monomials of weighted homogeneous degree $\ell$, that is, by monomials $s_0^{i_0}\cdot\ldots\cdot s_n^{i_n}$ with $\ell=\sum_ja_ji_j$.

\begin{Def} \label{qdef} Let $f: X \rightarrow \Spec \OO $ a flat quasi-projective morphism of schemes {with $X$ smooth over $k_0$ and  $X_\eta$ smooth over $\eta$.} Let $p\in X_\sigma$ be an isolated singular point and let $F\in k(p)[T_0,\ldots, T_n]$ be a homogeneous polynomial of weighted degree  $e$ for some weights $a_*=(a_0,\ldots, a_n)$ as above. We say that $(X_\sigma, p)$ {\em looks like the weighted homogeneous singularity defined by $F$} if there is a regular sequence of  generators
	for $m_p$  such that
	\[
	f^*(t)\equiv F(s_0,\ldots, s_n)\mod  m_p\cdot m_{p,s_*, a_*}^{(e)}.
	\]
\end{Def}
Here we have implicitly chosen a splitting of 
$\OO_{X,p}/m_p\cdot m_{p,s_*, a_*}^{(e)}\to k(p)$. 

\subsection{Weighted projective space}

We review the notion of weighted projective space as it appears in \cite{Le20b}. Let $R$ be a ring and $a=(a_0,..,a_n)$ a sequence of positive integers, which we call {\em weights}.  Let $R[X_0,..,X_n]$ be the graded ring with $X_i$ having degree $a_i$. Define 
\[ \P_R(a) = \Proj R[X_0,\ldots,X_n]. \]

An alternate description {of $\P_R(a)$} is as a quotient of $\P^n$ by the group scheme $\mu_a = \mu_{a_0} \times \ldots \times \mu_{a_n}$. \\
Let $\iota_a : R[X_0,..,X_n] \rightarrow R[Y_0,\ldots,Y_n] $ be the graded ring homomorphism mapping $X_i$ to $Y_i^{a_i} $, where the ring $R[X_0,..,X_n]$ is with the $a$-grading, and $R[Y_0,\ldots,Y_n] $ is with the usual grading on a polynomial ring.
Let $\mu_a$ act on $ R[Y_0,\ldots,Y_n]$ by $Y_i \mapsto \zeta_{a_i} Y_i$, for  $\zeta_{a_i} \in \mu_{a_i}$. Then the image of $\iota_a$ can be identified with the fixed ring $ R[Y_0,\ldots,Y_n] ^{\mu_a}$, hence defining 
\[ \pi: \P^n \rightarrow \P(a) \] as a quotient $\P(a) \simeq \P^n / \mu_a $. \\
{We may as well view the projective space $\P^n$ at the source of $\pi$ as achieved from $\P(a)$ by adjoining for each $i$ the $a_i$-th root of $X_i$.  We now describe a similar construction of a local version of a 'weighted blow-up' of our scheme $X$ in Definition~\ref{qdef}, retaining the notations from that definition. }

{As our construction is local around  the given point $p\in X_\sigma$, we replace $X$ with an affine open neighbourhood $U$ of $p$ in $X$, such that the local parameters $s_0,\ldots, s_n$ of Definition~\ref{qdef} extend to \'etale coordinates on $U$, that is, the morphism $(s_0,\ldots s_n):U\to \Aa^{n+1}_k$ is \'etale. We change notation and suppose $X=U$, and let $A$ denote the ring of functions on the affine scheme $X=\Spec A$. We let $\mathfrak{m}_p\subset A$ denote the maximal ideal of $p$ and following Definition~\ref{qdef}, we define $\mathfrak{m}_{p, s_*, a_*}^{(e)}\subset \mathfrak{m}_p$ as the ideal defined by monomials of weighted degree $e$ in the $s_i$.} 

\begin{Construction} \label{muconstruction}
	{With $p\in X=\Spec A$, $a_*=(a_0,\ldots, a_n)$ and $s_0,\ldots, s_n\in \mathfrak{m}_p$ \'etale coordinates on $X$ as above, define} $A[s^{1/a}] := A[\sigma_0, \ldots , \sigma_n] /(\sigma_0^{a_0}-s_0,\ldots,\sigma_n^{a_n}-s_n) $ and let $Z = \Spec A[s^{1/a}]$.
 {Let $\mu_a= \mu_{a_0} \times \ldots \times \mu_{a_n} $. We have the $\mu_a$-action on $A[s^{1/a}]$, where $\zeta\in \mu_{a_i}$ acts by
\[
\zeta\cdot \sigma_j:=\begin{cases}\zeta\sigma_i&\text{ for }j=i\\ \sigma_j&\text{ for }j\neq i.
\end{cases}
\]
Then $A$ is equal to the subring of $\mu_a$-invariants in $A[s^{1/a}]$, $A=A[s^{1/a}]^{\mu_a}$, and so the}
  map \[\pi : Z \rightarrow X \] realises $X$ as the quotient of $Z$ by the action of the group scheme $\mu_a$. {Also, there is a unique point $q\in Z$ lying over $Z$, and we have $k(q)=k(p )$. We let $\mathfrak{m}_q\subset A[s^{1/a}]$ denote the maximal ideal of $q\in Z$.}
An argument similar to that given in Remark~\ref{RemarkIntegral} shows that $Z$ is smooth over $k$ and if $X$ is integral, then so is $Z$.

{From Definition~\ref{qdef}, we have
 \[ 
 f^*(t) = F(s_0,..,s_n)+ h.
 \] 
After shrinking $X$ if necessary, and changing notation, we may assume that $h$ is in $\mathfrak{m}_p\cdot\mathfrak{m}_{p,s_*,a_*}^{(e)}\subset A$. 
Letting $g:=\pi\circ f:Z\to\Spec \OO$, we have
\[ g^*(t)= F(\sigma_0^{a_0},\ldots,\sigma_n^{a_n})+ h'\]
with $h' \in \mathfrak{m}_q^{e+1}\subset B$.
Let $G(Z_0,\ldots, Z_n)\in k(p )[Z_0,\ldots, Z_n]$ be the degree $e$ polynomial with
$G(\sigma_0,..,\sigma_n)=F(\sigma_0^{a_0},\ldots,\sigma_n^{a_n})$. }
\end{Construction}
 
 \begin{Def}[\hbox{\cite[Def. 4.2]{Le20b}}] \label{smhyper}
 	Let $F$, $G$ be defined as in the above Construction \ref{muconstruction}. We say that $V(F) \subset \P_{k(p)}(a)$ is a \textit{smooth quotient hypersurface} if the polynomial $G$ defines a smooth hypersurface $V(G) \subset \mathbb{P}^n_{k(p)}$ and in addition $V(F) \subset \P_{k(p)}(a)$ is smooth. {Furthermore, letting $v_i\in   \P_{k(p)}(a)$ be the point with $i$-th homogeneous coordinate 1 and all other coordinates 0,  we require that $F(v_i)\neq0$ if $a_i>1$. Finally, we require that the weights $a_i$ are pairwise relatively prime, each $a_i$ divides $e$, and $e$ is prime to the exponential characteristic of $k$.}
 \end{Def}
 
 \begin{Rem}The condition that each $a_i$  divides $e$ implies that $V(F)$ is a Cartier divisor on  $\P_{k(p )}(a)$. This being the case, the assumption that $V(F)\subset  \P_{k(p)}(a)$ is smooth implies that $V(F)$ does not contain any singular point of $\P_{k(p )}(a)$. If $n\ge2$, and if the $a_i$ are pairwise relatively prime, then $v_i$ is a singular point of $\P_{k(p)}(a)$ if $a_i>1$, so in case $n\ge2$, the last condition in the definition above is superfluous. 
 \end{Rem}

\subsection{The nearby cycles of a quasi-homogeneous singularity}

{As before, we fix a flat quasi-projective morphism $f: X \rightarrow B$ with $X/k_0$ and $X_\eta/\eta$ smooth, and with $B=\Spec\OO$ as  in Section~\ref{SectionNearbyCycles}.} We formulate conditions for the singularities in the quasi-homogeneous case, similar to Assumption \ref{assumption}.

\begin{Ass} \label{qassumption}
	\item[ (1)] The special fibre $X_\sigma$ has only isolated singularities $p_1, \ldots, p_r$.
	\item[ (2)] For each $p\in \{p_1, \ldots, p_r\}$ there is a polynomial  $F\in k(p)[T_0,\ldots, T_n]$ of weighted degree $e_p$ with respect to some weights  $a_*$, with $gcd(a_*)=1$ and $lcm(a_*)$ dividing $e_p$, such that $F$ defines a smooth {quotient} hypersurface in $\mathbb{P}_{k(p)}(a_*)$ (Definition~\ref{smhyper} above), and $(X_\sigma, p)$ looks like the weighted homogeneous singularity defined by $F$ (see Definition \ref{qdef}).
\end{Ass}

For later use we need the following fact:
\begin{Lemma} \label{lemma}
	Let $k$ be a field and let $Y$ be a $k$-scheme, separated and essentially of finite type over $k$. Let $D$ be an effective Cartier divisor on $Y$. Suppose that both $D$ and $Y\setminus D$ are smooth over $k$. Then $Y$ is smooth over $k$.
\end{Lemma}
\begin{proof}
	Since smoothness is invariant under field extensions we may assume $k$ is algebraically closed. Let $y$ be a point in $Y$. Since $D$ is a closed subscheme of $Y$, if $ y \notin D$ then it has a smooth neighbourhood. We have to show that also $y \in D$ is a smooth point in $Y$. {Since $D$ is an effective Cartier divisor, there is a neighbourhood $U$ of $y$ in $Y$, and a non-zero divisor $f$ on $U$ such that} $D \cap U$ is defined by the vanishing of  $f$. The exact sequence of sheaves  
	\[ 0 \rightarrow \OO_{U}(-D \cap U) \xrightarrow{\cdot f} \OO_{U} \rightarrow \OO_{D \cap U} \rightarrow 0\]  gives on stalks at $y$
	\[ 0 \rightarrow \OO_{Y,y} \rightarrow \OO_{Y,y} \rightarrow \OO_{D,y} \rightarrow 0 .\] 
	
	{Let } $d = \dim Y$ so $\dim D = d-1$. $D$ is smooth so $\OO_{D,y}$ is a regular local ring of dimension $d-1$, so we can write {the maximal ideal $m_{D,y}$ as generated by a regular sequence,}  $m_{D,y} = (\bar{f}_1,\ldots,\bar{f}_{d-1})$. The $\bar{f}_i$  lift to $f_1,\ldots,f_{d-1}$ in {$m_{Y,y}$}. Now since $ ker (\OO_{Y,y} \rightarrow \OO_{D,y}) = (f) \OO_{Y,y} $ we get from the exact sequence that $m_{Y,y} = (f,f_1,\ldots,f_{d-1})$, with $(f,f_1,\ldots,f_{d-1})$ a regular sequence. Then $\OO_{Y,y}$ is a regular local ring and hence $y$ is a smooth point of $Y$.  
\end{proof}

 Assuming that our only singularity is $p=p_1$, the main result of the section is an analogue of Theorem \ref{mainprop}. The assumptions in the statement of the theorem are meant to choose a convenient neighbourhood to work with, as it does not matter for the formulas deduced in Corollaries \ref{CorLocalqHomogFormula}, \ref{CorWtHomogFormula}. 
  
\begin{Th} \label{qmainprop}
	
	{Let  $f:X\to B$ be a  flat quasi-projective morphism
such that the generic fibre $X_\eta$ is smooth over $\eta$ and with $X$ smooth over $k_0$, satisfying Assumption \ref{qassumption}. Suppose in addition that $p\in X_\sigma$ is the only singular point of $X_\sigma$. Let $e=e_p$, let $F\in k(p)[T_0,\ldots, T_n]$ be as in Assumption \ref{qassumption} for $p$, with respect to weights $a_*$,  and let $\OO_e=\OO[t']/(t^{\prime e}-t)$. Finally, we assume that $X=\Spec A$ is affine with a system of \'etale coordinates $s_0,\ldots, s_n\in \mathfrak{m}_p$, and that the all the steps in Construction~\ref{muconstruction} can be carried out for $(X,p, F, s_*, a_*)$ without having to shrink $X$ to a smaller affine neighbourhood of $p$. } \\
Let $\pi : Z \to X \simeq Z/\mu_a $ be the $\mu_a$-quotient map given by Construction \ref{muconstruction} and {let $q\in Z$ be the unique point lying over $p$; note that $k(p)=k(q)$}. Let $\hat{Z} = Bl_q(Z)$ and let $Y_Z \rightarrow \hat{Z}$ be {the normalisation of the base-change $\hat{Z}_e:=\hat{Z}\times_{\Spec\OO}\Spec\OO_e$.} {Then the $\mu_a$-action on $Z$ extends to a $\mu_a$-action on $Y_Z$. Moreover, letting $Y : = Y_Z / \mu_a $ and letting $Y \rightarrow X$ be the resulting map on the quotients, we have }
{\item[ (1)] $Y$ is smooth over $k$ and $Y\to \Spec\OO_e$ is a semi-stable reduction of $X\to\Spec\OO$.}
	 {\item[ (2)] The special fibre $Y_{\sigma}\subset Y$ is a simple normal crossing divisor, of the form $Y_{\sigma}= \widetilde{D_{1}}+\widetilde{D_2}$ with $\widetilde{D_{1}}, \widetilde{D_{2}}$ smooth, and 
	\[
	\widetilde{D_1} \simeq V(F-T_{n+1}^{e}) \subset \mathbb{P}_{k(p)}(a,1)
	\]	
	\[
  \widetilde{D_{12}}:=\widetilde{D_{1}}\cap  \widetilde{D_2}\simeq V(F) \subset \mathbb{P}_{k(p )}(a).
	\]
Moreover,  the projection $q:\widetilde{D_2}\to X_\sigma$ is an isomorphism over $X_\sigma\setminus\{p\}$ and defines a resolution of singularities of $X_\sigma$, with $q^{-1}(p )=\widetilde{D_{12}}$.}	
\end{Th}

\begin{proof} 
{We may assume that $X$ is integral and we retain the notation from Construction~\ref{muconstruction}.}
{Let 
\[
A[s^{1/a}]= A[\sigma_0, \ldots , \sigma_n] /(\sigma_0^{a_0}-s_0,\ldots,\sigma_n^{a_n}-s_n).
\]
We have $Z= \Spec A[s^{1/a}]$, $Z$ is integral and is smooth over $k$, and we have a $\mu_a$-action on $Z$ with quotient $X$. Let $\pi : Z \rightarrow X=Z/\mu_a$ be the quotient map, induced by the inclusion  $A\hookrightarrow R := A[s^{1/a}]$.} Let $q \in Z$ be the unique point lying over $p \in X$. $(Z_\sigma, q)$ satisfies Assumption~\ref{assumption}, looking like a homogeneous singularity defined by $G(\sigma_0,..,\sigma_n):=F(\sigma_0^{a_0},\ldots,\sigma_n^{a_n})$ (see Construction \ref{muconstruction}). {$G$ has degree $e$ and} $V(G)$ is smooth by our assumption {on $F$}. We apply Theorem~\ref{mainprop} and construct the semi-stable reduction {$Y_Z\to \Spec\OO_e$ of $Z\to\Spec\OO$ by forming the blow-up $\hat{Z}:=Bl_qZ$, and letting $Y_Z$ be the normalisation of the base-change $\hat{Z}_e:=\hat{Z}\times_{\Spec\OO}\Spec\OO_e$. }

	{Since the $\mu_a$-action on $Z$ fixes $q$, this action lifts canonically to an action on $\hat{Z}$, which gives a $\mu_a$-action on $\hat{Z}_e$ over $\Spec\OO_e$ and finally induces a $\mu_a$-action on the normalisation $Y_Z$. Let $Y:=Y_Z/\mu_{a}$ and let $\pi:Y_Z \rightarrow Y$ denote the quotient map. Since $Y_Z\to Z_e$ is proper, it follows that the induced map on the quotients $Y\to X_e$ is also proper.} 
	
	{Let $E_1\subset \hat{Z}$ be the exceptional divisor, let $E_2\subset\hat{Z}$ be the strict transform of $Z_\sigma$ and let ${E_{12}}= E_1 \cap E_2$.}	Denote by $\widetilde{E_1}$, $\widetilde{E_2}$, $\widetilde{E_{12}}$ their {respective} coverings in $(Y_Z)_\sigma $, as in the proof of Theorem~\ref{mainprop}.  Let $\widetilde{D_i }:=\pi(\widetilde{E_i})=\widetilde{E_i}/\mu_a\subset Y_\sigma$. 
		{Since $Z$ is integral, it follows from Remark~\ref{RemarkIntegral} that  $Y_Z$ is integral and thus the quotient scheme $Y=Y_Z/\mu_a$ is integral as well.} 
	We use the standard presentation of the blow-up $\hat{Z}$ as
\[
\hat{Z}=\Proj A[s^{1/a}][Z_0,\ldots, Z_n]/(\{\sigma_iZ_j-\sigma_jZ_i\}_{0\le i,j\le n}) 
\]
giving the standard open cover of $\hat{Z}$ by the affine open subsets  $Z_i\neq0$. This induces the affine open cover $\{V_0,\ldots, V_n\}$  of $Y_{Z}$. As in the proof of Theorem \ref{mainprop}, we have the explicit description of the $V_i$, for instance, 
\[
	V_0 = \Spec( R_e[z_1,\ldots,z_n, z_{n+1} ] /( \{\sigma_i-\sigma_0  z_i\}_{1\le i\le n+1}  , g_0 - z_{n+1}^e )) 
	\]
	with  $R_e:=R\otimes_\OO\OO_e$, $z_i=Z_i / Z_0 $ for $i=1,\ldots,n$, $z_{n+1}=t' / \sigma_0 $ and $g_0=G(1,z_1,\ldots, z_n)+\sigma_0h'$ for suitable $h'$. Letting $A_e:=A\otimes_\OO\OO_e$, we can rewrite this as
\[
V_0 = \Spec( A_e [\sigma_0, z_1,\ldots,z_n, z_{n+1} ] /(\{s_i-\sigma^{a_i}_0  z^{a_i}_i\}_{1\le i\le n+1}  , g_0 - z_{n+1}^e, s_0-\sigma_0^{a_0} )).
	\]
{Again referring to  Theorem \ref{mainprop} and its proof, we have the global description of $\widetilde{E_1}$ as the closed subscheme $V(G(Z_0,\ldots, Z_n)-Z_{n+1}^e)$ of $\mathbb{P}^{n+1}_{k(q)}:=\Proj k(q)[Z_0,\ldots, Z_{n+1}]$, with $\widetilde{E_{12}}\subset \widetilde{E_1}$ defined by $Z_{n+1}=0$.  Finally, the projection $Y_Z\to Z$ restricts to a morphism $\pi_2:\widetilde{E_2}\to Z_\sigma$, $\pi_2$ is an isomorphism over $Z_\sigma\setminus\{q\}$ and the reduced inverse image $\pi_2^{-1}(q)$ is $\widetilde{E_{12}}$.}
	
Taking the $\mu_a$-quotients $U_i:=V_i/\mu_a$ gives the affine open cover $\{U_0,\ldots, U_n\}$ of $Y$.	Let us now describe the $\mu$-action {on $\hat{Z}_e$ and on $V_0$.} {For $\zeta\in \mu_{a_i}$, and $j=0,\ldots, n$, we have
\[
\zeta\cdot Z_j=\begin{cases} \zeta Z_i&\text{ for }j=i\\ Z_j&\text{ for } j\neq i \end{cases}
\]
and 
\[
\zeta\cdot \sigma_j=\begin{cases} \zeta \sigma_i&\text{ for }j=i\\ \sigma_j&\text{ for }  j\neq i.\end{cases}
\]
On the affine piece $V_0$,  and for  $\zeta\in \mu_{a_i}$, $i=1,\ldots, n$ and for $j=1,\ldots, n+1$, we thus have
\[
\zeta\cdot z_j=\begin{cases}\zeta z_i&\text{ for }j=i\\z_j&\text{ for } j\neq i,\end{cases}
\]
and $\zeta\cdot\sigma_0=\sigma_0$.}
{For $\zeta\in \mu_{a_0}$, we have $\zeta\cdot \sigma_0 =\zeta\sigma_0$ and
\[
\zeta\cdot z_j=\zeta^{-1}z_j 
\]
for all $j=1,\ldots, n+1$.
The $\mu_a$-action on the other open subschemes $V_i$ is defined similarly. We also have a global description of the $\mu_a$-action on $\widetilde{E_1}\subset \P^{n+1}_{k(q)}=\Proj k(q)[Z_0,\ldots, Z_{n+1}]$ by having $\mu_a$ act trivially on $Z_{n+1}$; one can easily check that this restricts to the action on each $V_i\cap \widetilde{E_1}$ defined above.
\\We have the following commutative diagram to our aid,
\[
\begin{tikzcd} [column sep = tiny, row sep = small]
	V_i \cap \widetilde{E_1}  \arrow[dd, hook]  \arrow[rd, hook]   \arrow[rr] & & U_i \cap \widetilde{D_1} \arrow[rd, hook] \arrow[dd, hook]&  \\
	&	\widetilde{E_1}  \arrow[rr, crossing over]  & &  \widetilde{D_1}   \arrow[dd, hook, ""] \\
	V_i \arrow[rd, hook] \arrow[rr] & & U_i \arrow[rd, hook] & \\
	&  Y_Z \arrow[rr] \arrow[from=uu, hook, crossing over]  \arrow[d, ""] & &  Y  \arrow[dd, ""] \\
	&	\hat{Z}_e  \arrow[d, ""] & &   \\
	&	Z_e \ar[ddr]\ar[rr]		&	& X_e\ar[ddr] \ar[d] & & \\
	& &	\hat{Z}  \arrow[d, ""] \arrow[from=uul, crossing over] &  B_e \ar[ddr] & \\
	& &	Z \arrow[rr, "\pi", crossing over] & & X \ar[d] \\
	& & & & B.
\end{tikzcd}
\] 
	We {will} describe the quotient by {$\mu_a$ in} two steps - first {taking the} quotient by the subgroup $\mu_{a_{>0}}:= \mu_{a_1} \times \ldots \times \mu_{a_n} $ and then by the remaining {factor} $\mu_{a_0} $.
	
{\textbf{Proof of (1).} The assertion (1) is local on $Y$, so it suffices to prove (1) after restricting to $U_i\subset Y$; we give the proof for $U_0$. We assume at first that $a_0>1$; the case $a_0=1$ is easier and will be dealt with at the end of the argument.} {Let
\[
C_0=A_e [\sigma_0, z_1,\ldots,z_n, z_{n+1} ] /( (s_i-\sigma^{a_i}_0  z^{a_i}_i)_{1\le i\le n+1}  , g_0 - (z_{n+1})^e, s_0-\sigma_0^{a_0} )
\]
and let $C\subset C_1\subset C_0$ be the rings of invariants
\[
C_1=C_0^{\mu_{a>0}}, C=C_0^{\mu_a}=C_1^{\mu_{a_0}},
\]
so $V_0=\Spec C_0$ and $U_0=\Spec C\subset Y$.  Since $V_0$ is smooth over $k$ and is integral, the invariant subrings $C, C_1$ are both integral and normal.}
\[
\begin{tikzcd} [column sep = normal, row sep = small]
	V_0 = \Spec C_0 \ar[r] \ar[d, hook] & V_0 / \mu_{a>0}  = \Spec C_1 \ar[r] & U_0 = \Spec C  \ar[d, hook] \\
	Y_Z \ar[rr] & & Y = Y_Z / \mu_a
\end{tikzcd}.
\] 
{We have $s_0\in C$ and $\sigma_0\in C_1$. We first show that $C[s_0^{-1}]$  is a smooth $\OO_e$-algebra. To see this, note that the special fibre $X_\sigma$ has only $p$ as singular point, so $A[s_0^{-1}]$ is a smooth $\OO$-algebra. Thus the base extension $A_e[s_0^{-1}]=A[s_0^{-1}]\otimes_\OO\OO_e$ is a smooth $\OO_e$-algebra. Moreover, since localization commutes with taking invariants,  $A_e[s_0^{-1}]$ is the $\mu_a$-invariants in $R_e[\sigma_0^{-1}]$, and since $\sigma_0$ defines $\widetilde{E_1}\cap V_0$ in $V_0$,
 $V_0\to \Spec R_e$ is an isomorphism over $\Spec R_e[\sigma_0^{-1}]$. This shows that 
$C[s_0^{-1}]=A_e[s_0^{-1}]$ and hence $C[s_0^{-1}]$ is a smooth $\OO_e$-algebra. \\ }
{The $\mu_{a>0}$-invariant subring of $A_e[\sigma_0, z_1,\ldots,z_n, z_{n+1} ]/(s_0-\sigma_0^{a_0})$ is
\[
[A_e[\sigma_0, z_1,\ldots,z_n, z_{n+1} ]/s_0-\sigma_0^{a_0})]^{\mu_{a>0}}=
A_e[\sigma_0, t_1,\ldots,t_n, z_{n+1} ]/(s_0-\sigma_0^{a_0}),
\]
with $t_i=z_i^{a_i}$. From this it follows that 
\[
C_1:=C_0^{\mu_{a>0}}=
A_e[\sigma_0, t_1,\ldots,t_n, z_{n+1} ]/(\{s_i-\sigma^{a_i}_0t_i\}_{i=1,\ldots, n}, f_0-z_{n+1}^e, s_0-\sigma_0^{a_0} ),
\]
where $f_0=F(1, t_1,\ldots, t_n)+\sigma_0\cdot h$ for a suitable $h$. Note that $\mu_{a_0}$ now acts by $\zeta \cdot t_i=\zeta^{-a_i}t_i$. }

{Our assumption that  $F$ defines a smooth quotient hypersurface in $\P(a)$ and our assumption $a_0>1$ implies that   $F(1,0,\ldots, 0)\neq 0$, that is
\[
\emptyset=V(\sigma_0, t_1,\ldots,t_n, z_{n+1})\cap V( f_0-z_{n+1}^e)\subset 
\Spec A_e[\sigma_0, t_1,\ldots,t_n, z_{n+1} ]/(s_0-\sigma_0^{a_0}, \{s_i-\sigma^{a_i}_0t_i\}_{i=1,\ldots, n}).
\]
The $\mu_{a_0}$-action on $\Spec A_e[\sigma_0, t_1,\ldots,t_n, z_{n+1} ]/(s_0-\sigma_0^{a_0}, \{s_i-\sigma^{a_i}_0t_i\}_{i=1,\ldots, n})$ is free outside the origin \\$V(\sigma_0, t_1,\ldots,t_n, z_{n+1})$.
Thus the $\mu_{a_0}$-action on $\Spec C_1$ is free and hence the ring extension $C\hookrightarrow C_1$ is \'etale. In particular,  $C_1[\sigma_0^{-1}]=C_1[s_0^{-1}]$ is \'etale over the smooth $k$-algebra $C[s_0^{-1}]$ and hence $C_1[\sigma_0^{-1}]$ is a smooth $k$-algebra.}

Since $\sigma_0$ is $\mu_{a>0}$-invariant, it follows that $(\sigma_0)C_1$ is the $\mu_{a>0}$-invariants in $(\sigma_0)C_0$, in other words
\[
(\sigma_0)C_1=C_1\cap (\sigma_0)C_0.
\]
This implies that the evident ring homomorphism $C_1/(\sigma_0)\to C_0/(\sigma_0)$ is injective and since $e$ is prime to the characteristic of $k$, taking $\mu_{a>0}$ invariants is an exact functor, and thus 
\[
C_1/(\sigma_0)= [C_0/(\sigma_0)]^{\mu_{a>0}}.
\]
Explicitly,
\[
C_0/(\sigma_0)= 
k(p )[z_1,\ldots,z_n, z_{n+1} ] /(G(1,z_1,\ldots, z_n) - z_{n+1}^e)
\]
Since $G(1,z_1,\ldots, z_n)=F(1, z_1^{a_1},\ldots, z_n^{a_n})$, $G(1,z_1,\ldots, z_n) - z_{n+1}^e$ is $\mu_{a>0}$ invariant, so as above, we have
\begin{align*}
C_1/(\sigma_0)=&[k(p )[z_1,\ldots,z_n, z_{n+1} ] /( G(1,z_1,\ldots, z_n) - z_{n+1}^e)]^{\mu_{a>0}}\\=&
k(p )[t_1,\ldots, t_n, z_{n+1}]/(F(1, t_1,\ldots, t_n)-z_{n+1}^e).
\end{align*}
Using again our smoothness assumption on $F$, we see that  $C_1/(\sigma_0)$ is a smooth $k$-algebra. By Lemma~\ref{lemma}, $C_1$ itself is a smooth $k$-algebra and since $C\hookrightarrow C_1$ is \'etale, $C$ is also a smooth $k$-algebra. }

\sloppy Similarly, to see that $V_0\to \Spec\OO_e$ is a semi-stable reduction, it suffices to see that the special fibre $\Spec C_1/(t')C_1$ is a simple normal crossing divisor on $\Spec C_1$. For this, we have $t'=\sigma_0z_{n+1}$.  We have already seen that
 $C_1/(\sigma_0)$ is a smooth $k$-algebra, in other words, the Cartier divisor $V(\sigma_0)$ on $\Spec C_1$ is smooth.  We have
 \[
 C_1/(z_{n+1}, \sigma_0)= k(p )[t_1,\ldots, t_n]/(F(1, t_1,\ldots, t_n))
 \]
 which again by our assumption on $F$ is a smooth $k$-algebra. This implies that the Cartier divisors $V(\sigma_0), V(z_{n+1})\subset \Spec C_1$ intersect transversely on $\Spec C_1$, which implies that $V(z_{n+1})$ is smooth in a neighbourhood of $V(\sigma_0)$ in $\Spec C_1$; this also implies that $(t')=(\sigma_0)\cap(z_{n+1})$. We have also shown that $C[s_0^{-1}]$ is smooth over $\OO_e$, which implies that $C_1[\sigma_0^{-1}]$ is also smooth over $\OO_e$, so $V(z_{n+1})\setminus V(\sigma_0)$ is smooth. Thus the Cartier divisor $V(t')$ on $\Spec C_1$ is $V(\sigma_0)+V(z_{n+})$, which we have just shown is a normal crossing divisor. This completes the proof of (1), and also shows that $Y_\sigma$ is a union of two smooth components, intersecting transversely, proving the first part of (2).
 
In case $a_0=1$, we have $C=C_1$ and a much simpler version of the arguments given above takes care of this case.

 \textbf{Proof of (2).} We have just shown that $Y_{\sigma}$ is the Cartier divisor   $\widetilde{D_1}+\widetilde{D_2}$, with  $\widetilde{D_1}, \widetilde{D_2}$ both smooth and with transverse intersection $\widetilde{D_{12}}$.   We have the global description of $\widetilde{E_1}$ given by Theorem~\ref{mainprop}, namely $\widetilde{E_1}$ is the closed subscheme 
 $V(G(Z_0,\ldots, Z_n)-Z_{n+1}^e)$ of $\P^{n+1}_{k(q)}$. We have
 \[
 \widetilde{D_1}=\widetilde{E_1}/\mu_a.
 \]
 The $\mu_a$-action on $\widetilde{E_1}$ extends to an action on $\P^{n+1}_{k(q)}=\Proj k(q)[Z_0,\ldots, Z_n, Z_{n+1}]$ as described in the proof of (1) by having $\mu_a$ act trivially on $Z_{n+1}$. Then
 \[
 \P^{n+1}_{k(q)}/\mu_a=\P(a_0,\ldots, a_n,1).
 \]
 Let $a_{n+1}=1$, let $T_i=Z_i^{a_i}$, $i=0,\ldots,  n+1$, and let $K \subset  \P(a_0,\ldots, a_n,1)$ be the hypersurface $V(F-T_{n+1}^e)$. We wish to identify $\widetilde{D_1}$ with $K$. Let $W_i\subset \P(a_0,\ldots, a_n,1)$ be the open subscheme $T_i\neq0$. Giving $T_j$ weight $a_j$, we have
 \[
 W_i=\Spec k(p )[T_0,\ldots, T_{n+1}][T_i^{-1}]_0 .
 \]
 We concentrate on the case $i=0$ to simplify the notation. In the diagram
 		\[
 \begin{tikzcd}[row sep=huge ]
 	\widetilde{E_1} \cap V_0 = V_{ \mathcal{W}_0}(G(1,z_1,...,z_n) - z_{n+1}^e) \arrow[r, hook ] \arrow[d,  "/ \mu_{a_{>0}}"]  &   \mathcal{W}_0  = \Spec k(p)[z_1,...,z_{n+1}] \arrow[d, ""] \arrow[r, ] &  \P^{n+1} \arrow[dd, "/ \mu_a"] \\
 	V_{\bar{W}_0}(F(1,t_1,...,t_n) - z_{n+1}^e ) \arrow[r, hook]  \arrow[d, "/ \mu_{a_{0}}"] &  \bar{W}_0  = \Spec k(p)[t_1,...,t_n,z_{n+1}]  \arrow[d]  &  \\
 	\widetilde{D_{1}} \cap U_0  = V_{W_0}(F(1,t_1,...,t_n) - z_{n+1}^{e}) \arrow[r, hook]  & {W}_0 \cong \Spec k(p)[t_1,...,t_n,z_{n+1}]^{\mu_{a_0}} \arrow[r] &  \P(a,1)  \end{tikzcd}	\] 
 The first row describes the restriction of the embedding of $\widetilde{E_1}$ in $\P^{n+1}$ to the affine $\mathcal{W}_0$, as in the proof of Theorem \ref{mainprop}. The objects in the rest of the diagram are defined and discussed below.
  Let $S_0:= k(p )[T_0,\ldots, T_{n+1}][T_0^{-1}]_0$ and let $S_0':=[k(p )[t_1,\ldots, t_n, z_{n+1}]]^{\mu_{a_0}}$, with the $\mu_{a_0}$ action as defined in the proof of (1). A direct computation shows that $S_0\cong S_0'$.   Indeed, a monomial $\prod_jt_j^{b_j}\cdot z_{n+1}^{b_{n+1}}$ is $\mu_{a_0}$-invariant if and only if $\sum_{j=1}^{n+1}a_jb_j$  is divisible by $a_0$. Similarly, a monomial $\prod_{j=1}^{n+1}T_j^{b_j}\cdot T_0^{-b_0}$ has weighted degree zero if and only if $\sum_{j\ge1}a_jb_j=a_0b_0$. So, sending $\prod_{j=1}^{n+1}T_j^{b_j}\cdot T_0^{-b_0}$ to 
 $\prod_jt_j^{b_j}\cdot z_{n+1}^{b_{n+1}}$ gives an isomorphism of $S_0$ with $S_0'$.
 
 Similarly, recalling that $a_0$ divides $e$,  the weighted homogeneous polynomial $F(T_0,\ldots, T_n)-T_{n+1}^e$ gives the element $F(T_0,\ldots, T_n)/T_0^{e/a_0}-T_{n+1}^e/T_0^{e/a_0}$ in $S_0$, which corresponds to the element $F(1, t_1,\ldots, t_n)-z_{n+1}^e$ of 
 $[k(p )[t_1,\ldots, t_n, z_{n+1}]]^{\mu_{a_0}}$.
 
Let $\bar{W}_0:=\Spec k(p )[t_1,\ldots, t_n, z_{n+1}]$. The finite extension 
\[
[k(p )[t_1,\ldots, t_n, z_{n+1}]]^{\mu_{a_0}}\to k(p )[t_1,\ldots, t_n, z_{n+1}] 
\]
defines a finite morphism $p:\bar{W}_0\to W_0$. By our computations in the proof of (1) and that given in the previous paragraph, we see that
\[
p^{-1}(K\cap W_0)=V(F(1, t_1,\ldots, t_n)-z_{n+1}^e)=\Spec C_1/(\sigma_0)=(\widetilde{E_1}\cap V_0)/\mu_{a>0},
\]
and thus 
\[
K\cap W_0=(\widetilde{E_1}\cap V_0)/\mu_a=\widetilde{D_1}\cap W_0 .
\]
An analogous computation shows that $K\cap W_i=\widetilde{D_1}\cap W_i$ for $i=1,\ldots, n+1$, so $\widetilde{D_1}=K=V(F-T_{n+1}^e)$, as desired. 

A similar argument shows that $\widetilde{D_{12}}=V(F-T_{n+1}^e)\cap V(T_{n+1})$, in other words, $\widetilde{D_{12}}=V(F)\subset \P(a)$. 

In the proof of (1), we showed that the projection $U_0\setminus V(s_0)\to X_e\setminus V(s_0)$ is an isomorphism; a similar argument shows that  $U_i\setminus V(s_i)\to X_e\setminus V(s_i)$ is an isomorphism for all $i$. This shows that $Y\setminus \widetilde{D_1}\to X_e\setminus \{p\}$ is an isomorphism. Passing to the fibre over the closed point of $\Spec\OO_e$, it follows that $\widetilde{D_2}\setminus\widetilde{D_{12}}\to X_\sigma\setminus\{p\}$ is an isomorphism. Since $\widetilde{D_2}$ is smooth, $\widetilde{D_2}\setminus\widetilde{D_{12}}$ is dense in $\widetilde{D_2}$ and $\widetilde{D_2}\to X_\sigma$ is proper, we see that $q:\widetilde{D_2}\to X_\sigma$ is a resolution of singularities of $X_\sigma$, with $q^{-1}(p )=\widetilde{D_{12}}$, finishing the proof of (2). 
\end{proof}

\begin{Cor}\label{CorLocalqHomogFormula}
	Let $f:X\to \Spec\OO$ be a flat quasi-projective morphism with $X$ smooth over $k_0$ and with $X_\eta$ smooth over $\eta$. Suppose that the special fibre $X_\sigma$ has an isolated singular point $p$, at which $f$ looks at like the weighted homogeneous singularity defined  $F\in k(p)[T_0,\ldots, T_n]$ of weighted degree $e$ (Definition~\ref{qdef}),  with $V(F)\subset \P_{k(p)}(a_*)$ a smooth quotient hypersurface satisfying assumption~\ref{qassumption} (2). Assuming $\Psi_f \one$ is dualisable (e.g. in zero characteristic), we have the formula
	\[
	\chic(\Psi_f({\one_X}_{\eta})|_p)= \chic(V(F-T_{n+1}^{e}))-
	\<-1\> \chic(V(F)).
	\]
\end{Cor}

\begin{proof}
	At the singular point $p_i$ of $X_\sigma$, by Theorem~\ref{qmainprop}, there is an affine open neighbourhood $U$ of $p$ in $X$ such that the restriction $f_U:U\to \Spec\OO$ admits a semi-stable reduction $Y\to \Spec\OO_e$, with special fibre $Y_\sigma=\widetilde{D_1}+\widetilde{D_2}$, with $\widetilde{D_1}\cong V(F - T_{n+1}^e)$ the preimage of the point $p$, and with $\widetilde{D_{12}}=V(F)$. By Corollary \ref{semistablered}
	$
	\chic(\Psi_f \one_{X_{\eta}} |_p) = \chic (\Psi_g \one_{Y_\eta}|_{\widetilde{D_{1}}}),
	$
	and we get the result by applying Proposition~\ref{example2} and Theorem~\ref{qmainprop}.
\end{proof}

\begin{Rem} \label{qRem}
	 This formula can be viewed as an extension of the formula appearing in Corollary \ref{CorLocalHomogFormula} (and hence of Proposition \ref{AIS} in the case considered), as the expressions for the coverings $\widetilde{D_{1}}$, $\widetilde{D_{12}}$ are the same as the Denef-Loeser covers considered in those previous propositions.
\end{Rem}

\begin{Cor} \label{CorWtHomogFormula} Let $f:X\to \Spec\OO$ be a quasi-projective flat morphism with $X$ smooth over $k_0$. 
Suppose that $X_\eta$ is smooth over $\eta$ and the special fibre $X_\sigma$ is satisfying Assumption~\ref{qassumption} with singular points $p_1,\ldots, p_r$; for each $i$   $(X_\sigma, p_i)$ looks like the weighted homogeneous singularity defined by a weighted homogeneous polynomial $F_i$ of weighted degree $e_i$ for weights $a_*^{(i)}$ (Definition~\ref{qdef}); assume $\Psi_f \one$ is dualisable. Let $X_\sigma^\circ=X_\sigma\setminus\{p_1,\ldots, p_r\}$. Then
\[
\chic(\Psi_f({\one_X}_{\eta}))=\chic(X_\sigma^\circ)+ \sum_{i=1}^r \chic(V(F_i-T_{n+1}^{e_i}))-
\<-1\> \sum_{i=1}^r \chic( V(F_i)).
\]
\end{Cor}

\begin{proof} Immediate from Proposition~\ref{local} and the previous Corollary.
\end{proof}

\section{Comparison of local Euler classes}

{In this section we discuss a motivic local invariant, \emph{the $\Aa^1$-local Euler class}, as it is defined in \cite{Le20a} and \cite{BW}, that gives an effective tool for computing the quadratic Euler characteristic of the nearby cycles. We will show that, for the type of morphism $f:X\to\Spec\OO$ that we have been considering, when $f$ looks at a point $p\in X_\sigma$ like a (weighted) homogeneous singularity defined by a (weighted) homogeneous polynomial $F(T_0,\ldots, T_n)$, the local Euler class at $p$ for $df$ is the same as the  local Euler class for the map $F:\Aa^{n+1}\to \Aa^1$ at the origin $0\in \Aa^{n+1}$ (see Definition~\ref{HypersurfDef} and Corollary~\ref{Eulerclass} for a precise statement).}

\subsection{The local Euler class}
 
 We recall here some preliminary definitions and define $\Aa^1$-local Euler class with respect to a section of a vector bundle following \cite[5.1]{BW}.
 
\begin{Def}
{ For a vector bundle $p: V \to X$ with zero section $s_0 : X \to V$, and dual bundle $V^*$, define the functor $ \Sigma^{V^*} : \SH(X) \to \SH(X) $ by $\Sigma^{V^*} := p_\# s_{0*}$.}
 
{We have the identity} $\Sigma^{V^*} \one_X  = V / ( V \setminus X )\in \SH(X)$, see \cite[5.2]{Hoy}.

\end{Def}

 \begin{Def}
 	Let $S$ be a scheme, let $E \in \SH(S)$, $f: X \to S$ an $S$-scheme, $i : Z \hookrightarrow X $ a closed subscheme, and $p: V \rightarrow X$ a vector bundle.
 	We define the $V$-twisted $E$-cohomology of $X$ with support on $Z$, which we denote $E_Z^V(X)$, to be
 	\[ E_Z^V(X) = [\one_S, f_* i_! \Sigma^{i^*V} i^! f^* E]_{\SH(S)} \simeq [X / (X \setminus Z) , \Sigma^{V} f^* E]_{\SH(X)} ; \]
 	see \cite[4.2.1]{BW}. \\
 	When $Z=X$, we drop $Z$ from the notation. We also denote $E^n_Z(X) = E_Z^{\OO_X^{\otimes n}}(X)$. \\
 	For $\mathcal{L}$ a line bundle over $X$, we put $E^n_Z(X, \mathcal{L})=E^{n-1 + \mathcal{L}}_Z(X)$.
 \end{Def}

\begin{Def} 
	Let $E \in \SH(S) $ be a motivic ring spectrum. We denote by $(V, \rho)$ pairs consisting of a vector bundle $p:  V \to X$ and an isomorphism $ \rho : det V \xrightarrow{\sim} \OO_X $ . \\
		 An \textit{$SL$-orientation} on $E$ is an assignment of an element $th(V, \rho) \in E_0^{p^*V^*}(V) $ for each such pair $(V, \rho)$, satisfying some axioms as in \cite[Definition 3.4]{LR}.	An \textit{$SL$-oriented ring spectrum} $E$ is a motivic ring spectrum $E \in \SH(S)$ together with a given $SL$-orientation $th(-,-)$. If $E$ is an $SL$-oriented motivic spectrum, and $p: V \to X$ is a vector bundle of rank $n$, we have $E^V_Z(X) = E^n(X, det V)$. \\
		Let $k$ be a field. The motivic ring spectrum that we commonly use in this paper is $E = H\mathcal{K}^{MW} \in \SH(k) $, the Eilenberg MacLane spectrum representing the Milnor-Witt homotopy module $\mathcal{K}^{MW}_*$; it admits a canonical $SL$-orientation. For details on the construction of this motivic spectrum and its $SL$-orientation see \cite[Section 3]{Le20a}; for the definition of $\mathcal{K}^{MW}$ see \cite[Section 6]{Mo}. \\
	Let $X$ be a smooth scheme over a perfect field $k$ and $p \in X$ a closed point. Then we have an isomorphism $ (H\mathcal{K}^{MW} )^n (X, \omega_{X/k})  \simeq \GW(k(p)) $ (\cite[Cor. 3.3]{Le20a}), so classes in cohomology groups defined by this motivic ring spectrum can be expressed by quadratic forms.
	 We also use the notation $H^n_Z(X, \mathcal{K}^{MW}(\mathcal{L}))$ for the group $(H\mathcal{K}^{MW} )^n_Z(X, \mathcal{L}) $.
\end{Def}
 
 \begin{Def} \label{LocalEulerClass}
 Let $V \rightarrow X$ be a vector bundle of rank n, $s: X \rightarrow V$ a section and $i: Z  =Z(s) \hookrightarrow X $ the zero locus of $s$. The local Euler class of $(V,s)$, also called the refined Euler class, is the element $e(V,s) \in E_Z^{V^*}(X)$  defined by the composition
 \[ X/X \setminus Z \xrightarrow{s} V/V\setminus 0 \simeq \Sigma^{V^*} \one_k \rightarrow \Sigma^{V^*} E|_X \in \SH(X) . \]  
 \end{Def}

\begin{Rem} \label{LocalEulerClassRem}
 
In the case of an $SL$-oriented theory $E$, and a rank $n$ bundle $V$, we have $E_Z^{V^*}(X)=E_Z^n(X,  det^{-1}V) $, giving the local Euler class $e(V,s)\in E_Z^{V^*}(X)=E_Z^n(X,  det^{-1}V) $. \\
{We also have the {\em Thom class} $th(V) \in E_{0_V}^{p*V^*}(V)$, defined as the local Euler class $e(t, p^*V)$, where $t:V\to p^*V$ is the tautological section (with zero-locus the zero-section in $V$).} In that case, 
\[ e_Z(V,s) = s^* th(V) \in E_Z^{V}(X), \]
  see \cite[Def. 5.12]{BW}.
\end{Rem} 
 
\begin{Ex} \label{LocaEulerClassEx}
	
 In the case {the section $s$ has as the zero locus $Z$ a single point $p$, then for}  $E = H\mathcal{K}^{MW}  $, $V= \Omega_{X/k} $, we have  $e_p(\Omega_{X/k}, s) \in H\mathcal{K}^{MW}_p(X, \omega_{X/k})$. By the purity isomorphism for $H\mathcal{K}^{MW}$, this latter group is canonically isomorphic to $\GW(k(p))$. This element can be computed as the Scheja-Storch quadratic form on the Jacobian ring at the point, see \cite[Cor. 3.3]{Le20a} and below \ref{SSEulerClass}.
  \end{Ex}
 
\begin{Def}
		Let $X\in \Sm_k$ and let $\Omega_{X/k}$ be the sheaf of K\"ahler differentials. Let $f:X\to \Spec\OO$ be a flat morphism with an isolated critical point $p\in X_\sigma$, so the section $df\in H^0(X, \Omega_{X/k})$ has zero locus $Z(s)=\{p\}$ in a neighbourhood of $p$. We define the \emph{quadratic Milnor number at $p$} by
 \[
 \mu^q_{f,p} : = e_p(\Omega_{X/k}, s) \in \GW(k(p)).
 \] 
 \end{Def}

\subsection{Comparing Euler classes}

\begin{Def}\label{HypersurfDef} Let $\kappa$ be a field, let $a_*=(a_0,\ldots, a_n)$ be a sequence of weights and let $F(T_0,\ldots, T_n)\in \kappa[T_0,\ldots, T_n]$ be an $a_*$-weighted homogeneous polynomial of weighted degree $e$. Let $\OO_\kappa=\kappa[t]_{(t)}$; we denote the closed point of $\Spec\OO_\kappa$ by $\sigma_\kappa$ and the generic point by $\eta_\kappa$. \\
We assume that the $a_i$ are pairwise relatively prime, that $a_i$ divides $e$ for all $i$ and that $V(F)\subset \P_\kappa(a_*)$ is a smooth quotient hypersurface; in particular $e$ is prime to the exponential characteristic of $\kappa$. \\
Define $H^F\subset \P_{\OO_\kappa}(a_*,1)$ to be the hypersurface 
$V(F-tT_{n+1}^e)$, and let $f_F:H^F\to \Spec\OO_\kappa$ denote the projection. 
\end{Def}
One can see that $H^F$ is smooth over $\kappa$, the generic fibre $H^F_{\eta_\kappa}$ is smooth over $\eta=\Spec\kappa(t)$ and the special fibre $H^F_{\sigma_\kappa}$ has a single isolated singular point $0:=(0:\ldots:0:1)$.

{We return to our main object of study, a quasi-projective flat map $f:X\to \Spec\OO$ with an isolated critical point $p\in X$. Our goal is to show that, under the assumption that $f$ looks near $p$ like a quasi-homogeneous singularity defined by a polynomial $F\in k(p )[T_0,\ldots, T_n]$,  the local Euler class $e_p(\Omega_{X/k_0}, df)$ at the critical point $p\in X$ is equal to the local Euler class $e_0(\Omega_{H_F/k(p )}, df_F)$. By $df$ we mean the section $d(f^*(t))$ of $\Omega_{X/k_0}$, and similarly for $df_F$.  We first make some elementary simplifications.}

{First of all, due to the Nisnevich descent enjoyed by all cohomology theories defined by motivic spectra, the local Euler class $e_p(\Omega_{X/k_0}, df)\in \GW(k(p ))$ is unchanged if we replace $(X,p)$ by a Nisnevich neighbourhood $(X',p)\to (X,p)$, and also depends only on $df$ restricted to $\Spec\OO_{X,p}$. Thus, we may replace $X$ with $\Spec\OO_{X,p}$, and, changing notation, assume that $X=\Spec\OO_{X,p}$ is local. Similarly, we may assume that the local ring $\OO_{X,p}$ contains its residue field $k(p )$; changing notation, we may assume that $k(p )=k$. The special fibre $X_\sigma$ is just the subscheme of $X$ defined by $f\in \OO_{X,p}$, so we may replace $f : X \to \OO$ be the  morphism $f:X\to \Spec k[t]$ given by the $k$-algebra homomorphism $t\mapsto f^*(t)$. Choosing a system of parameters $s_0,\ldots, s_n$ so that $f^*(t)=F(s_0,\ldots, s_n)+h$ as in Assumption~\ref{qassumption}, we have the morphism over $\Spec k[\lambda]$,
\[
f_\lambda:X\times\Spec k[\lambda]\to \Spec k[t,\lambda],
\]
defined by $f_\lambda^*(t):=F(s_0,\ldots, s_n)+\lambda\cdot h$. }

\begin{Prop} \label{Euler1}
Let $X = \Spec \OO_{X,p}$ be a flat $\Aa^1$-family $f:X \to \Spec k[t]$, satisfying Assumption \ref{assumption} (homogeneous case) or \ref{qassumption} (quasi-homogeneous case). Define $\mathcal{X} = \Spec{\OO_{X,p}}[\lambda] = X \times \Aa^1$. Let $f_\lambda : X \times \Aa^1 \rightarrow \Aa^1 \times \Aa^1$ be defined as above,
and let $\mathcal{X}_\sigma = f_\lambda^{-1} (0 \times \Aa^1)$ with induced morphism $(f_\lambda)_\sigma: \mathcal{X}_\sigma \to \Aa^1$. Then there exists an open neighbourhood $U \supset p \times \Aa^1 $ in $\mathcal{X}_\sigma$, such that $U \setminus (p \times \Aa^1 )  $ is smooth over $\Aa^1$. 
\end{Prop}

\begin{proof}
	{We start with the homogeneous case. Let $\rho: Bl_{p \times \Aa^1} \mathcal{X} \rightarrow  \mathcal{X}$ be the blow-up of $ \mathcal{X}$ along $p \times \Aa^1 $. }  Let $ q: \bigcup U_i = Bl_{p \times \Aa^1} \mathcal{X} \rightarrow X $ be the standard covering and denote by $\mathcal{D}_{12}$ the intersection of the strict transform of $\mathcal{X}_\sigma$ and the exceptional divisor $\mathcal{D}_1\subset Bl_{p \times \Aa^1} \mathcal{X}$. 
	
	We can describe the morphism $\mathcal{D}_{12}\to \Aa^1 $ similarly to our description of $D_{12}$ in Proposition~\ref{mainprop}, just adding the variable $\lambda$. The blow-up  $Bl_{p\times X}(X\times\Aa^1)$  is the same as $(Bl_{p}X)\times\Aa^1$ and is covered by the open subsets $U_i\times\Aa^1$, with the $U_i$ as in Proposition~\ref{mainprop}. Over $U_0\times\Aa^1$ we have
	\[
	f_\lambda=s_0^e(F(1,t_1,\ldots, t_n)+s_0\lambda h')
	\]
	and since the exceptional divisor is defined by $s_0$ on $U_0\times\Aa^1$, we see that $(\mathcal{D}_{12}\cap U_0)\times\Aa^1=(V(F)\cap U_0)\times\Aa^1\subset (Bl_{p}X)\times\Aa^1$. Thus $\mathcal{D}_{12}=V(F) \times\Aa^1$, and this scheme is smooth by our assumption on $F$.
	 
	In the quasi-homogenous case we go through the same construction as in the last section. First let
	\[
	 \OO_{X,p}[s^{1/a}]:=\OO_{X,p}[\sigma_0,\ldots,\sigma_n]/(\{\sigma_i^{a_i}-s_i\}_i) .
	 \]
and let $g^*(t)\in \OO_{X,p}[s^{1/a}]$ be the image of $f^*(t)=F(s_0,\ldots, s_n)+h$ under the inclusion $\OO_{X,p}\subset \OO_{X,p}[s^{1/a}]$. Letting $Z= \Spec \OO_{X,p}[s^{1/a}]$, we have the usual $\mu_a$-action on $Z$ with $X=Z/\mu_a$. The element $g^*(t)\in \OO_{X,p}[s^{1/a}]$ defines the morphism
\[
Z= \Spec \OO_{X,p}[s^{1/a}] \xrightarrow{g}\Spec k[t],
\]
making the the diagram
\[
\begin{tikzcd}
Z\arrow[r]\arrow[dr,"g"]&X\arrow[d, "f"]\\
&\Spec k[t]
\end{tikzcd}
\]
to commute. Moreover,  $g^*(t)=F(\sigma_0^{a_0},\ldots,\sigma_n^{a_n}) + h'$ with  $G(\sigma_0,\ldots,\sigma_n) = F(\sigma_0^{a_0},..,\sigma_n^{a_n})$ homogeneous of degree $e$, and with $h'\in m_q^{e+1}$. Define the morphism
\[
g_\lambda:\mathcal{Z}:=Z\times\Aa^1\to \Spec k[t,\lambda]
\]
by $g_\lambda^*(t)=G+\lambda\cdot h'$. Then $\mathcal{X}=\mathcal{Z}/\mu_a$ and we have a commutative diagram
\[
\begin{tikzcd}
\mathcal{Z}\arrow[r]\arrow[dr, "g_\lambda"]&\mathcal{X}\arrow[d, "f_\lambda"]\\
&\Spec k[t,\lambda]
\end{tikzcd} .
\]
Next,  blow up $\mathcal{Z} = Z \times \mathbb{A}^1$ at $p \times \Aa^1$ to get $\hat{\mathcal{Z}}$ and denote by $\hat{\mathcal{X}}$ the quotient by the action of $\mu_a$. Let $q: \hat{\mathcal{X}} \rightarrow \mathcal{X}$ {be} the natural map. Denote the intersection of the strict transform of $\mathcal{Z}_\sigma$ and the exceptional divisor in $\hat{\mathcal{Z}}$ by $\mathcal{E}_{12}= V_{\P^n} (F) \times \Aa^1$ (see the paragraph above) and its image under the $\mu_a$-quotient map by $\mathcal{D}_{12}$.  Then we get $\mathcal{D}_{12} = V_{\P(a)}(F) \times \Aa^1$  which is smooth by our Assumption \ref{qassumption}.   
	
{Let $\mathcal{X}_\sigma:=f_\lambda^{-1}(\sigma\times\Aa^1)\subset\mathcal{X}$. We have in both cases the proper map $q:\hat{\mathcal{X}} \rightarrow \mathcal{X}$, which is an isomorphism over
$\mathcal{X}\setminus p\times\Aa^1$. Let $q^{-1}[ \mathcal{X}_\sigma]$ be the closure of 
$q^{-1}(\mathcal{X}_\sigma\setminus  p\times\Aa^1)$ in $\hat{\mathcal{X}}$.} In both cases, the Cartier divisor $\mathcal{D}_{12}$ on the {reduced} scheme  $q^{-1}[\mathcal{X}_\sigma]$ is smooth over $\Aa^1$. Let $r: q^{-1}[\mathcal{X}_\sigma]\to \Aa^1$ be the morphism induced by $f_\lambda$. Then $r$ is flat and  the set $W$ of points $x\in q^{-1}[\mathcal{X}_\sigma]$ such that  $x$ is a smooth point of the fibre $r^{-1}(r(x))$ is an open subset of $q^{-1}[\mathcal{X}_\sigma]$, and is equal to the set of points of $q^{-1}[\mathcal{X}_\sigma]$ at which $r$ is a smooth morphism. By Lemma~\ref{lemma}, $W$ is an open neighbourhood of $\mathcal{D}_{12}$ in $q^{-1}[\mathcal{X}_\sigma]$. Letting $F$ be the closed complement of $W$ in $q^{-1}[\mathcal{X}_\sigma]$, and noting the $q$ is proper, $q(F)$ is a closed subset of $\mathcal{X}_\sigma$, disjoint from $p\times\Aa^1$. Set $U : = \mathcal{X}_\sigma \setminus q(F)$. Then $U$ is open and $U \setminus (p \times \Aa^1) \simeq W \setminus \mathcal{D}_{12}$ (via the strict transform identification) is smooth  over $\Aa^1$.
	  
\end{proof}

\begin{Prop} \label{Euler2}
	{Let $X$ be a smooth quasi-projective scheme over a field $k$, with} $Z \subset X$ closed, let $p: V \rightarrow X$ be a vector bundle, and let $s_1,s_2 : X \rightarrow V$ be two sections. Let $E$ be an $SL$-oriented motivic spectrum with respect to it Euler classes are defined. \\
		Consider $\tilde{p}: \pi^* V \rightarrow X \times \Aa^1 $ with $\pi$ the projection $\pi : X\times \Aa^1 \rightarrow X $.
	Define a section $s : X \times \Aa^1 \rightarrow \pi^*V $ by $s= \lambda s_1 + (1-\lambda) s_2 $ and assume that we have an open neighbourhood $U$ of $Z\times \Aa^1$ in $X\times \Aa^1$ such that  $Z(s) \cap U = Z \times \Aa^1 $. Then
	\[ e_{Z}(X,s_1) = e_{Z}(X,s_2).\] 
\end{Prop}

\begin{proof}
	Let  $s_0 : X \rightarrow V$ be the zero section. 
	 We have the Thom class
	\[ th(V) = s_{0*} \one_X \in E^{V^*}_{0_{V}} ( V) .\]
	We have the two embeddings $i_1 : X \hookrightarrow X\times {0} \subset X\times \Aa^1$ and $i_2 : X \hookrightarrow X\times {1} \subset X\times \Aa^1$. By homotopy invariance the two maps \[i_{1}^* , i_{2}^*: E^{\pi^*V^\vee}_{Z \times \Aa^1} (X\times \Aa^1) \to E_Z^{V^\vee}(X) \] are equal. 
	Using the excision property in cohomology we can remove the piece $(X\times \Aa^1) \setminus U$ to get the equivalence \[ \alpha :	E^{V_U^*}_{Z \times \Aa^1} (U) \simeq E^{\pi^*V^*}_{Z \times \Aa^1} (X\times \Aa^1) . \]
Here $V_U$ is the pullback of $V$ over $U \hookrightarrow X \times \Aa^1$. \\
		Let $s' = s|_U : U \rightarrow V_U $ and $\tilde{p}'=\tilde{p}|_{V_U} : V_U \to U$. Since $Z(s')= Z(s) \cap U = Z \times \Aa^1$, we have a map \[ s'^* : E^{p'^*V_U^*}_{0} ( V_U) \to	E^{V_U^*}_{Z\times \Aa^1} (U) . \]

	Denote by $\tilde{\pi} $ the pullback map $V_U \rightarrow V$ of the vector bundle $V \to X$ along $U \hookrightarrow X\times \Aa^1 \to X$ and
	consider the following commutative diagram -
	\[
	\begin{tikzcd} [column sep = large]
		 E^{p'^*V_U}_{0} ( V_U) \arrow[d, "s'^*"]
		&E_{0_{V}}^{p^*V^*} ( V)  \arrow[dd, shift left, "s_2^*"] \arrow[l, "\tilde{\pi}^*"']
		\arrow[dd, shift right, "s_1^*"'] \\
			E^{V_U^*}_{Z\times \Aa^1} (U)  \arrow[d, "\simeq"', "\alpha"] 
		&  \\	
		E^{\pi^*V^*}_{Z \times \Aa^1} (X\times \Aa^1) 
			\arrow[r, shift left, "i_{1}^*"]
			\arrow[r, shift right, "i_{2}^*"'] 
			&	E_Z^{V^*}(X)
	\end{tikzcd} .
	\] 	
	We have
	\[ s_1^* th(V) = {i}_{1}^* \circ \alpha \circ s'^* \circ \tilde{\pi}^* th(V) = {i}_{2}^* \circ \alpha \circ s'^* \circ \tilde{\pi}^* th(V) = s_2^*th(V) \] 
	which gives the desired equality of local Euler classes.
\end{proof}
Let now $E =  H\mathcal{K}^{MW}$.
\begin{Cor} \label{Eulerclass}
	Let $f:X\to\Spec\OO$ be a flat quasi-projective morphism with $X$ smooth over $k_0$ and with an isolated critical point $p\in X_\sigma$.  Suppose that $f$ looks like $F=F(T_0,\ldots, T_n)$ at $p$ (see \ref{lookslike}).   Then 
 \[ 
 \mu^q_{f,p} = e_p(\Omega_{X/k(p)}, df)  =e_0(\Omega_{\Aa^{n+1}_{k(p )}/k(p )}, d(F(t_0,\ldots, t_n))= e_0(\Omega_{H^F/k(p )}, df_F) = \mu^q_{f_F,0}
 \]
 in $\GW(k(p ))$. In particular the quadratic Milnor number only depends on the principal part (i.e. $F$) at the expansion of $f$ at $p$.
\end{Cor}

\begin{proof} Proposition \ref{Euler1} proves that the assumptions in Proposition \ref{Euler2} are satisfied for $E= H\mathcal{K}^{MW} $, $Z=\{p\}$, $V= \Omega_{X/k} \to X$, $s_1 = df$, and $s_2 = dF$. This gives the following identity in $\GW(k(p))$,
\[
e_p(\Omega_{X/k(p)}, df)=e_p(\Omega_{X/k(p)}, d(F(s_0,\ldots, s_n))).
\]
The parameters $s_0,\ldots, s_n\in \OO_{X,p}$  define an \'etale map $\alpha :\Spec \OO_{X,p} \rightarrow \Aa^{n+1}_{k(p )}:=\Spec k(p )[t_0,\ldots, t_n]$ which maps $p$ to $0$ and {with} $\alpha^*F(t_0,\ldots,t_n) = F(s_0,\ldots,s_n)$. {Thus $(s_0,\ldots,s_n)$ expresses $(X,p)$ as a Nisnevich neighbourhood of $(\Aa^{n+1}_{k(p )},0)$. Since  \[(s_0,\ldots,s_n)^*(F(t_0,\ldots, t_n))=F(s_0,\ldots, s_n),\] we have 
\[
e_p(\Omega_{X/k(p)}, d(F(s_0,\ldots, s_n)))=\overline{(s_0,\ldots, s_n)}^*(e_0(\Omega_{\Aa^{n+1}_{k(p )}/k(p )}, d(F(t_0,\ldots, t_n))
\]
where $\overline{(s_0,\ldots, s_n)}^*:\GW(k(p))\to \GW(k(p))$ is the isomorphism induced by 
$(s_0,\ldots,s_n): p\to 0$, is just the identity map, so we can write the above equation as
\[
e_p(\Omega_{X/k(p)}, d(F(s_0,\ldots, s_n)))= e_0(\Omega_{\Aa^{n+1}_{k(p )}/k(p )}, d(F(t_0,\ldots, t_n)).
\]
The singular point $0=(0:\ldots:0:1)$ of $H^F_{\sigma_{k(p )}}$ is in the affine open subscheme $U_{n+1}\subset \P_{\OO_{k(- )}}(a_*,1)$, so to compute $ e_0(\Omega_{H^F/k(p )}, df_F)$, we can restrict to $U_{n+1}$. We have
\[
U_{n+1}=\Spec \OO_{k(p )}[T_0,\ldots, T_n, T_{n+1}][T_{n+1}^{-1}]_0
\]
and $\OO_{k(p )}[T_0,\ldots, T_n, T_{n+1}][T_{n+1}^{-1}]_0$ is the polynomial ring
$\OO_{k(p )}[t_0,\ldots, t_n]$, with $t_i=T_i/T_{n+1}^{a_i}$. On $U_{n+1}$, $H^F$ has defining equation 
\[
(F(T_0,\ldots, T_n)-tT_{n+1})/T_{n+1}^e=F(t_0,\ldots, t_n)-t.
\]
Thus, $H^F\cap U_{n+1}$ is just the graph of the morphism 
\[
F(t_0,\ldots, t_n):\Aa^{n+1}_{k(p )}=\Spec k(p )[t_0,\ldots, t_n]\to \Spec k[t]_{(t)}.
\]
If we replace the graph $H^F\cap U_{n+1}$ with  the isomorphic scheme $\Spec k(p )[t_0,\ldots, t_n]$ via the isomorphism given by the first projection, then $f_F$ transforms to the map  $F(t_0,\ldots, t_n)$ and $0$ goes to the origin $(0,\ldots, 0)\in \Aa^{n+1}_{k(p )}$. In other words,
\[
 e_0(\Omega_{H^F/k(p )}, df_F)=e_0(\Omega_{\Aa^{n+1}_{k(p )}/k(p )}, d(F(t_0,\ldots, t_n)).
 \]}
\end{proof}

\subsection{The local Euler class and the Jacobian ring} 

 We recall here an algebraic construction of a distinguished quadratic form related to the  Scheja-Storch element, which gives the local Euler class $e_p(\Omega_{X/k}, s)$ of Definition \ref{LocalEulerClass}. This gives an explicit algebraic interpretation to the $\Aa^1$-Milnor number.

\begin{Def} Let $k$ be a field and $X$ be a smooth finite type scheme over $k$. Let $p\in X$ be a closed point, take $f\in \OO_{X,p}$, and let $s_0,\ldots, s_n\in m_p$ be a  regular system of parameters at $p$. Suppose that  $\sqrt{(\partial f/\partial s_0\ldots \partial f / \partial s_n)}=m_p$, so $df$ has an  isolated zero at $p$; note that the ideal $(\partial f/\partial s_0\ldots \partial f / \partial s_n)$ does not depend on the choice of the $s_i$. Let $k(p)$ be the residue field of $\OO_{X,p}$. \\
	The \textit{Jacobian ring of $f$ at $p$}, $J(f,p)$, is defined as
	\[
	J(f,p):=\OO_{X,p}/(\partial f/\partial s_0\ldots \partial f / \partial s_n) .
	\] 
	For $k$ algebraically closed, the dimension of $J(f,p)$ over $k$ is the {\em Milnor number} $\mu_{f,p}$. \\
	Since $\partial f/\partial s_i $ is in $m_p=(s_0,\ldots, s_n)$,  we can write for each $i$,
	\[
	\partial f/\partial s_i = \sum_j a_{ij} s_j
	\]
	with $a_{ij} \in \OO_{X,p}$.  The {\em Scheja-Storch element} $e_{f,p} \in J(f,p)$ is defined as the  image of the determinant $det(a_{ij})$ in $J(f,p)$; $e_{f,p}$ is independent of the choices made.         
	Since $J(f,p)$ is an Artinian local $k$-algebra, $J(f,p)$ contains the residue field $k(p )$. \\
	Let $Tr : J(f,p) \to k(p)$ be a $k(p )$-linear map sending $e_{f, p}$ to $1$. Define \[B_{f,p} : J(f,p) \times_{k(p)} J(f,p) \to k(p)\] by $B_{f,p} (x,y) = Tr(xy)$. The class $[B_{f,p}] \in \GW(k(p))$ does not depend on the choices of generators $(s_0,...,s_n)$ or the map $Tr$,  see \cite[Theorem 3.1]{Le20a}. In addition this class computes the quadratic Milnor number of $f$ at $p$ (\cite[Proposition 2.32 and Theorem 7.6]{BW} and \cite[Corollary 3.3]{Le20a}), 
		\[
	\mu_{f,p}^q = e_p(\Omega_{X/k}, df) = [B_{f,p}] . \label{SSEulerClass}
	\]
\end{Def}
	
	By taking the rank of the corresponding quadratic form, $\rk \mu^q_{f,p} = \dim J(f,p) = \mu_{f,p}$, so the class $\mu_{f,p}^q \in \GW(k(p))$ can be viewed as a quadratic refinement of the Euler number $\mu_{f,p} \in \mathbb{Z}$.

It follows from this discussion and Corollary~\ref{Eulerclass} that for a semi-quasi-homogeneous singularity $p$, $\mu_{f,p}^q$ can be defined purely algebraically in terms of the polynomial $F \in k(p)[T_0, \ldots T_n]$, by the Scheja-Storch form. For a beautiful survey on the quadratic Milnor number with some computed examples see \cite{Orm}.

\section{The generalized conductor formula}

In this section we use the results of the previous sections computing  $\chi(\Psi_f\one|_p) $ at a singular point $p$ and reinterpret them in terms of the difference $\Delta_t(F/k)$  considered in  \cite{Le20b}. Using the functoriality of $\Psi_f$, this allows to generalize the formula proven in \cite{Le20b}  to the case of $f:X\to\Spec\OO$ with  finitely many isolated critical points, all satisfying Assumptions \ref{assumption} or \ref{qassumption}, which is our main result in this paper. In particular, this verifies the conjecture formulated in \cite[Conjecture 5.7]{Le20b} in characteristic zero, for a somewhat wider class of singularities than what was considered there. We explain this with more detail 
below.

{We retain in this section our notations and assumptions for $\OO$ and $B=\Spec \OO$ as in Section~\ref{SectionNearbyCycles}, and assume in addition that the subfield $k_0\subset \OO$ has characteristic zero. We have the characteristic zero residue field $k$ and fraction field $K$ of $\OO$.} Let $f: X \rightarrow B$  be a flat, quasi-projective morphism such that $X$ is smooth over $k_0$, $X_\eta$ is smooth over $\eta$ and such that $X_\sigma$ has finitely many singular points.  

\sloppy Fix a sequence of pairwise relative prime weights $a:=(a_0,\ldots, a_n)$ and a field $\kappa$, and  let $F\in \kappa[T_0,\ldots, T_n]$ be a degree $e$ $a$-weighted homogeneous polynomial such that $V(F)\subset \P_\kappa(a)$ is a smooth quotient hypersurface, in the sense of Definition~\ref{smhyper}. We have the discrete valuation ring $\OO_\kappa:=\kappa[t]_{(t)}$,  the hypersurface $H^F:=V(F-tT_{n+1}^e)\subset \P_{\OO_\kappa}(a,1)$ with projection $f_F:H^F\to \Spec\OO_\kappa$.  $H^F$ is smooth over $\kappa$,  $H^F_{\eta_k}$ is smooth over $\eta_\kappa$,  and $H^F_{\sigma_\kappa}$ has a single singularity at $p:=(0:\ldots:0:1)$. In fact, $H^F_{\sigma_\kappa}$ is the cone over $V(F,T_{n+1})\subset V(T_{n+1})=\P(a)_\kappa$ with vertex $p$. In \cite{Le20b}, Levine, Pepin Lehalleur and Srinivas consider the invariant
\[
\Delta_t(F/k):=\Sp_t(\chic(H^F_{\eta_\kappa}/\kappa(t)))-\chic(H^F_{\sigma_\kappa}/\kappa)\in \GW(\kappa)
\]
and derive an expression, named  {\em a conductor formula},  for $\Delta_t(F/k)$ in terms of the local Euler class $e_p(\Omega_{H^F/\kappa}, dt)\in \GW(\kappa)$. Note that $f_F:H^F\to \Spec\OO_\kappa$ looks at $p=(0:\ldots:0:1)$ like the weighted homogeneous singularity defined by $F$. A generalization of the conductor formulas for $\Delta_t(F/k)$ for degenerations with finitely many singularities of a certain type is conjectured in {\it loc. cit.} \cite[Conjecture 5.7]{Le20b}.

For an $a_*$-weighted homogeneous $F$ the conductor formula of Levine, Pepin Lehalleur and Srinivas has the form (\cite[Theorem 5.6]{Le20b}) - 
\[ \Delta_t(F/\kappa) =  \<\prod_j a_j \cdot e\> - \<1\> + (-\<e\>)^{n} \cdot e_{0}( \Omega _{H^F/\kappa} , dt )  \in \GW(\kappa) . \] 
Here $ e_{0}( \Omega _{H^F/k(p)} , dt ) $ is the local Euler class at $0:=(0:\ldots:0:1)$ \cite[5]{Le20b}, also see Definition~\ref{LocalEulerClass}. We wish to extend this to a formula  {in the case of a morphism with isolated critical points} that look like homogeneous or quasi-homogeneous singularities. In order to do that we give a comparison between $\chic(\Psi_f\one|_p)$ of the scheme and the motivic Euler characteristic of the hypersurface $H^F$ defined by the polynomial $F$. Recall from Section~\ref{SectionMotivicEulerChar} that for a finite separable field extension $k_1\subset k_2$, we have the transfer map $\Tr_{k_2/k_1}:\GW(k_2)\to \GW(k_1)$. 

\begin{Th} \label{mainthm}
	Let $\OO$ and $B:=\Spec\OO$ be as in Section~\ref{SectionNearbyCycles}, with the field $k_0\subset \OO$ being of characteristic zero.
		Let $f: X \rightarrow B$  be a flat quasi-projective morphism with $X$ smooth over $k_0$ and with $X_\eta$ smooth over $\eta$,  and let $p\in X_\sigma$ be an isolated critical point of $f$, satisfying assumption \ref{assumption} or \ref{qassumption}. Let $F\in k(p )[T_0,\ldots, T_n]$ be the corresponding (weighted) homogeneous polynomial.  Then
	\[ \chic(\Psi_f({\one_X}_{\eta})|_p) =  \Tr_{k(p )/k}(\Delta_t(F/k(p )) +  \<1\>)\in \GW(k).  \]	
\end{Th}

\begin{proof} The homogeneous case is a special case of the weighted homogeneous case, with all weights equal to 1, so we need only handle the weighted homogeneous case. Since $\chic(\Psi_f(\one_{X_{\eta}})|_p)$ is determined by a neighbourhood of $p$ we can assume $p$ is the only critical point of $f$.
	
	Note that we have families $f:X\to\Spec\OO$ and $f_F: H_F\to \Spec k(p )[t]_{(t)}$ over different bases, so we need to keep track of the base fields for the Euler characteristics and the base schemes for the nearby cycles functors. First we show that the terms in the difference $\Delta_t(F/k(p ))$ are closely related to the Denef-Loeser covers we computed in   Theorem \ref{mainprop} and Theorem \ref{qmainprop}.
		{By Property~\ref{property}, Proposition~\ref{PsiId} and Proposition~\ref{simon}, we have
	 \begin{align*}
\Sp_t(\chic(H^F_{\eta_{k(p )}}/k(p )(t)))&=\Psi_{\id_{k(p )[t]_{(t)}*}}(\chic(H^F_{\eta_{k(p )}}/k(p )(t)))\\
&=\Psi_{\id_{k(p )[t]_{(t)}*}}(\chi(f_{F{\eta_{k(p )}}*}(\one_{H^F_{\eta}})))\\
&=\chi(\Psi_{\id_{k(p )[t]_{(t)}}}(f_{F{\eta_{k(p )}}*}(\one_{H^F_{{\eta_{k(p )}}}})))\\
&=\chi(f_{F\sigma_{k(p )}*}(\Psi_{f_F}(\one_{H^F_{{\eta_{k(p )}}}})))\\
&=\chic(\Psi_{f_F}(\one_{H^F_{{\eta_{k(p )}}}})/k(p )).
\end{align*}}
 {On the other hand, we can apply   Corollary~\ref{CorWtHomogFormula}   to give 
\[
\chic(\Psi_{f_F}(\one_{H^F_{\eta_{k(p )}}}))/k(p ))=\chic(V(F-T_{n+1}^e)/k(p ))+\chic(H^{F\circ}_{\sigma_{k(p )}}/k(p ))-
\chic(\Aa^1\times V(F)/k(p )).
\]
However, $H^{F\circ}_{\sigma_{k(p )}}$ is an $\Aa^1$-bundle over $V_{\P_{k(p )}(a)}(F)\cong V_{\P_{k(p )}(a,1)}(F,T_{n+1})\subset\P_{k(p )}(a,1)$, so we have 
\[
\chic(H^{F\circ}_{\sigma_{k(p )}/k(p )})=\chic(V(F)/k(p ))\cdot \chic(\Aa^1/k(p ))=\chic(\Aa^1\times V(F)/k(p )),
\]
which yields
\[
\chic(\Psi_{f_F}(\one_{H^F_{\eta_{k(p )}}})/k(p ))=\chic(V(F-T_{n+1}^e)/k(p )).
\]
Thus}
\[  \Sp_t \chic(H^F_{\eta_{k(p )}}/k(p )(t)) =\chic (V(F-T_{n+1}^e)/k(p )) = \chic(\widetilde{D_1}/k(p )). \]
 {Now $H_{\sigma_{k(p )}}^F= H^{F\circ}_{\sigma_{k(p )}}\amalg (0:\ldots:0:1)_{k(p )}$, and} $V_{\mathbb{P}_{k(p )} (a,1)}(F, T_{n+1}) \simeq V_{\mathbb{P}(a)_{k(p )}}(F) \simeq \widetilde{D_{12}} $, so
 \[  \chic (H^F_{\sigma_{k(p )}}/k(p )) = \chic(\widetilde{D_{12}}/k(p )) \cdot \<-1\> + \<1\>\in \GW(k(p )). \] 
 Adding this up (or rather subtracting) we have the formula
  \[ \Delta_t(F/k(p )) = \chic(\widetilde{D_1}/k(p )) - \chic(\widetilde{D_{12}}/k(p )) \cdot \<-1\> - \<1\>\in \GW(k(p )). \]
Applying Proposition~\ref{Transfer}, this gives
 \[
 \Tr_{k(p )/k}(\Delta_t(F/k(p )))=\chic(\widetilde{D_1}/k) - \chic(\widetilde{D_{12}}/k) \cdot \<-1\> - \Tr_{k(p )/k}(\<1\>)\in \GW(k) .
 \]
  On the other hand, by Corollary \ref{CorLocalHomogFormula} and Theorem (\ref{qmainprop}), we have
\[ \chic(\Psi_f({\one_X}_{\eta})|_p/k) = 
 \chi(\widetilde{D_1}/k)  - \chi(\widetilde{D_{12}}/k) \cdot \<-1\>  ,
 \]
so comparing both terms above, we have 
\[  
\chic(\Psi_f({\one_X}_{\eta})|_p/k) =\Tr_{k(p )/k}(\Delta_t(F/k(p )) + \< 1 \>),
\]
concluding the proof.
\end{proof}

Combining the result of the theorem above with \cite[Theorem 5.6]{Le20b} and the result of Section 6, we obtain the local formula for $\chic(\Psi_f \one)$ at a singularity $p$.

\begin{Cor}
	In the setting of Theorem \ref{mainthm}, let $(a_0,\ldots, a_n)$ be the weights and $e$ the weighted degree for $F$ (where all weights are $1$ when $F$ is homogeneous), then
	\[ \chic(\Psi_f({\one_X}_{\eta})|_p) =  \Tr_{k(p )/k}(\<\prod_j a_j \cdot e\> + (-\<e\>)^{n} \cdot e_{0}( \Omega _{\Aa^{n+1}/k(p )} , dF) )\in \GW(k). 
	\]
\end{Cor}

\begin{proof}
	The formula follows from that of Theorem \ref{mainthm}, together with the formula of \cite[Theorem 5.6]{Le20b} mentioned above, and the identity
	\[
	e_{0}(\Omega _{\Aa^{n+1}/k(p )} , dF)=e_{0}( \Omega _{H^F/\kappa} , dt ),
	\]
	of Corollary~\ref{Eulerclass}.
\end{proof}

We now proceed to obtain a global formula in the general case, when $X$ has multiple singular points at the special fibre $p_1,\ldots , p_r$ satisfying Assumption \ref{assumption} (or \ref{qassumption}). We state our main result in the weighted homogeneous setting as this also includes the  homogeneous case.

\begin{Cor}[Generalized quadratic conductor formula] \label{conductor}
	Let $X \rightarrow \Spec \OO$ be a flat {projective} morphism of relative dimension $n$, with $X$ smooth over $k_0$ of characteristic $0$,  and $X_\eta $ smooth over $\eta$. Suppose that the special fibre $X_\sigma$ has isolated singularities ${p_1,..,p_r}$ satisfying Assumption \ref{qassumption} with $F_i\in k(p_i)[T_0,\ldots,T_n]$ an $a_*^{(i)}$-weighted homogeneous polynomial of degree $e_i$. Then
\[
	\Sp_t(\chic(X_\eta/k(\eta)))-\chic(X_\sigma/k) = \sum_i \Tr_{k(p_i)/k} \left( \< \prod_j a_j^{(i)} \cdot e_i\> - \<1\> + (-\<e_i\>)^{n} \cdot \mu^q_{f,p_i} \right) .
\]
	
\end{Cor} 

\begin{proof} By applying Proposition~\ref{Transfer}, Proposition \ref{local} and Theorem \ref{mainthm} we obtain the formula 

\begin{align*}
\chic(\Psi_f({\one_X}_{\eta})) = & \sum_i  \chic(\Psi_f({\one_X}_{\eta})|_{p_i}) + \chic(X_\sigma\setminus \{p_1,\ldots,p_k\}/k) 
\\ =  & \sum_i \Tr_{k(p_i)/k}( \Delta_t (F_i/k(p_i)) + \<1\>) + \chic(X_\sigma) -\sum_i \Tr_{k(p_i)/k}(\<1\>) .
\end{align*}
This gives the global formula 
\[
\chic(\Psi_f({\one_X}_{\eta}))- \chic(X_\sigma) = \sum_i \Tr_{k(p_i)/k} ( \Delta_t (F_i/k(p_i))) .
\]
Substituting Levine-Pepin Lehalleur-Srinivas's conductor formula \cite[Theorem 5.6]{Le20b} gives
\[
\chic(\Psi_f({\one_X}_{\eta}))-\chic(X_\sigma) = \sum_i \Tr_{k(p_i)/k} [ \<\prod_j a_j^{(i)} \cdot e_i\> - \<1\> + (-\<e_i\>)^{n} \cdot e_{0}( \Omega _{H^{F_i}/k(p_i)} , dt ) ] . \]
But as we proved in Section 6, Corollary \ref{Eulerclass} we can replace  $e_{0}( \Omega _{H^{F_i}/k(p_i)} , dt )$ with $\mu^q_{f,p_i} = e_{p_i}( \Omega _{X/k(p_i)} , dt ) $. Then by \cite[Proposition 8.3]{Le20b} which states that $\chic(\Psi_f({\one_X}_{\eta})) = \Sp_t\chic(X_\eta/k(\eta))$ (in the proof), we get the desired result
\[
\Sp_t\chic(X_\eta/k(\eta))-\chic(X_\sigma/k) = \sum_i \Tr_{k(p_i)/k} [ \<\prod_j a_j^{(i)} \cdot e_i\> - \<1\> + (-\<e_i\>)^{n} \cdot\mu^q_{f,p_i} ] . \]
\end{proof}

\begin{Rem} \label{RefMilnor}
	
	Notice that this formula refines in quadratic forms the formula by Milnor \eqref{MilnorFormula} mentioned in the introduction. Assume $k=\mathbb{C}$, and let $f: X \to \Aa^1=\Spec \mathbb{C}[t]$ be a flat family of varieties, $X$ being an {$n+1$-}dimensional smooth $\mathbb{C}$-scheme, and let $X_t ={f^{-1}(\mathbb{G}_m)}$, $X_0 = f^{-1} (0)$.  Suppose that $f|_{X_t}: X_t \to \mathbb{G}_m$ is smooth, and $f|_{X_0} : X_0 \to \mathbb{C}$ has isolated $F_i$-weighted-homogeneous singular points $p_i$. We can specialize to $X \to \Spec k[t]_{(t)}$ and use our formula above. 
Then since $\rk \mu_{F_i,p_i}^q = \dim J(F_i,p_i) = \mu_{F_i,p_i}$, and from Remark~\ref{ComplexChi}, taking ranks on both sides of the equation in the formula above gives 
\[
\chi^{top}(X_t)-\chi^{top}(X_0) = (-1)^{n} \sum_i   \mu_{F_i,p_i} .
\]
which is Milnor's formula mentioned in the introduction \eqref{MilnorFormula}. Note that at each point, the difference $\< \prod_j a_j^{(i)} \cdot e_i\> - \<1\>$ vanishes under the rank map, as a difference of two rank $1$ quadratic forms{; similarly, the term $(-\<e_i\>)^{n}$ maps to $(-1)^n$. This simplification also occurs for $k=\mathbb{R}$, as $\prod_j a_j^{(i)} \cdot e_i$ and $e_i$ are squares in $\mathbb{R}$}{. Thus, these terms  are only apparent}  in the refined formulas{;} see also \cite[Section 1]{Le20b}. Similarly, the formula refines the Deligne-Milnor formula \eqref{DeligneFormula} in equal characteristic zero with isolated singularities of the type discussed here, by taking $\ell$-adic realisation. 
\end{Rem}

\section{The case of curves on a surface}

As an application of our main theorem, we develop here a formula for the difference between the quadratic Euler characteristic of curves on a surface, refining a formula for complex varieties deduced from {the formula of Jung-Milnor}.

Let $C$ be a reduced curve on a smooth projective surface $S$ over an algebraically closed field $k$ of characteristic zero. Let $\pi:\tilde{C}\to C$ be the normalisation. Let $p$ be a singular point of $C$. \\
Let $r_p$ be the the number of points in $\pi^{-1}(p)$; let $\delta_p$ be the length of the (finite length) $\OO_{C,p}$-module $\pi_*(\OO_{\tilde{C}, \pi^{-1}(p)})/\OO_{C,p}$; and let $\mu_p$ be the Milnor number defined above {for} the local defining equation for $C$, $f_p\in \OO_{S,p}$, at $p$. The Jung-Milnor formula \cite[Chapter 10]{Mil} states that
\[
2\delta_p=\mu_p+r_p-1 .
\]
If $C$ is irreducible, we have $h^0(C, \OO_C)=1=h^0(\tilde{C},\OO_{\tilde{C}})$ and the short exact sequence
\[
0\to \OO_C\to \pi_*\OO_{\tilde{C}}\to  \pi_*\OO_{\tilde{C}}/\OO_C\to0
\]
gives
\[
h^1(C,\OO_C)=h^1(\tilde{C},\OO_{\tilde{C}})+\sum_{x\in C_{sing}}\delta_p.
\]
Let $f_0$ be the canonical section of the invertible sheaf $\OO_S(C )$ and assume that $\OO_S(C )$ has a section $f_1$ whose divisor is a smooth curve $C_1$, such that each point of $C\cap C_1$ is a smooth point of $C$, and that the intersection is transverse.  In case $S=\P^2$, and $C$ is a curve of degree $e$, then $\OO_S(C )\cong \OO_{\P^2}(e)$, the  canonical section is just the section given by the defining equation $f_0$ of $C$, and a general homogeneous polynomial $f_1$ of degree $e$ will have the desired properties. $C_1$ is a smooth deformation of $C$, and so we have $g(C_1) = h^1(C,\OO_C)$; $g(\tilde{C}) = h^1(\tilde{C},\OO_{\tilde{C}})$. The classical formula obtained, relating the genus of $\tilde{C}$ and of  $C_1$ in case $C$ is irreducible, is then
\[
g(\tilde{C})-g(C_1)=\sum_{p\in C_{sing}} (1/2)(1-\mu_p-r_p)
\]
or in terms of the topological Euler characteristic ($=2-2g(-)$) of $C_1$ and $\tilde{C}$,
\begin{equation}\label{TopComp1}
\chi^{top}(C_1)-\chi^{top}(\tilde{C})=\sum_{p\in C_{sing}} 1-\mu_p-r_p ,
\end{equation}
which holds even if $C$ is not irreducible. We consider this as the Jung-Milnor formula with several singular points. We can also compare  with  $\chi^{top}(C )$. Since for a curve we have
$
C\setminus C_{sing}\cong \tilde{C}\setminus \pi^{-1}(C_{sing})
$,
we deduce
\[
\chi^{top}(\tilde{C})-\sum_{p\in C_{sing}}r_p=\chi^{top}(C )-\sum_{p\in C_{sing}}1.
\]
Putting this into the genus formula above, we see that this formula is equivalent to
\begin{equation}\label{TopComp2}
\chi^{top}(C_1)-\chi^{top}(C)=\sum_{p\in C_{sing}}(- \mu_p)=- \sum_{p\in C_{sing}} \dim J(f_{p},p),
\end{equation}
where we use some local defining equation $f_{p}\in \OO_{S,p}$ for $C$ to define the Jacobian ring. Using our main result we can deduce a refinement of formulas \ref{TopComp1}, \ref{TopComp2} with quadratic forms.

\begin{Cor} \label{CurveFormula}
	Let $C$ be a reduced curve on a smooth projective surface $S$ over a field $k$ of characteristic zero. Suppose that $\OO_S(C)$ admits a section $s$ with smooth divisor $C_1$ that intersects $C$ transversely. Suppose in addition that each singular point $p$ of $C$ is a quasi-homogeneous singularity; let $a_0^p, a_1^p$ denote the homogeneous weights (with $a_0^p, a_1^p$ relatively prime), let $e_p$ denote the homogeneous degree at $p$. Let $\pi:\tilde{C}\to C$ be the normalisation of $C$. Then  
	\[
	\Sp_t(\chic(C_\eta/\eta))-\chic(C/k)=\sum_{p \in C_{sing}}\Tr_{k(p)/k}(\<a_0^p a_1^p e_p \>-\<1\>-\<e_p\>\mu^q_{f_{p},p}),
	\]
	refining \eqref{TopComp2} by taking the rank; and
	\[
	\Sp_t(\chic(C_\eta/\eta))-\chic(\tilde{C}/k) = \sum_{p\in C_{sing}}\Tr_{k(p)/k}\left(\<a_0^p a_1^p e_p\>-\<e_p\>\mu^q_{f_{p},p}-(\sum_{q\in \pi^{-1}(p)}\Tr_{k(q)/k(p)}\<1\>) \right), \] refining \eqref{TopComp1} by taking the rank.
\end{Cor}

\begin{proof}
	Let $f_0$ be the canonical section of $\OO_S(C )$ and $s$ as in the statement. Let $ B := \Spec k[t]_{(t)}$, let $H=ts+(1-t)f_0$, form the surface $X:=V(H)\subset S \times B$, and let $f:X\to B$ be the projection.
	$H_t =  s -  f_0 $, the assumption on $C\cap C_1$ implies that $X$ is smooth over $k$ with generic fibre  $X_\eta$ a smooth curve over $\eta=\Spec k(t)$, and with special fibre $C$. Since each singular point $p$ looks like a weighted homogeneous singularity of degree $e_p$ with weights $a_0^p, a_1^p$, the formula of Corollary~\ref{conductor} for $f:X\to B$ becomes
	\[
	\Sp_t(\chic(C_\eta/\eta))-\chic(C/k)=\sum_{p\in C_{sing}}\Tr_{k(p)/k}(\<a_0^p a_1^p e_p \>-\<1\>-\<e_p\>e_p(\Omega_{X/k}, dt)).
	\]
	Note that $e_p(\Omega_{X/k}, dt)=e_p(\Omega_{S/k}, df_{p})$, where $f_{p}\in \OO_{S,p}$ is any local expression for $f_0$ (this is independent of choice of local expression, since $\Omega_{S,p}$ has rank 2), and so get then first formula
	\[
	\Sp_t(\chic(C_\eta/\eta))-\chic(C/k)=\sum_{p \in C_{sing}}\Tr_{k(p)/k}(\<a_0^p a_1^p e_p \>-\<1\>-\<e_p\>\mu^q_{f_{p},p}) .
	\]
	For the second formula, we just have to recall that since the normalisation of a curve, $\tilde{C} \to C$, satisfies $\tilde{C} \setminus \pi^{-1}(C_{sing}) \simeq C \setminus C_{sing}$, and using cut and paste property, we have
	\[ \chic(\tilde{C}/k)-\chic(C/k) = \chic(\pi^{-1}(C_{sing})/k) - \chic(C_{sing} /k)=  \sum_{p\in C_{sing}} \Tr_{k(p)/k} ( \sum_{q\in \pi^{-1}(p)}\Tr_{k(q)/k(p)}\<1\>  - \<1\>) ;\]
	this gives the last formula for the difference
	 \[
	\Sp_t(\chic(C_\eta/\eta))-\chic(\tilde{C}/k) = (\Sp_t(\chic(C_\eta/\eta))-\chic(C/k))-(   \chic(\tilde{C}/k) - \chic(C/k) ) . \]
	To see that those formulas refine the classical formulas over $\mathbb{C}$ by taking ranks, use remark~\ref{ComplexChi}, note that $C_\eta$ is a smooth deformation of $C_1$, so $C_\eta$ and $C_1$ have the same topological Euler characteristic after choosing an embedding of $k(p)$ into $\mathbb{C}$, and that $\rk q_{f_p,p} = \dim J(f,p) = \mu_{F_p,p}$.
\end{proof}
 
We conclude with the following identity in the Witt ring $W(k)$.

\begin{Cor} \label{WittIdentity} Let $C$ be a reduced curve on a smooth projective surface $S$ over a field $k$ of characteristic zero. Suppose that $\OO_S(C)$ admits a section $s$ with smooth divisor $C_1$ that intersects $C$ transversely. Suppose in addition that each singular point $p$ of $C$ is a quasi-homogeneous singularity; let $a_0^p, a_1^p$ denote the homogeneous weights (with $a_0^p, a_1^p$ relatively prime), let $e_p$ denote the homogeneous degree at $p$. Then 
\[
\sum_{p\in C_{sing}}\Tr_{k(p)/k}\left(\<a_0^p a_1^p e_p\>-\<e_p\>\mu^q_{f_{p},p}+\sum_{q\in \pi^{-1}(p)}\Tr_{k(q)/k(p)}\<1\>\right)=0
\]
in $W(k)$. 
\end{Cor}

\begin{proof} For $Y$ smooth and projective of odd dimension over $k$, $\chic(Y/k)=0$ in $W(k)$ (see \cite[Example 1.7, 2.]{Le20a}).
\end{proof}
\small
\bibliographystyle{alphamod}
\let\mathbb=\mathbf

\begin{flushleft}
\small Current address of the author: \\
 \textsc{Institut Galil\'ee, Universit\'e Sorbonne Paris Nord\\
	99 avenue Jean Baptiste cl\'ement,
	93430 Villetaneuse, France} \\
ran@math.univ-paris13.fr
\end{flushleft}

\begin{flushleft}
	\small Address where work on the paper has mostly been made: \\
	\textsc{Universität Duisburg-Essen, Fakultät Mathematik\\
		Campus Essen, 45117 Essen, Germany} \\
\end{flushleft}

\begin{flushleft}
	\small Keywords: \\
	\textsc{Quadratic Euler characteristic \\
		Conductor formulas \\
		Motivic nearby cycles \\
		Grothendieck-Witt ring \\
		Motivic homotopy theory \\
		Invariants of Singularities}
\end{flushleft}

\end{document}